\documentclass[reqno,10pt]{amsart}
\usepackage{amsthm,amsfonts,amssymb,euscript,mathrsfs,graphics,color,amsmath,latexsym,marginnote,hyperref}
\usepackage[latin1]{inputenc}
\usepackage{amsmath}
\usepackage{graphicx}
\usepackage{amsfonts}
\usepackage{amssymb}
\usepackage{amsthm}
\usepackage{dsfont}    
\usepackage{mathrsfs}
\usepackage{amsthm}
\usepackage{verbatim}
\usepackage{esint}
\usepackage{stmaryrd}

\theoremstyle{plain}

\numberwithin{equation}{section}

\newtheorem{hyp}{Hypothesis}[section]

\newtheorem{theo}{Theorem}[section]

\newtheorem{de}[theo]{Definition}
\newtheorem{lemma}[theo]{Lemma}
\newtheorem{prop}[theo]{Proposition}
\newtheorem{rmk}[theo]{Remark}

\newcommand{\R}{\mathbb{R}}

\newcommand{\N}{\mathbb{N}}
\newcommand{\C}{\mathbb{C}}

\newcommand{\norm}[2]{\left \lVert {#1} \right \rVert_{{#2}}}
\newcommand{\asso}[1]{\left \lvert {#1} \right \rvert}
\newcommand{\set}[1]{\left\{{#1}\right\}}

\newcommand{\ZZZ}{\mathds{Z}}
\newcommand{\CCC}{\mathds{C}}
\newcommand{\NNN}{\mathds{N}}

\newcommand{\RRR}{\mathds{R}}
\newcommand{\TTT}{\mathbb{T}}
\newcommand{\uno}{\mathds{1}}

\newcommand{\BB}{{\mathcal B}}

\newcommand{\calF}{{\mathcal F}}
\newcommand{\calG}{{\mathcal G}}

\newcommand{\calL}{{\mathcal L}}
\newcommand{\MM}{{\mathcal M}}
\newcommand{\NN}{{\mathcal N}}
\newcommand{\calO}{{\mathcal O}}

\newcommand{\RR}{{\mathcal R}}

\newcommand{\calU}{{\mathcal U}}

\newcommand{\gotR}{{\mathfrak R}}

\newcommand{\weyl}{{\rm Op}^{W}}
\newcommand{\bony}{{\rm Op}^{\mathcal{B}}}
\newcommand{\bonyw}{{\rm Op}^{\mathcal{B}W}}
\newcommand{\ol}{\overline}

\newcommand{\e}{\varepsilon}
\newcommand{\al}{\alpha}
\newcommand{\be}{\beta}

\newcommand{\x}{\xi}

\newcommand{\g}{\gamma}

\newcommand{\h}{\eta}
\newcommand{\la}{\lambda}

\newcommand{\f}{\varphi}
\newcommand{\s}{\sigma}

\newcommand{\del}{\partial}

\newcommand{\ii}{{\rm i}}

\newcommand{\hcic}{{\bf{H}}}
\newcommand{\jap}[1]{\left \langle {#1} \right \rangle}
\newcommand{\pro}[1]{\llbracket {#1} \rrbracket}

\def\hat{\widehat}

\def\bar{\overline}

\providecommand{\vect}[2]{{\bigl[\begin{smallmatrix}#1\\#2\end{smallmatrix}\bigr]}}   
\providecommand{\sm}[4]{{\bigl[\begin{smallmatrix}#1&#2\\#3&#4\end{smallmatrix}\bigr]}}

\makeatletter

\def\l@subsection{\@tocline{2}{0pt}{2.5pc}{5pc}{}}
\def\l@subsubsection{\@tocline{3}{0pt}{4.5pc}{5pc}{}}
\renewcommand\tocchapter[3]{%
  \indentlabel{\@ifnotempty{#2}{\ignorespaces#2.\quad}}#3%
}
\newcommand\@dotsep{4.5}
\def\@tocline#1#2#3#4#5#6#7{\relax
  \ifnum #1>\c@tocdepth 
  \else
    \par \addpenalty\@secpenalty\addvspace{#2}%
    \begingroup \hyphenpenalty\@M
    \@ifempty{#4}{%
      \@tempdima\csname r@tocindent\number#1\endcsname\relax
    }{%
      \@tempdima#4\relax
    }%
    \parindent\z@ \leftskip#3\relax \advance\leftskip\@tempdima\relax
    \rightskip\@pnumwidth plus1em \parfillskip-\@pnumwidth
    #5\leavevmode\hskip-\@tempdima{#6}\nobreak
    \leaders\hbox{$\m@th\mkern \@dotsep mu\hbox{.}\mkern \@dotsep mu$}\hfill
    \nobreak
    \hbox to\@pnumwidth{\@tocpagenum{#7}}\par
    \nobreak
    \endgroup
  \fi}
\makeatother
\AtBeginDocument{%
\makeatletter
\expandafter\renewcommand\csname r@tocindent0\endcsname{0pt}
\makeatother
}
\def\l@subsection{\@tocline{2}{0pt}{2.5pc}{5pc}{}}

\begin{document}

\title[Long time existence for fully nonlinear NLS on $\mathbb{T}$]
{Long time existence for fully nonlinear NLS \\with small Cauchy data on the circle}

\date{}

\author{Roberto Feola}
\email{roberto.feola@univ-nantes.fr}

\author{Felice Iandoli}
\email{fiandoli@unice.fr}

\thanks{This research was 
supported by PRIN 2015 ``Variational methods, with applications to problems in 
mathematical physics and geometry''; the second author has been supported also by 
ERC grant ANADEL 757996.
}

\begin{abstract}
In this paper we prove long time existence for a large class of fully nonlinear, reversible and parity preserving Schr\"odinger equations on the one dimensional torus. 
We show that, if some non-resonance conditions 
are fulfilled,  for any $N\in \mathbb{N}$ and for any initial condition,
which is even in $x$ and size $\e$ in an appropriate Sobolev space, the lifespan of the solution
is of order $\e^{-N}$.
After a paralinearization of the equation we perform several 
para-differential changes of variables which diagonalize 
the system up to a very regularizing term. 
Once achieved the diagonalization, we construct modified 
energies for the solution by means of Birkhoff normal forms techniques.
\end{abstract}

\maketitle

\tableofcontents

\section{Introduction}

This paper is devoted to get lower bounds for the lifespan of small solutions of the following fully nonlinear  Schr\"odinger type equation
\begin{equation}\label{NLS}
\ii \partial_t u+\del_{xx} u+P_{\vec{m}}*u+f(u,u_x,u_{xx})=0,
\quad u=u(x,t), \quad x\in \TTT,\\
\end{equation}
where $\TTT:=\RRR/2\pi\ZZZ$. The nonlinearity  $f$ is  a polynomial of degree $\bar{q}\geq 2$ defined on $\CCC^3$  vanishing at order $2$ near the origin  of the form
\begin{equation}\label{funz1}
f(z_{0},z_{1},z_{2})=
\sum_{p=2}^{\bar{q}}\sum_{(\alpha,\be)\in A_{p} }
C_{\al,\be}z_{0}^{\al_{0}}\ol{z}_{0}^{\be_{0}}z_{1}^{\al_{1}}\ol{z}_{1}^{\be_{1}}z_{2}^{\al_{2}}
\ol{z}_{2}^{\be_{2}}, \quad C_{\al,\be}\in\CCC,
\end{equation} 
where 
\begin{equation}\label{nonlinear2}
A_{p}:=\big\{ (\al,\be):=(\al_0,\be_0,\al_1,\be_1,\al_2,\be_2)\in \NNN^{6} \; {\rm s. t. }\; 
\sum_{i=0}^{2}\al_i+\be_i=p 
\big\}\,.
\end{equation}
The potential $P_{\vec{m}}(x)=(\sqrt{2\pi})^{-1}\sum_{j\in\ZZZ} \hat{p}(j)e^{\ii jx}$ is a \emph{real} function with 
\emph{real} Fourier coefficients  and the term $P_{\vec{m}}* u$ denotes the convolution between the potential $P_{\vec{m}}(x)$ and $u(x)=(\sqrt{2\pi})^{-1}\sum_{j\in\ZZZ}\hat{u}(j)e^{\ii jx}$
\begin{equation*}
P_{\vec{m}}*u(x)=\int_{\TTT}P_{\vec{m}}(x-y)u(y)dy=\sum_{j\in\ZZZ}\hat{p}(j)\hat{u}(j)e^{\ii j x}.
\end{equation*}
Concerning the convolution potential $P_{\vec{m}}(x)$ we define its $j$-th Fourier coefficient as follows. Fix $M>0$ and set
\begin{equation}\label{potenziale1}
\hat{p}(j):=\hat{p}_{\vec{m}}({j})=\sum_{k=1}^{M}\frac{m_{k}}{\langle j\rangle^{2k+1}},
\end{equation}
where $\vec{m}=(m_1,\ldots,m_M)$ is a vector in $\calO:=[-1/2,1/2]^{M}$ and $\langle j\rangle=\sqrt{1+|j|^2}$. 
We shall assume some extra structure on the polynomial nonlinearity $f$. Setting $z=\xi+\ii\eta$ in $\CCC$ 
(with $\xi$ and $\eta$ in $\RRR$) we define the Wirtinger derivatives 
$\del_{z}=\frac{1}{2}(\del_{\x}-\ii\del_{\eta}) $, $\del_{\ol{z}}=\frac{1}{2}(\del_{\x}+\ii\del_{\eta})$ 
and we assume the following
\begin{hyp}\label{parity}
The function $f$ in \eqref{NLS} and in \eqref{funz1} satisfies the following 
:
\begin{enumerate}
\item  {\bf parity-preserving}: $f(z_0,z_1,z_{2})=f(z_0,-z_1,z_{2})$;
\item {\bf Schr\"odinger-type}: $(\del_{z_{2}}f)(z_0,z_1,z_{2})\in \RRR$;
\item {\bf reversibility-preserving}: $f(z_0,z_1,z_{2})=\ol{f(\bar{z}_0,\bar{z}_1,\bar{z}_{2})},$
\end{enumerate}
for any $(z_0,z_1,z_2)$ in $\CCC^3$.
\end{hyp}

We shall study equation \eqref{NLS}
on the  Sobolev space
\begin{equation}\label{1.1}
H^{s}:=H^{s}(\TTT; \CCC):=
\left\{
u(x)=\sum_{j\in \ZZZ} \hat{u}(j)\frac{e^{\ii j  x}}{\sqrt{2\pi}} :
\|u\|^{2}_{H^{s}}:=\sum_{j\in\ZZZ}|\hat{u}(j)|^{2}\langle j\rangle^{2s}<+\infty
\right\}
\end{equation}
with $s$ to be chosen big enough.

The goal of this article is to show that solutions of \eqref{NLS} produced by  initial data even in $x$ of size $\varepsilon\ll 1$ are defined over a time interval of length $c_N\varepsilon^{-N}$ for any $N$ smaller than $M$ and for a large set of parameters $\vec{m}$ in $[-1/2,1/2]^M$.

The main result of the paper is the following.
\begin{theo}[{\bf Long time existence}]\label{teototale}
Fix $M\in \NNN$ and consider equation \eqref{NLS}.
Assume that $f$ satisfies Hypothesis \ref{parity}.
Then there is a zero Lebesgue measure set $\NN \subset \calO$ such that for any  integer $0\leq N\leq M$ 
and any $\vec{m}\in \calO\setminus \NN$
there exists $s_0\in \RRR$ such that for any $s\geq s_0 $ there are constants $r_0\in (0,1)$,  $c_N>0$ and $C_N>0$
such that the following holds true.
For any $0<r\leq r_0$ and any even function $u_0$ in the ball of radius $r$ 
of $H^{s}(\TTT;\CCC)$,  the equation \eqref{NLS} with initial datum $u_0$ has a unique solution,
which is even in $x\in \TTT$, and 
$$
u(t,x) \in C^{0}\Big([-T_r,T_{r}]; H^{s}(\TTT)\Big),
\;\;\; {\rm with} \;\;\; T_{r}\geq c_Nr^{-N}.
$$
Moreover one has that 
$$
\sup_{t\in (-T_{r},T_{r})}\|u(t,\cdot)\|_{H^{s}}\leq C_Nr. 
$$
\end{theo}

As far as we know this is the first long time existence result 
for a fully nonlinear Schr\"odinger equation on a compact manifold.
We remark that, besides the mathematical interest, fully nonlinear Schr\"odinger type equations 
 often appear in the description of phenomena 
 in which the wave packet disperses in media, see for instance \cite{zach}. 
We quote moreover the paper \cite{LLNR} in which fully nonlinear Schr\"odinger equations
(see for instance eq. (8) therein)
appear in the study of Kelvin waves in the superfluid turbulence.

\subsection*{Comments on the hypotheses}
Since the Fourier coefficients in \eqref{potenziale1} decay as $\langle j\rangle^{-3}$ as $j$ goes to $\infty$,  the potential $P_{\vec{m}}(x)$ is a function in $H^{s}$ for any $s<5/2$ (in particular it is  of class $C^1(\TTT;\RRR)$).
In  \cite{FIloc} (see Theorem 1.2 therein) it is shown that, if $P_{\vec{m}}(x)$ is a function of class $C^{1}$ with real Fourier coefficients, under the Hypothesis \ref{parity} (in such a theorem the reversibility structure of the nonlinearity in item 3 of Hyp. \ref{parity} is not needed) for any even function $u_0$ in the Sobolev space $H^s$
the Cauchy problem associated to the equation \eqref{NLS} with initial datum $u_0$ is locally in time well posed in $H^s(\TTT)$ if $s$ is big enough and the Sobolev norm of $u_0$ is small enough.
In order to treat more general initial data (not even in $x$)
one has to require a different algebraic structure on the equation because, in general, 
problem \eqref{NLS} is not well posed. For more details  we refer to the introduction 
of \cite{FIloc}.
An important case in which the equation \eqref{NLS} is well-posed on the whole $H^{s}(\mathbb{T})$
is the \emph{Hamiltonian} one.
Therefore it is interesting to understand whether a result similar to Theorem \ref{teototale}
holds in the Hamiltonian case.
This does not follow straightforward from the arguments  developed in this paper.
We shall give some ideas on this in the last lines of the introduction.

  In more natural problems the convolution potential is replaced by a multiplicative one. Since the convolution 
  is a diagonal operator on the Fourier space it is  easier the study of the resonances of the equation.
The particular structure of the Fourier coefficients of the convolution potential in \eqref{potenziale1} is inspired by
the Dirichlet spectrum  of $-\partial_{xx}+V(x)$. Indeed  for any $\rho\in\N^{*}$ it admits an asymptotic expansion of the form
\begin{equation*}
\lambda_j = j^2+c_0(V)+c_1(V)j^{-2}+\ldots+c_{\rho}(V)j^{-2\rho}+C_{\rho}(j,V)j^{-2-2\rho},
\end{equation*}
where $c_k(V)$ are certain multilinear  functions of the Fourier coefficients of $V(x)$,
$C_{\rho}(j,V)$ is  a constant uniformly bounded in $j$ 
depending on the derivatives of $V(x)$. For more details 
we refer to Section 5.3 of \cite{BG} and the references therein.
In order to treat the case of the multiplicative potential, for instance 
if it is 
smooth and nonnegative,
one should use the basis of $L^{2}(\mathbb{T})$ given  by the eigenfunctions
of the operator $-\del_{xx}+V(x)$ instead of the Fourier basis.
This would be a possible extension  of our result
by adapting the ideas introduced in \cite{BDGS}, \cite{BG}.

Item 1 in Hypothesis \ref{parity} implies that if $u(x)$ is even in $x$ then so is the function $f(u,u_x, u_{xx})$. Since the Fourier coefficients of $P_{\vec{m}}(x)$ in \eqref{potenziale1} are even in $j$,
 the flow of the equation \eqref{NLS} leaves invariant the space of even functions.

We assume item 2 in order to avoid the presence of parabolic terms in the nonlinearity, so that the equation \eqref{NLS} is a Schr\"odinger type one.

Item 3, together with the fact that the convolution potential $P_{\vec{m}}(x)$ is real valued, makes the equation \eqref{NLS} \emph{reversible} with respect to the involution 
\begin{equation}\label{involuzione}
S:u(x)\mapsto\bar{u}(x),
\end{equation}
in the sense that it has  the form $\partial_t u=X(u)$ with  $S\circ X=-X\circ S$. 
Since $f$ is assumed  to be a polynomial function as in  \eqref{funz1},  item 3 of the hypothesis is equivalent to requiring  that the coefficients $C_{\alpha,\beta}$ are real.
 One of the important dynamical consequences of the \emph {reversible} structure of the equation is that if $u(t,x)$ is a solution of the equation with initial condition $u_{0}$ then $S(u(-t,x))=\ol{u}(-t,x)$ solves the same equation with initial condition $\ol{u}_0$. This symmetry of the 
 equation is essential for our strategy and  will play a fundamental 
 role in the paper.

 We have chosen to study a polynomial nonlinearity in order to avoid extra technicalities.

\subsection*{Birkhoff Normal Form approach and some related literature}
Equation \eqref{NLS} 
 belongs to the following general class of problems:   
%
%
\begin{equation}\label{esempio}
u_t=Lu+\mathtt{f}(u),
\end{equation}
where 
$L$ is an unbounded linear operator with discrete spectrum made of purely imaginary eigenvalues $\la_{j}\in \ii\R$, $\mathtt{f}(u)$ is a non linear
function  having a zero of order at least two in the origin
and $u$ belongs to some Sobolev space. 
In the last years several authors investigated whether there is a stable behavior  of solutions of small amplitude.
By stable solution we mean that    its Sobolev norms $\|\cdot\|_{H^{s}}$ 
are bounded from above, up to moltiplicative constants, by the Sobolev norms 
of the initial datum
for long times.

This problem is non trivial when  the system \eqref{esempio} 
does not enjoy conservation laws able to control  Sobolev norms with high index $s$. 
In such a case the only general fruitful approach seems to be 
the Birkhoff Normal Form (BNF) procedure. 
Below we briefly describe the basic ideas and the difficulties that arise 
in implementing such a procedure.

According to the local existence  theory (at a sufficiently large order of regularity), 
if it exists,
we deduce that if the size of the initial datum is $\varepsilon\ll 1$ then the corresponding solution may be extended up to a time of magnitude $1/\varepsilon$.
The basic  idea to prove a longer time of 
existence using  a BNF approach is to reduce the \emph{size} of the non linearity near the origin.
In other words one looks for  a change of coordinates in order to cancel out,  from the non linearity, when possible, all the monomials of  homogeneity less than 
$N$ for some $N\geq2$. 
In this way, in the 
new coordinate system, 
one has that $\mathtt{f}(u)\sim u^{N}$, and 
hence the lifespan  is of order $\e^{-N+1}$.
In performing such changes of coordinates 
 non trivial problems arise:
\begin{itemize}
\item[$(i)$]  \emph{small divisors} appear: the small divisors involve 
linear combinations of the eigenvalues $\la_{j}$, $j\in\NNN$,
 of the linear operator $L$ in \eqref{esempio}
 of the form
  \begin{equation}\label{lowerLambda}
\lambda_{j_1}+\dots+\lambda_{j_{\ell}}-\lambda_{j_{\ell+1}}-\ldots-\lambda_{j_{N}}
\end{equation}
 for $0\leq \ell\leq N$ with $N\in \NNN$.
 One must impose 
  \emph{non-resonance} conditions, i.e. lower bounds on the quantity in \eqref{lowerLambda} 
  whenever  possible.

 \item[$(ii)$]  one has to check that the changes of coordinates are well-defined and bounded, on sufficiently regular Sobolev spaces,
 even if some loss of regularity appears due to the small divisors;
  
\item[$(iii)$] it is not possible to cancel out \emph{all} the monomials of low degree of homogeneity from the non linearity
but, starting form \eqref{esempio}, one obtains a system of the form
\[
u_t=Lu+Z(u)+P(u),
\] 
where $P(u)\sim u^{N}$ and  the non linear term $Z$ (which is usually called ``resonant normal form'') 
commutes with the operator $L$. 
Under some algebraic assumptions on the nonlinearity $\mathtt{f}(u)$ 
 the dynamics generated by the resonant term $Z(u)$
 is stable.
The most studied models in literature are the \emph{Hamiltonian} and the 
  \emph{reversible} PDEs. 
\end{itemize}

Without trying to be exhaustive we quote below some relevant contributions to this subject.

Concerning  \emph{semi-linear} PDEs (i.e. when the non linearity $\mathtt{f}(u)$ does not contain derivatives of $u$)
the long time existence problem has been extensively studied in literature in the case of \emph{Hamiltonian} PDEs.
We quote for instance the papers by Bambusi \cite{Bambu}, Bambusi-Gr\'ebert \cite{BG}, 
 Delort-Szeftel \cite{DelortSzeft1, DelortSzeft2}, and the more recent result 
 by Bernier-Faou-Gr\'ebert \cite{BFG}.

 Regarding BNF theory for \emph{reversible} PDEs we mention 
\cite{grefaou} by Gr\'ebert-Faou.
The paper \cite{BDGS}  regards long time existence of solutions for the semi-linear Klein-Gordon equation on Zoll manifolds, here are collected all the ideas of the preceding (and aforementioned) literature.

 In the case that the non linearity $\mathtt{f}$ contains derivatives of $u$ 
  if one would  follow the strategy used in the semilinear case, 
  one would end up with only formal results in the following sense:
the change of coordinates would be unbounded because one faces   
the well known problem of \emph{loss of derivatives}.
We remark that this loss of derivatives is  originated by  the presence of derivatives in the nonlinearity and it is not a 
small divisors problem.
In this direction we quote the early paper concerning the \emph{pure-gravity water waves} (WW) equation by Craig-Worfolk \cite{Craigworfo}.

 In the case that $\mathtt{f}(u)$ in \eqref{esempio} contains
 derivatives of $u$ of order strictly less than the order of $L$, we quote  the paper by Yuan-Zhang \cite{YuZhan}.
 They studied  an equation of the form \eqref{NLS} with the particular nonlinearity $f(u,u_x)=-(\ii/2\pi)(|u|^{2}u)_{x}$
 exploiting its   Hamiltonian  structure.

The first rigorous long time existence result concerning $\emph{quasi-linear}$ equations, i.e.
when  $\mathtt{f}$ contains derivatives of $u$ of the same order of  $L$
has been obtained by Delort. 
In
\cite{Delort-2009} the author studied 
quasi-linear Hamiltonian perturbations of  the \emph{Klein-Gordon} (KG) equation
on the circle, and in \cite{Delort-sphere} the same equation on higher dimensional spheres.
Here the author  introduces some  classes of multilinear maps which
 defines \emph{para-differential} operators (in the case of (KG) operators of order $1$) enjoying a \emph{symbolic} calculus. 
  We remark that in such papers  the author  deeply use the fact that the (KG) has a linear \emph{dispersion law} (i.e. the operator $L$
 in this case has order $1$).

A  different approach in the case of \emph{super-linear } dispersion law (i.e. $L$ has order $> 1$) has been proposed by Berti-Delort in \cite{maxdelort} for the \emph{capillary water waves} equation. 
In this paper we adopt the strategy proposed in \cite{maxdelort}.
In the following we briefly explain this approach. In the next paragraph 
we shall enter more in detail introducing the  appropriate notation.

The starting point is to rewrite the equation as a \emph{para-differential} system which involves a \emph{para-differential} term (see Definition \ref{quantizationtotale})
and a smoothing remainder (see Definition \ref{smoothoperator}). This 
procedure is known in literature as the \emph{Bony paralinearization} of the equation 
(see section
\ref{siparalin}).
Consecutively the BNF procedure is divided into two steps:
 \begin{enumerate}
 \item
 Instead of reducing directly the \emph{size} of the non linearity (as done in \cite{Delort-2009} for (KG) or formally in \cite{CraigSulem} and \cite{Craigworfo} for the (WW))
we perform
 some 
 para-differential reductions in order   
 to conjugate the para-differential term 
 to an other one which is diagonal 
 with constant coefficients in $x$ up to 
 a remainder which is 
 a very regularizing term. 
 In this procedure it is fundamental that the symbols of positive order are 
 purely imaginary, in such a way that the associated para-differential 
 operator is skew self-adjoint. 
 This condition is ensured by Hypothesis \ref{parity}.
 A related regularization procedure 
 of the unbounded terms of the equations has been 
 previously developed in order to study the linearized equation associated to a quasi-linear system in the context of a Nash-Moser iterative scheme (see for instance \cite{alaz-baldi-periodic,BM1,BBM,BBM1,FP,FP2}).

 \item The second part of the procedure consists  in two sub-steps. In the first one a BNF procedure is used in order to reduce the size of the paradifferential term. The loss of derivatives appearing in the BNF  procedure affects only the coefficients of the equations which are low frequencies thanks to the paradifferential structure; we may afford to loose a large number of derivatives on these coefficients since we are working with very smooth functions. 
Concerning the reduction in size of the smoothing remainder we construct 
some \emph{modified energies} by means of, again, a BNF-type procedure.
By modified energy at order $N\in \mathbb{N}$ we mean
a quantity $E_s(U)$ such that $E_s(U)\sim\norm{U(t,\cdot)}{H^s}^2$ and 
 \begin{equation}\label{botte10}
 E_s(U(t,\cdot))\leq E_s(U(0,\cdot))+\asso{\int_0^t\norm{U(\tau,\cdot)}{H^s}^{N+2}d \tau}.
 \end{equation}
 The loss of derivatives due to the small divisors, in this case,  
 is compensated by the fact that the  remainder is 
 a very smoothing operator.
 \end{enumerate}
For more details on this strategy we refer the reader to the introduction of \cite{maxdelort}.

We mention that in \cite{FP2}, \cite{FP} it has been shown that  a large class of   \emph{fully non linear}   Schr\"odinger type equations
admits quasi-periodic in time, and hence globally defined and stable, small amplitude solutions.
Hence it would be interesting to study whether other stability phenomena appear.

The goal of this paper is to extend the BNF theory to a class of   \emph{fully non linear}   Schr\"odinger equations by adapting the ideas of \cite{maxdelort}.

%
%

\subsection*{Plan of the paper}
First of all it is convenient to work on product spaces and consider instead of \eqref{NLS} the so called \emph{vector NLS}. We need some further notation. We define, for $s>0$, the following Sobolev spaces
\begin{equation}\label{Hcic}
\begin{aligned}
&{\bf{ H}}^s:={\bf{ H}}^s(\TTT,\CCC^2):=\big(H^{s}\times H^{s}\big)\cap \gotR, \qquad  \gotR:=\{(u^{+},u^{-})\in L^{2}(\TTT;\CCC)\times L^{2}(\TTT;\CCC)\; : \; u^{+}=\ol{u^{-}}\},
\end{aligned}
\end{equation}
endowed with the product topology. We set ${\bf{ H}}^{\infty}:=\cap_{s\in\R} {\bf{ H}}^s$. On ${\bf{ H}}^0$ we define the scalar product
\begin{equation}\label{comsca}
(U,V)_{{\bf{ H}}^0}:=\int_{\TTT}U\cdot\ol{V}dx.
\end{equation}
We introduce also the following subspace of even functions of $x$ in $\hcic^{s}$:
\begin{equation}\label{spazipari}
\begin{aligned}
{\bf{ H}}_{e}^s={\bf H}^{s}_{e}(\TTT;\CCC^{2})&:=(H^{s}_{e}\times H^{s}_{e})\cap \hcic^{0},\;\;\;\;\;\;
H^{s}_{e}=H^{s}_{e}(\TTT;\CCC):=\{u\in H^{s}\; : \; u(x)=u(-x)\}.
\end{aligned}
\end{equation}

We 
define 
the operator $\lambda$ as $\lambda[u]:=\del_{xx}u+P_{\vec{m}}(x)*u$. One has that  
\begin{equation}\label{NLS1000}
\lambda [e^{\ii jx}]:= \lambda_{j} e^{\ii jx}, \qquad \lambda_{j}:=(\ii j)^{2}+\hat{p}({j}), \quad \;\; j\in \ZZZ,
\end{equation}
where $\hat{p}(j)$ are defined in \eqref{potenziale1}. 
Let us introduce the following matrices
\begin{equation}\label{matrici}
E:=\sm{1}{0}{0}{-1}
\quad \uno:=\sm{1}{0}{0}{1}
\end{equation}
and
we define the operator $\Lambda$ on $H^{s}\times H^{s}$ as 
\begin{equation}\label{DEFlambda}
 \Lambda:=\sm{\lambda}{0}{0}{\lambda}.
\end{equation}
With this notation equation \eqref{NLS} is equivalent to the system

\begin{equation}\label{6.666}
\dot{U}:=\ii E\left[\Lambda U
+
\left(\begin{matrix} f(u,{ u}_{x},{ u}_{xx})\vspace{0.4em}\\ 
\ol{f({u },{u}_{x},{u}_{xx})}
\end{matrix}
\right)\right], \quad U=(u,\bar{u})\in \hcic^{s}.
\end{equation}
From now on we will study the system \eqref{6.666} instead of equation \eqref{NLS}.

We describe here the ideas of the proof of Theorem \ref{teototale}. In Section \ref{paracalc} we define the classes of operators and of symbols we need and we develop some composition theorems for classes of para-differential and smoothing operators. Such classes of operators have been introduced and widely studied in \cite{maxdelort}.
 
The first step is to rewrite the system \eqref{6.666} as a para-differential system of the form \eqref{sistemainiziale}
 with $U=U(t,x)=(u,\bar{u})$ by using the results of Section \ref{paracalc}.
 This is the content of Theorem \ref{paralineariza} in Section 
  \ref{siparalin}. 
Let us describe briefly the  structure of the system obtained in Theorem \ref{paralineariza}. 
Consider a symbol $a(x,\x)$ having finite regularity in $x$.  
Let $\chi$ be a $C^{\infty}_0$ cut-off function with sufficiently small support and equal to $1$ close to $0$, 
then we set
$a_{\chi}(x,\xi)=\mathcal{F}_{\hat{x}}^{-1}\big(\hat{a}(\hat{x},\xi)\chi(\hat{x}/\langle\xi\rangle)\big)$.  
In other words the new symbol $a_{\chi}(x,\xi)$ is a localization in the Fourier space, 
and therefore a regularization in the physical space, of the symbol $a(x,\xi)$.
 Then we can define the Bony-Weyl quantization of the symbol $a$ as follows
\begin{equation*}
\bonyw(a(x,\xi))\varphi=\frac{1}{2\pi}\int e^{\ii(x-y)\xi}a_{\chi}\Big(\frac{x+y}{2},\xi\Big)\varphi(y)dyd\xi.
\end{equation*}
Theorem \ref{paralineariza} ensures that
the original system is equivalent to the following para-differential one
\begin{equation}\label{carta}
\partial_t U= \ii E(\Lambda U+\bonyw(A(U;x,\xi))U+R(U)U)
\end{equation}
where $R(U)$ is a $2\times 2$ matrix of smoothing remainder, $A(U;x,\x)$ is a $2\times2$ matrix of symbol with the following properties:
\begin{itemize}
\item the map $(x,\xi)\mapsto A(U;x,\xi)$ depends in a non linear way on the function $U$ solution of \eqref{6.666};
\item for any $N>1$ the matrix of symbols $A(U;x,\xi)$ 
admits an expansion in homogeneous matrices of  symbols
up to a non homogeneous one of size $O(\|U\|^{N}_{H^s})$;
\item the operator $\bonyw(A(U;x,\x))$ maps $\hcic^{s}$ in $\hcic^{s-2}$ for any $s$, provided that
$U$ belongs to $\hcic^{s_0}$ for $s_0$ large enough;
\item $R(U)$ maps $\hcic^{s}$ in $\hcic^{s+\rho}$ for $s$ large enough and $\rho\sim s$;

\item the operators $\bonyw(A(U;x,\x))$, $R(U)$ are \emph{reversibility} and \emph{parity} preserving
according to Definitions \ref{riassunto-mappe} and \ref{riassunto-simbo}.
This structure is inherited by Hypothesis \ref{parity}.

\end{itemize}

We remark that we cannot  use directly  the paralinearization performed in Section $4$ of \cite{FIloc} 
since we need to adapt it to symbols and operators which admit multilinear expansions.

  As first step 
  in subsections \ref{diagosecondord}, \ref{diagosecondord2}
  we perform several changes of coordinates which diagonalize the matrix $A(U;x,\x)$. 
   Since the non zero terms  of the new diagonal matrix of symbols
  depends on $x$, this system does not admit ${\bf H}^s$-energy estimates (i.e. an estimate 
  of the form \eqref{botte11}).
  Therefore subsections \ref{diagosecondord3}, 
  \ref{diagosecondord4}, \ref{diagosecondord5} are devoted to conjugate this matrix to another one 
  whose symbols are constant in $x$.
%
 All these results are collected in Theorem \ref{regolarizza}, where we exhibit a nonlinear map 
 $\Phi(U)U$ with the following properties:
 \begin{itemize}
 \item[(a)] for any fixed $U$ in $\hcic^{s_0}$, $s_0$ large enough, the map
 $\Phi(U)[\cdot]$ is a bounded linear map form $\hcic^{s}$ to $\hcic^{s}$ for any $s\geq0$;
 
 \item[(b)] set $V:=\Phi(U)U$, then one has
 $\|V\|_{\hcic^s}\sim \|U\|_{\hcic^{s}}$;
 
 \item[(c)] the map $\Phi(U)$ is \emph{reversibility} and \emph{parity} preserving;
 
 \item[(d)] the function $U$ solves \eqref{carta} if and only if $V=\Phi(U)U$ solves a system of the form
 (see \eqref{problemafinale})
 \begin{equation}\label{carta2}
 \del_{t}V=\ii E( \Lambda V+\bonyw(L(U;\x))V+Q(U)U),
 \end{equation}
for some diagonal and constant coefficients in $x$ matrix of 
symbol $L(U;\x)$ (see \eqref{constCoeff})
and where $Q(U)$ is a $\rho$-smoothing 
remainder for some 
$\rho\gg N$ large.
  \end{itemize}
 The function $V$ solving \eqref{carta2} satisfies
\begin{equation}\label{botte11}
 \del_{t}\|V(t)\|^{2}_{\hcic^s}\leq C\|U(t)\|_{\hcic^{s_0}}\|V(t)\|^2_{\hcic^s},
 \end{equation}
 therefore, as a consequence of Theorem \ref{regolarizza}, we have obtained 
 \begin{equation}\label{stimabella}
 \|U(t)\|_{\hcic^s}^2\leq C\|U(0)\|^2_{\hcic^s}+C\int_0^t\|U(\tau)\|_{\hcic^{s_0}}\|U(\tau)\|^2_{\hcic^s}d\tau, \quad
 s\geq s_0\gg1.
 \end{equation}
  The estimate \eqref{botte11} is a consequence of the fact that
  the symbol $\mathtt{m}_{2}(U;t)$ in \eqref{constCoeff}
  is real and that there are  not symbols of order $1$ in $\mathtt{m}(U;t)$.
  This follows from the parity structure of the equation and the parity preserving structure of the 
  map $\Phi(U)$.

There are two key differences between this paper and the procedure followed in \cite{FIloc}.
In the quoted paper we are only  interested in giving 
some energy estimates on the solution in order to prove a local existence result. 
Here the situation is more complicated and we need further information in order  to obtain a much longer time of existence. 
First of all in Theorem \ref{regolarizza}  we  take into account that our operators
and symbols admit multilinear expansions. This justify our definition of operators and symbols
in Definitions \ref{smoothoperator} and \ref{simbototali}. On the contrary in \cite{FIloc} we use classes more similar to 
the non homogeneous classes defined in Definitions \ref{nonomoop} and \ref{nonomorest}. 
  The second fundamental difference is that the final system in \eqref{problemafinale}
  is diagonal, constant coefficients in $x\in\TTT$, up to terms which are $\rho$-smoothing operators 
  with $\rho$ arbitrary large. We remark that in \cite{FIloc} we only need 
  \emph{bounded} remainders.

   In Section \ref{sezBNF} we  give the proof of Theorem \ref{teototale}. 
   Notice that the r.h.s. in  \eqref{stimabella} is linear in $\|U(t)\|_{\hcic^{s_0}}$ 
   since both 
   $\|\bonyw\big({\rm Im}(\mathtt{m}_0(U;t,\x))\big)\|_{\mathcal{L}({\bf H}^{s},{\bf H}^{s-2})}$
   (see \eqref{constCoeff})
    and 
   $\|Q(U)\|_{\mathcal{L}({\bf H}^{s},{\bf H}^{s+\rho})}$
   are $O(\|U\|_{{\bf H}^{s_0}})$.
   The aim of Sec. \ref{sezBNF} is to prove an estimate of the form
    \begin{equation}\label{stimabella2}
 \|U(t)\|_{\hcic^s}^2\leq C\|U(0)\|^2_{\hcic^s}+C\int_0^t\|U(\tau)\|^{{N}}_{\hcic^{s_0}}\|U(\tau)\|^2_{\hcic^s}d\tau, \quad
 s\geq s_0\gg1, \;\;\; N>2.
 \end{equation}
  After the reduction performed in Theorem \ref{regolarizza}, we have that the system \eqref{carta2} is very similar to a semi-linear one. Therefore we construct modified energies (see \eqref{botte10}) 
  to prove 
  the bound \eqref{stimabella2}.  In such construction we exploit the reversibility and parity preserving structures
   in order to prove that the resonant terms do not contribute to the energy estimates
   (see Definition \ref{kernel} and Lemma \ref{FEI}).
   As explained before (below formula \eqref{botte10}) to get the \eqref{stimabella2}
   we shall face a loss of derivatives 
   of magnitude $ N\cdot N_0$, due to small divisors appearing in the BNF procedure, 
   where $N_0$ is a fixed quantity
   given by Proposition
   \ref{stimemisura}. This is the reason for which we fix $\rho\gg N$ at the beginning of the 
procedure.
In the  Proposition \ref{stimemisura} we show that
 the $\la_j$'s (defined in \eqref{NLS1000}) satisfy the needed  non resonance conditions.
%
  
  We conclude the introduction discussing  briefly why this strategy does not 
  straightforward apply to the \emph{Hamiltonian} case.
As explained above we need to exploit some algebraic structures to ensure that the resonant terms
do not contribute to energy inequality. 
Therefore one has to preserve such structures in performing changes of coordinates. 
  In the present paper it is rather simple to do it.
A key point is that after the paralinearization in \eqref{carta}
  both the term $\bonyw(A(U;t,x,\x))$ and $R(U)$ are reversibility and parity preserving.
  In the Hamiltonian case,   among all the difficulties,  the latter terms are not Hamiltonian vector fields.
  For this reason it is not trivial  to build \emph{symplectic} versions of the changes 
  of coordinates used in this paper.
  
  \medskip
\noindent
{\bf Acknowledgements}. We warmly thanks Prof. Massimiliano Berti for suggesting us this problem and Prof. Jean-Marc Delort for many useful suggestions and comments.

\section{Para-differential calculus}\label{paracalc}

In this section we develop a para-differential calculus 
following the ideas  (and notation) in \cite{maxdelort}.

 In subsections \ref{subsec:Smooth} and \ref{subsec:Simboli} 
 we introduce, respectively, 
 several classes of smoothing operators and symbols 
 depending on some extra function $U$.
 More precisely our symbols and operators are polynomials in $U$ up to a degree of homogeneity $N-1$ plus a non-homogeneous term  which vanishes as $O(\|U\|^N)$ as $U$ goes to $0$
 (see Definitions \ref{smoothoperator} and \ref{simbototali}). 
 We  define a para-differential quantization, see Definition \ref{quantizationtotale},
  of such symbols 
  and we  prove, in subsection \ref{subsec:COmpo}, 
  that they enjoy a \emph{symbolic} calculus up to smoothing operators introduced in
  subsection \ref{subsec:Smooth}.  
  In subsection \ref{subsec:Mapp} we introduce a class of general maps 
  (see Definition \ref{smoothoperatormaps}) which will be used in
  some contexts where will not be important to keep track of the 
  loss of derivatives coming from unbounded operators.
  When applying this theory to the \emph{reversible} and \emph{parity preserving} 
  Schr\"odinger equation \eqref{NLS}, we shall deal with subspaces of symbols 
  and operators defined above enjoying some algebraic properties. 
   These subclasses are introduced and analyzed in subsection \ref{subsec:Alge}. 
  The differences between our classes and those in \cite{maxdelort} depend only on the extra function $U$: in their case it  is a function of time and space $(x,t)$ which is of class $C^k$, w.r.t. the variable $t$, with values in $H^{s-\frac{3}{2}k}$ for any $0\leq k\leq K$ ($K$ big enough) and it has zero mean, in our case it can have non zero mean and it is a function of class $C^k$, w.r.t. the variable $t$, with values in $H^{s-2k}$ for any $0\leq k\leq K$ ($K$ big enough).

We introduce some notation. If $K\in\N$, $I$ is an interval of $\R$ containing the origin and  $s\in\R^{+}$ we denote by $C^K_{*}(I,{{H}}^{s}(\TTT,\CCC^2))$ (respectively  $C^K_{*}(I,{{H}}^{s}(\TTT;\CCC))$), the space of continuous functions $U$ of $t\in I$ with values in  ${{H}}^{s}(\TTT,\CCC^2)$ (resp. $H^{s}(\TTT;\CCC)$), 
which are $K$-times differentiable and such that the $k-$th derivative is continuous with values in 
${{H}}^{s-2k}(\TTT,\CCC^2)$ (resp. $H^{s-2k}(\TTT;\CCC)$) 
for any $0\leq k\leq K$. We endow the space  $C^K_{*}(I,{{H}}^{s}(\TTT;\CCC^{2}))$ 
(resp. $C^K_{*}(I,{{H}}^{s}(\TTT;\CCC))$) with the norm
\begin{equation}\label{spazionorm}
\sup_{t\in I}\norm{U(t,\cdot)}{K,s}, \quad \mbox {where} \quad \norm{U(t,\cdot)}{K,s}:=\sum_{k=0}^{K}\norm{\partial_t^k U(t,\cdot)}{{{H}}^{s-2k}}.
\end{equation}
We denote by $C^K_{*\RRR}(I,{{H}}^{s}(\TTT,\CCC^2))$, sometimes with $C^{K}_{*\RRR}(I;\hcic^{s})$, 
the subspace of $C^K_{*}(I,{{H}}^{s}(\TTT,\CCC^2))$
made of the functions of $t$  with values in $\hcic^{s}(\TTT;\CCC^{2})$ (see \eqref{Hcic}).
Recalling \eqref{spazipari} we shall denote 
$C^{K}_{*}(I;H_{e}^{s}(\TTT;\CCC^{2}))$ (resp. $C^{K}_{*}(I;H_e^{s}(\TTT;\CCC))$) 
the subspace of $C^K_{*}(I,{{H}}^{s}(\TTT,\CCC^2))$ (resp. $C^{K}_{*}(I;H^{s}(\TTT;\CCC))$)
made of the functions of $t$ with values in $H_{e}^{s}(\TTT;\CCC^{2})$ (resp. $H^{s}_{e}(\TTT;\CCC)$).
Analogously 
$C^K_{*\R}(I,{\bf{H}}_{e}^{s}(\TTT;\CCC^{2}))$ denotes the subspace of $C^K_{*\R}(I,{\bf{H}}^{s}(\TTT;\CCC^{2}))$ made of those functions which are even in $x$.
Moreover if $r\in\R^{+}$ we set
\begin{equation}\label{pallottola}
B_{s}^K(I,r):=\set{U\in C^K_{*}(I,H^{s}(\TTT;\CCC^{2})):\, \sup_{t\in I}\norm{U(t,\cdot)}{K,s}<r}.
\end{equation}

For $n\in \NNN^*$ we denote by $\Pi_{n}$ the orthogonal projector from $L^{2}(\TTT;\CCC^{2})$ (or $L^{2}(\TTT,\CCC)$)
to the subspace spanned by $\{e^{\ii n x}, e^{-\ii nx}\}$ i.e.
\begin{equation}\label{spectralpro}
(\Pi_{n}u)(x)=\hat{u}({n}) \frac{e^{\ii nx}}{\sqrt{2\pi}}+\hat{u}(-n)\frac{e^{-\ii nx}}{\sqrt{2\pi}},
\end{equation}
while in the case $n=0$ we define the mean
$\Pi_0u=\frac{1}{\sqrt{2\pi}}\hat{u}(0)=\frac{1}{2\pi}\int_{\TTT}u(x)dx.$

If $\mathcal{U}=(U_1,\ldots,U_{p})$
is a $p$-tuple
 of functions, $\vec{n}=(n_1,\ldots,n_p)\in \NNN^{p}$, we set
 \begin{equation}\label{ptupla}
 \Pi_{\vec{n}}\mathcal{U}:=(\Pi_{n_1}U_1,\ldots,\Pi_{n_p}U_p).
 \end{equation}
 
 For a family $(n_1,\ldots,n_{p+1})\in \NNN^{p+1}$ we denote by
 $\max_{2}(\langle n_1\rangle,\ldots,\langle n_{p+1}\rangle)$,
 the second largest among the numbers $\langle n_1\rangle,\ldots,\langle n_{p+1}\rangle$.

 \subsection{Spaces of Smoothing operators}\label{subsec:Smooth}

 The following is the definition of a class of multilinear smoothing operators.
 \begin{de}[{\bf $p-$homogeneous smoothing operator}]\label{omosmoothing}
 Let $p\in \NNN$, $\rho\in \RRR$ with $\rho\geq0$. We denote by $\widetilde{\RR}^{-\rho}_{p}$
 the space of $(p+1)$-linear maps
 from the space $(C^{\infty}(\TTT;\CCC^{2}))^{p}\times C^{\infty}(\TTT;\CCC)$ to 
 the space $C^{\infty}(\TTT;\CCC)$ symmetric
 in $(U_{1},\ldots,U_{p})$, of the form
 $ (U_{1},\ldots,U_{p+1})\to R(U_1,\ldots, U_p)U_{p+1},$
 that satisfy the following. There is $\mu\geq0$, $C>0$ such that 
  \begin{equation}\label{omoresti1}
 \|\Pi_{n_0}R(\Pi_{\vec{n}}\mathcal{U})\Pi_{n_{p+1}}U_{p+1}\|_{L^{2}}\leq
 C\frac{\max_2(\langle n_1\rangle,\ldots,\langle n_{p+1}\rangle)^{\rho+\mu}}{\max(\langle n_1\rangle,\ldots,\langle n_{p+1}\rangle)^{\rho}}
 \prod_{j=1}^{p+1}\|\Pi_{n_{j}}U_{j}\|_{L^{2}},
 \end{equation}
  for any 
 $\mathcal{U}=(U_1,\ldots,U_{p})\in (C^{\infty}(\TTT;\CCC^{2}))^{p}$, any 
 $U_{p+1}\in C^{\infty}(\TTT;\CCC)$,
 any $\vec{n}=(n_1,\ldots,n_p)\in \NNN^{p}$, any $n_0,n_{p+1}\in \NNN$.
 Moreover, if 
 \begin{equation}\label{omoresti2}
 \Pi_{n_0}R(\Pi_{n_1}U_1,\ldots,\Pi_{n_{p}}U_{p})\Pi_{n_{p+1}}U_{p+1}\neq0,
 \end{equation}
 then there is a choice of signs $\s_0,\ldots,\s_{p+1}\in\{-1,1\}$ such that 
 $\sum_{j=0}^{p+1}\s_j n_{j}=0$.
 \end{de}

 We shall need also a class of non-homogeneous smoothing operators.

 \begin{de}[\bf{Non-homogeneous smoothing operators}]\label{nonomoop}
Let $K'\leq K\in\N$, $N\in \NNN$ with $N\geq1$, $\rho\in \RRR$ with $\rho\geq0$ 
and $r>0$. We define the class of remainders $\mathcal{R}^{-\rho}_{K,K',N}[r]$ as the space of maps $(V,u)\mapsto R(V)u$ defined on $B^K_{s_0}(I,r)\times C^K_{*}(I,H^{s_0}(\TTT,\CCC))$ which are linear in the variable $u$ and such that the following holds true. For any $s\geq s_0$ there exist a constant $C>0$ and $r(s)\in]0,r[$ such that for any $V\in B^K_{s_0}(I,r)\cap C^K_{*}(I,H^{s}(\TTT,\CCC^2))$, any $u\in C^K_{*}(I,H^{s}(\TTT,\CCC))$, any $0\leq k\leq K-K'$ and any $t\in I$ the following estimate holds true
\begin{equation}\label{porto20}
\norm{\partial_t^k\left(R(V)u\right)(t,\cdot)}{H^{s-2k+\rho}}\leq \sum_{k'+k''=k}C\Big[\norm{u}{k'',s}\norm{V}{k'+K',s_0}^{N}
+\norm{u}{k'',s_0}\|V\|_{k'+K',s_0}^{N-1}\norm{V}{k'+K',s}\Big].
\end{equation}
\end{de}

We will often use the following general class.
\begin{de}[{\bf Smoothing operator}]\label{smoothoperator}
Let $p,N\in \NNN$, with $p\leq N$, $N\geq1$, $K,K'\in\NNN$ with $K'\leq K$
and $\rho\in \RRR$, $\rho\geq0$. We denote by $\Sigma\RR^{-\rho}_{K,K',p}[r,N]$
the space of maps $(V,t,u)\to R(V,t)u$
that may be written as 
\begin{equation}\label{smoothoperator1}
R(V;t)u=\sum_{q=p}^{N-1}R_{q}(V,\ldots,V)u+R_{N}(V;t)u,
\end{equation}
for some $R_{q}\in \widetilde{\RR}^{-\rho}_{q}$, $q=p,\ldots, N-1$ and $R_{N}$ belongs to 
$\RR^{-\rho}_{K,K',N}[r]$.
\end{de}

\begin{rmk}
Let $R_1(U)$ be a smoothing operator in $\Sigma\mathcal{R}^{-\rho_1}_{K,K',p_1}[r,N]$ and $R_2(U)$ in $\Sigma\mathcal{R}^{-\rho_2}_{K,K',p_2}[r,N]$, then the operator $R_1(U)\circ R_{2}(U)[\cdot]$ belongs to  $\Sigma\mathcal{R}^{-\rho}_{K,K',p_1+p_2}[r,N]$, where $\rho=\min(\rho_1,\rho_2)$.
\end{rmk}

The following is a subclass of the previous class made of those operators which are autonomous, i.e. they depend on the variable $t$ only through the function $U$.
\begin{de}[{\bf Autonomous smoothing operator}]\label{autsmoothop}
We define, according to the notation of Definition \ref{nonomoop}, the class of autonomous 
non-homogeneous smoothing operator $\RR^{-\rho}_{K,0,N}[r,{\rm aut}]$
as the subspace of $\RR^{-\rho}_{K,0,N}[r]$ made of those maps 
$(U,V)\to R(U)V$ satisfying estimates \eqref{porto20}
with $K'=0$, the time dependence being only through $U=U(t)$.
In the same 
way, we denote by $\Sigma\RR^{-\rho}_{K,0,p}[r,N,{\rm aut}]$ the space of maps 
$(U,V)\to R(U,V)$ of the form \eqref{smoothoperator1}
with $K'=0$ and where the last term belongs to $\RR^{-\rho}_{K,0,N}[r,{\rm aut}]$.
\end{de}

\begin{rmk}\label{starship}
We remark that if $R$ is in $\widetilde{R}^{-\rho}_{p}$, $p\geq N$, then
$(V,U)\to R(V,\ldots,V)U$ is in $\RR^{-\rho}_{K,0,N}[r,{\rm aut}]$.
This inclusion follows by the multi-linearity of $R$ in each argument, and 
by estimate \eqref{omoresti1}. For further details we refer to 
the remark after Definition $2.2.3$ in \cite{maxdelort}.

\end{rmk}

\subsection{Spaces of Maps}\label{subsec:Mapp}

In the following, sometimes, we 
shall treat operators 
without having to keep track of the number of lost derivatives in a very precise way. We introduce some further classes.

 \begin{de}[{\bf $p-$homogeneous maps}]\label{omosmoothingmaps}
 Let $p\in \NNN$, $m\in \RRR$ with $m\geq0$. We denote by $\widetilde{\MM}^{m}_{p}$
 the space of $(p+1)$-linear maps
 from the space $(C^{\infty}(\TTT;\CCC^{2}))^{p}\times C^{\infty}(\TTT;\CCC)$ to 
 the space $C^{\infty}(\TTT;\CCC)$ symmetric
 in $(U_{1},\ldots,U_{p})$, of the form
 $ (U_{1},\ldots,U_{p+1})\to M(U_1,\ldots, U_p)U_{p+1},$
 that satisfy the following. There is $\mu\geq0$, $C>0$, and for any 
 $\mathcal{U}=(U_1,\ldots,U_{p})\in (C^{\infty}(\TTT;\CCC^{2}))^{p}$, any 
 $U_{p+1}\in C^{\infty}(\TTT;\CCC)$,
 any $\vec{n}=(n_1,\ldots,n_p)\in \NNN^{p}$, any $n_0,n_{p+1}\in \NNN$
 \begin{equation}\label{omoresti1maps}
 \|\Pi_{n_0}M(\Pi_{\vec{n}}\mathcal{U})\Pi_{n_{p+1}}U_{p+1}\|_{L^{2}}\leq
 C(\langle n_0\rangle+\langle n_1\rangle+\ldots+\langle n_{p+1}\rangle)^{m}
 \prod_{j=1}^{p+1}\|\Pi_{n_{j}}U_{j}\|_{L^{2}}.
 \end{equation}
 Moreover, if 
 \begin{equation}\label{omoresti2maps}
 \Pi_{n_0}M(\Pi_{n_1}U_1,\ldots,\Pi_{n_{p}}U_{p})\Pi_{n_{p+1}}U_{p+1}\neq0,
 \end{equation}
 then there is a choice of signs $\s_0,\ldots,\s_{p+1}\in\{-1,1\}$ such that 
 $\sum_{j=0}^{p+1}\s_j n_{j}=0$.
 When $p=0$ the conditions above mean that $M$ is a linear map on $C^{\infty}(\TTT;\CCC)$ into itself.
 \end{de}

 \begin{de}[\bf{Non-homogeneous maps}]\label{nonomoopmaps}
Let $K'\leq K\in\N$, $N\in \NNN$ with $N\geq1$, $m\in \RRR$ with $m\geq0$ 
and $r>0$. We define the class  $\mathcal{M}^{m}_{K,K',N}[r]$ as the space of maps $(V,u)\mapsto M(V)u$ defined on $B^K_{s_0}(I,r)\times C^K_{*}(I,H^{s_0}(\TTT,\CCC))$, for some $s_0>0$,  which are linear in the variable $u$ and such that the following holds true. For any $s\geq s_0$ there exist a constant $C>0$ and $r(s)\in]0,r[$ such that for any $V\in B^K_{s_0}(I,r)\cap C^K_{*}(I,H^{s}(\TTT,\CCC^2))$, any $u\in C^K_{*}(I,H^{s}(\TTT,\CCC))$, any $0\leq k\leq K-K'$ and any $t\in I$ the following estimate holds true
\begin{equation}\label{porto20maps}
\norm{\partial_t^k\left(M(V)u\right)(t,\cdot)}{H^{s-2k-m}}\leq \sum_{k'+k''=k}C\Big[\norm{u}{k'',s}\norm{V}{k'+K',s_0}^{N}
+\norm{u}{k'',s_0}\|V\|_{k'+K',s_0}^{N-1}\norm{V}{k'+K',s}\Big].
\end{equation}
\end{de}

\begin{de}[{\bf Maps}]\label{smoothoperatormaps}
Let $p,N\in \NNN$, with $p\leq N$, $N\geq1$, $K,K'\in\NNN$ with $K'\leq K$
and $\rho\in \RRR$, $m\geq0$. We denote by $\Sigma\MM^{m}_{K,K',p}[r,N]$
the space of maps $(V,t,u)\to M(V,t)u$
that may be written as 
\begin{equation}\label{smoothoperator1maps}
M(V;t)u=\sum_{q=p}^{N-1}M_{q}(V,\ldots,V)u+M_{N}(V;t)u,
\end{equation}
for some $M_{q}\in \widetilde{\MM}^{m}_{q}$, $q=p,\ldots, N-1$ and $M_{N}$ belongs to 
$\MM^{m}_{K,K',N}[r]$.
Finally we set $\widetilde{\MM}_{p}:=\cup_{m\geq0}\widetilde{\MM}_{p}^{m}$,
$\MM_{K,K',p}[r]:=\cup_{m\geq0}\MM^{m}_{K,K',p}[r]$ and $\Sigma\MM_{K,K',p}[r,N]:=\cup_{m\geq0}\Sigma\MM^{m}_{K,K',p}[r]$.
\end{de}

\begin{de}[{\bf Autonomous maps}]\label{autsmoothopmaps}
We define, with the notation of Definition \ref{nonomoopmaps}, the class of autonomous 
non-homogeneous smoothing operator $\MM^{m}_{K,0,N}[r,{\rm aut}]$
as the subspace of $\MM^{m}_{K,0,N}[r]$ made of those maps 
$(U,V)\to M(U)V$ satisfying estimates \eqref{porto20}
with $K'=0$, the time dependence being only through $U=U(t)$.
In the same 
way, we denote by $\Sigma\MM^{m}_{K,0,p}[r,N,{\rm aut}]$ the space of maps 
$(U,V)\to M(U,V)$ of the form \eqref{smoothoperator1}
with $K'=0$ and where the last term belongs to $\MM^{m}_{K,0,N}[r,{\rm aut}]$.
\end{de}

\begin{rmk}
We remark that if $M$ is in $\widetilde{\MM}^{m}_{p}$, $p\geq N$, then
$(V,U)\to M(V,\ldots,V)U$ is in $\MM^{m}_{K,0,N}[r,{\rm aut}]$.
For further details we refer to 
the remark after Definition $2.2.5$ in \cite{maxdelort}.
\end{rmk}

\subsection{Spaces of Symbols}\label{subsec:Simboli}
We give the definition of a class of multilinear symbols.
\begin{de}[{\bf $p$-homogeneous symbols}]\label{pomosimb}
Let $m\in \RRR$, $p\in \NNN$. We denote by $\widetilde{\Gamma}_{p}^{m}$ the space of symmetric $p$-linear maps
from $(C^{\infty}(\TTT;\CCC^{2}))^{p}$ to the space of $C^{\infty}$ functions in $(x,\x)\in \TTT\times\RRR$
\[
\mathcal{U}\to ((x,\x)\to a(\mathcal{U};x,\x))
\]
satisfying the following. There is $\mu>0$ and 
for any $\al,\be\in \NNN$ there is $C>0$ such that
\begin{equation}\label{pomosimbo1}
|\del_{x}^{\al}\del_{\x}^{\be}a(\Pi_{\vec{n}}\mathcal{U};x,\x)|\leq C\langle \vec{n}\rangle^{\mu+\al}\langle\x\rangle^{m-\be}
\prod_{j=1}^{p}\|\Pi_{n_j}U_{j}\|_{L^{2}},
\end{equation}
for any $\mathcal{U}=(U_1,\ldots, U_p)$ in $(C^{\infty}(\TTT;\CCC^{2}))^{p}$,
and $\vec{n}=(n_1,\ldots,n_p)\in \NNN^{p}$, 
where $\langle \vec{n}\rangle:=\sqrt{1+|n_1|^{2}+\ldots+|n_{p}|^{2}}$.
Moreover we assume that, if for some $(n_0,\ldots,n_{p})\in \NNN^{p+1}$,
\begin{equation}\label{pomosimbo2}
\Pi_{n_0}a(\Pi_{n_1}U_1,\ldots, \Pi_{n_p}U_{p};\cdot)\neq0,
\end{equation}
then there exists a choice of signs $\s_0,\ldots,\s_p\in\{-1,1\}$ such that $\sum_{j=0}^{p}\s_j n_j=0$.
For $p=0$ we denote by $\widetilde{\Gamma}_{0}^{0}$ the space of constant coefficients symbols
$\x\mapsto a(\x)$ which satisfy the \eqref{pomosimbo1} with $\al=0$
and the r.h.s. replaced by $C\langle \x\rangle^{m-\be}$.
\end{de}

\begin{rmk}
In the sequel we shall consider functions $\mathcal{U}=(U_1,\ldots, U_{p})$ which depends also on time $t$,
so that the above definition are functions of $(t,x,\x)$ that we denote by $a(\mathcal{U};t,x,\x)$.
\end{rmk}

\begin{rmk}\label{prodottodisimboli}
One can easily note that, if $a\in \widetilde{\Gamma}_{p}^{m}$ and $b\in \widetilde{\Gamma}_{q}^{m'}$
then $ab\in \widetilde{\Gamma}_{p+q}^{m+m'}$, and $\del_{x}a\in \widetilde{\Gamma}_{p}^{m}$ while
$\del_{\x}a\in \widetilde{\Gamma}_{p}^{m-1}$.
\end{rmk}

\begin{rmk}\label{simboloPot}
We have that the function 
$\mathtt{p}(\x):=\sum_{k=1}^{M}\frac{m_{k}}{\langle \x\rangle^{2k+1}}$, 
$\x\in \RRR$,
%
belongs to the class $\widetilde{\Gamma}^{-3}_{0}$.
\end{rmk}

We shall need also a class of non-homogeneous nonlinear symbols.

\begin{de}[\bf{Non-homogeneous Symbols}]\label{nonomorest}
Let $m\in\R$, $p\in \NNN$, $p\geq 1$, $K'\leq K$ in $\N$, $r>0$. We denote by $\Gamma^m_{K,K',p}[r]$ the space of functions $(U;t,x,\xi)\mapsto a(U;t,x,\xi)$, defined for $U\in B_{\s_0}^K(I,r)$, for some large enough $\s_0$, with complex values such that for any $0\leq k\leq K-K'$, any $\s\geq \s_0$, there are $C>0$, $0<r(\s)<r$ and for any $U\in B_{\s_0}^K(I,r(\s))\cap C^{k+K'}_{*}(I,{{H}}^{\s}(\TTT;\CCC^{2}))$ and any $\alpha, \beta \in\N$, with $\alpha\leq \s-\s_0$
\begin{equation}\label{simbo}
\asso{\partial_t^k\partial_x^{\alpha}\partial_{\xi}^{\beta}a(U;t,x,\xi)}\leq C\langle\xi\rangle^{m-\beta} 
\|U\|^{p-1}_{k+K',\s_0}
\norm{U}{k+K',\s}.
\end{equation}
\end{de}

\begin{rmk}\label{prodottodisimboli2}
We note that if $a\in \Gamma_{K,K',p}^{m}[r]$ with $K'+1\leq K$, then $\del_{t}a\in \Gamma^{m}_{K,K'+1,p}[r]$.
Moreover if $a\in \Gamma_{K,K',p}^{m}[r]$
 then $\del_x a\in \Gamma_{K,K',p}^{m}[r]$ and $\del_{\x}a\in \Gamma_{K,K',p}^{m-1}[r]$. Finally
 if $a\in \Gamma_{K,K',p}^{m}[r]$ and $b\in \Gamma_{K,K',q}^{m'}[r]$ then $ab\in \Gamma_{K,K',p+q}^{m+m'}[r]$.
\end{rmk}
The following is a subclass of the class defined in \ref{nonomorest} made of those symbols which depend on the variable $t$ only through the function $U$.
\begin{de}[{\bf Autonomous non-homogeneous Symbols}]\label{autnonomosimbo}
We denote by $\Gamma_{K,0,p}^{m}[r,{\rm aut}]$ the
subspace of $\Gamma^{m}_{K,0,p}[r]$ made of the non-homogeneous symbols $(U,x,\x)\to a(U;x,\x)$
that satisfy estimate \eqref{simbo} 
  with $K'=0$, the time dependence being only through $U=U(t)$.
\end{de}

\begin{rmk}\label{nonomoOMOSIMB}
A symbol $a(\mathcal{U};\cdot)$ of $\widetilde{\Gamma}_{p}^{m}$ defines,
by restriction to the diagonal, the symbol $a(U,\ldots,U;\cdot)$ for $\Gamma_{K,0,p}^{m}[r,{\rm aut}]$ for any $r>0$.
For further details we refer the reader to the first remark after Definition $2.1.3$ in \cite{maxdelort}.

\end{rmk}
The following is the general class of symbols we shall deal with.
\begin{de}[{\bf Symbols}]\label{simbototali}
Let $m\in \RRR$, $p\in \NNN$, $K,K'\in \NNN$ with $K'\leq K$, $r>0$
and $N\in \NNN$ with $p\leq N$. One denotes by $\Sigma\Gamma^{m}_{K,K',p}[r,N]$
the space of functions 
$(U,t,x,\x)\to a(U;t,x,\x)$ such that there are homogeneous symbols $a_{q}\in \widetilde{\Gamma}_{q}^{m}$
for $q=p,\ldots, N-1$ and a non-homogeneous symbol $a_{N}\in \Gamma^{m}_{K,K',N}[r]$
such that
\begin{equation}\label{simbotot1}
a(U;t,x,\x)=\sum_{q=p}^{N-1}a_{q}(U,\ldots,U;x,\x)+a_{N}(U;t,x,\x).
\end{equation}
We set $\Sigma\Gamma^{-\infty}_{K,K',p}[r,N]=\cap_{m\in \RRR}\Sigma\Gamma^{m}_{K,K',p}[r,N]$.

We define the subclasses of autonomous symbols $\Sigma\Gamma^{m}_{K,K',p}[r,N,{\rm aut}]$
by \eqref{simbotot1} where $a_N$ is in the class $\Gamma_{K,0,N}^{m}[r,{\rm aut}]$
of Definition \ref{autnonomosimbo}. Finally we 
set $\Sigma\Gamma^{-\infty}_{K,K',p}[r,N,{\rm aut}]=\cap_{m\in \RRR}\Sigma\Gamma^{m}_{K,K',p}[r,N,{\rm aut}]$.
\end{de}

We also introduce the following class of ``functions'', i.e. those bounded symbols which are independent of the variable $\xi$.
\begin{de}[{\bf Functions}]\label{apeape} Fix $N\in \NNN$, 
$p\in \NNN$ with $p\leq N$,
 $K,K'\in \NNN$ with $K'\leq K$, $r>0$.
We denote by $\widetilde{\calF}_{p}$ (resp. $\calF_{K,K',p}[r]$, resp. $\calF_{K,K',p}[r,{\rm aut}]$,
resp. $\Sigma\calF_{p}^{q}[r,N]$, resp. $\Sigma\calF_{K,K',p}[r,N,{\rm aut}]$)
the subspace of $\widetilde{\Gamma}^{0}_{p}$ (resp. $\Gamma^{0}_{p}[r]$, resp. 
$\Gamma^{0}_{p}[r,{\rm aut}]$, 
resp. $\Sigma\Gamma^{0,q}_{p}[r,N]$, resp. $\Sigma\Gamma^{0}_{p}[r,N,{\rm aut}]$)
made of those symbols which are independent of $\x$.
\end{de}

\subsection{Quantization of symbols}

Given a smooth symbol $(x,\x) \to a(x,\x)$,
we define, for any $\s\in [0,1]$,  the quantization of the symbol $a$ as the operator 
acting on functions $u$ as 
\begin{equation}\label{operatore}
{\rm Op}_{\s}(a(x,\x))u=\frac{1}{2\pi}\int_{\RRR\times\RRR}e^{\ii(x-y)\x}a(\s x+(1-\s)y,\x)u(y)dy d\x.
\end{equation}
This definition is meaningful in particular if $u\in C^{\infty}(\TTT)$ (identifying $u$ to a $2\pi$-periodic function). By decomposing 
$u$ in Fourier series  as $u=\sum_{j\in\ZZZ}\hat{u}(j)(1/\sqrt{2\pi})e^{\ii jx}$, we may calculate the oscillatory integral in \eqref{operatore} obtaining
\begin{equation}\label{bambola}
{\rm Op}_{\s}(a)u:=\frac{1}{\sqrt{2\pi}}\sum_{k\in \ZZZ}\left(\sum_{j\in\ZZZ}\hat{a}\big(k-j,(1-\s)k+\s j\big)\hat{u}(j)\right)\frac{e^{\ii k x}}{\sqrt{2\pi}}, \quad \forall\;\; \s\in[0,1],
\end{equation}
where $\hat{a}(k,\xi)$ is the $k^{th}-$Fourier coefficient of the $2\pi-$periodic function $x\mapsto a(x,\xi)$.
For convenience in the paper we shall use two particular quantizations:

\vspace{0.5em}
\noindent
{\bf Standard quantization.}
We define the standard quantization by specifying formula \eqref{bambola} for $\s=1$:
\begin{equation}\label{bambola2}
{\rm Op}(a)u:={\rm Op}_{1}(a)u=\frac{1}{\sqrt{2\pi}}\sum_{k\in \ZZZ}\left(\sum_{j\in\ZZZ}\hat{a}\big(k-j, j\big)\hat{u}(j)\right)\frac{e^{\ii k x}}{\sqrt{2\pi}};
\end{equation}

\vspace{0.5em}
\noindent
{\bf Weyl quantization.}
We define the Weyl quantization by specifying formula \eqref{bambola} for $\s=\frac{1}{2}$:
\begin{equation}\label{bambola202}
{\rm Op}^{W}(a)u:={\rm Op}_{\frac{1}{2}}(a)u=\frac{1}{\sqrt{2\pi}}\sum_{k\in \ZZZ}
\left(\sum_{j\in\ZZZ}\hat{a}\big(k-j, \frac{k+j}{2}\big)\hat{u}(j)\right)\frac{e^{\ii k x}}{\sqrt{2\pi}}.
\end{equation}

Moreover
the above formulas allow to transform the symbols between different quantizations, in particular we have
\begin{equation}\label{bambola5}
{\rm Op}(a)={\rm Op}^{W}(b), \qquad {\rm where} \quad \hat{b}(j,\x)=\hat{a}(j,\x-\frac{j}{2}).
\end{equation}

We want to define a \emph{para-differential} quantization.
First we give the following definition.
\begin{de}[{\bf Admissible cut-off functions}]\label{cutoff1}
Fix $p\in \NNN$ with $p\geq1$.
We say that  $\chi_{p}\in C^{\infty}(\RRR^{p}\times \RRR;\RRR)$ and $\chi\in C^{\infty}(\R\times\R;\R)$ are  admissible cut-off functions if  they are even with respect to each of their arguments and there exists $\delta>0$ such that
\begin{equation*}
\begin{aligned}
&{\rm{supp}}\, \chi_{p} \subset\set{(\xi',\xi)\in\R^{p}\times\R; |\xi'|\leq\delta \langle\xi\rangle},\qquad \chi_p (\xi',\xi)\equiv 1\,\,\, \rm{ for } \,\,\, |\xi'|\leq\frac{\delta}{2} \langle\xi\rangle,\\
&\rm{supp}\, \chi \subset\set{(\xi',\xi)\in\R\times\R; |\xi'|\leq\delta \langle\xi\rangle},\qquad \chi(\xi',\xi) \equiv 1\,\,\, \rm{ for } \,\,\, |\xi'|\leq\frac{\delta}{2} \langle\xi\rangle.
\end{aligned}
\end{equation*}
We assume moreover that for any derivation indices $\alpha$ and $\beta$
\begin{equation*}
\begin{aligned}
&|\partial_{\xi}^{\alpha}\partial_{\xi'}^{\beta}\chi_p(\xi',\xi)|\leq C_{\alpha,\beta}\langle\xi\rangle^{-\alpha-|\beta|},\,\,\forall \alpha\in \NNN, \,\beta\in\NNN^{p},\\
&|\partial_{\xi}^{\alpha}\partial_{\xi'}^{\beta}\chi(\xi',\xi)|\leq C_{\alpha,\beta}\langle\xi\rangle^{-\alpha-\beta},\,\,\forall \alpha, \,\beta\in\NNN.
\end{aligned}
\end{equation*}
\end{de}
An example of function satisfying the condition above, and that will be extensively used in the rest of the paper, is $\chi(\xi',\xi):=\widetilde{\chi}(\xi'/\langle\xi\rangle)$, where $\widetilde{\chi}$ is a function in $C_0^{\infty}(\RRR;\RRR)$  having a small enough support and equal to one in a neighborhood of zero.
For any $a\in C^{\infty}(\TTT)$ we shall use the following notation
\begin{equation}\label{pseudoD}
(\chi(D)a)(x)=\sum_{j\in\ZZZ}\chi(j)\Pi_{j}{a}.
\end{equation}

\begin{de}[{\bf The  Bony quantization}]\label{quantizationtotale}
Let $\chi$ be an admissible cut-off function according to Definition \ref{cutoff1}.
If a is a symbol in $\widetilde{\Gamma}^{m}_{p}$ and $b$ is in $\Gamma^{m}_{K,K',p}[r]$,
we set, using notation \eqref{ptupla}, 
\begin{equation}\label{simbotroncati}
\begin{aligned}
a_{\chi}(\mathcal{U};x,\x)&=\sum_{\vec{n}\in \NNN^{p}}\chi_{p}\left(\vec{n},\x\right)a(\Pi_{\vec{n}}\mathcal{U};x,\x),\qquad
b_{\chi}(U;t,x,\x)=\frac{1}{2\pi}\int_{\TTT}  \chi\left(\h,\x\right)\hat{b}(U;t,\h,\x)e^{\ii \h x}d \h.
\end{aligned}
\end{equation}
We define the  Bony quantization as
\begin{equation}\label{simbotroncati3}
\begin{aligned}
\bony(a(\mathcal{U};\cdot))&={\rm Op}(a_{\chi}(\mathcal{U};\cdot)),\qquad
\bony(b(U;t,\cdot))={\rm Op}(b_{\chi}(U;t,\cdot)).
\end{aligned}
\end{equation}
and the Bony-Weyl quantization as
\begin{equation}\label{simbotroncati2}
\begin{aligned}
\bonyw(a(\mathcal{U};\cdot))&=\weyl(a_{\chi}(\mathcal{U};\cdot)),\qquad 
\bonyw(b(U;t,\cdot))=\weyl(b_{\chi}(U;t,\cdot)).
\end{aligned}
\end{equation}
Finally, if $a$ is a symbol in the class $\Sigma\Gamma^{m}_{K,K',p}[r,N]$, that we decompose as in \eqref{simbotot1},
we define its Bony quantization as 
\begin{equation}\label{simbotroncati4}
\bony(a(U;t,\cdot))=\sum_{q=p}^{N-1}\bony(a_{q}(U,\ldots,U;\cdot))+\bony(a_{N}(U;t,\cdot)),
\end{equation}
and its Bony-Weyl quantization as
\begin{equation}\label{simbotroncati5}
\bonyw(a(U;t,\cdot))=\sum_{q=p}^{N-1}\bonyw(a_{q}(U,\ldots,U;\cdot))+\bonyw(a_{N}(U;t,\cdot)).
\end{equation}
For symbols belonging to the autonomous subclass
$\Sigma\Gamma_{K,0,p}^{m}[r,N,{\rm aut}]$
we shall
not write the time dependence in  \eqref{simbotroncati4} and  \eqref{simbotroncati5}.

\end{de}

\begin{rmk}
Let $a\in \Sigma\Gamma^{m}_{K,K',p}[r,N]$.
We note that
\begin{equation}\label{barratooperatore}
\ol{\bony(a(U;t,x,\x)[v]}=\bony(\ol{a^{\vee}(U;t,x,\x)})[\bar{v}], \quad
\ol{\bonyw(a(U;t,x,\x)[v]}=\bonyw(\ol{a^{\vee}(U;t,x,\x)})[\bar{v}],
\end{equation}
where 
\begin{equation}\label{barratoagg}
a^{\vee}(U;t,x,\x):=a(U;t,x,-\x).
\end{equation}
Moreover if we define the operator $A(U,t)[\cdot]:=\bonyw(a(U;t,x,\x))[\cdot]$
we have that $A^{*}(U,t)$, its adjoint operator w.r.t. the $L^{2}(\TTT;\CCC)$ scalar product,
can be written as
\begin{equation}\label{aggiunto}
A^{*}(U,t)[v]=\bonyw\Big(\ol{a(U;t,x,\x)}\Big)[v].
\end{equation}
\end{rmk}

\begin{rmk}\label{simbolopot2}
Recalling Remark \ref{simboloPot}
we define the symbol $\mathtt{l}(\x):=(\ii\x)^{2}+\mathtt{p}(\x)$
which belongs to $\widetilde{\Gamma}^{2}_0$. Moreover we
note that 
the operator $\lambda$ defined in \eqref{NLS1000} can be written as 
$\la[\cdot]={\rm Op}(\mathtt{l}(\x))[\cdot]$.
\end{rmk}

\begin{rmk}\label{aggiungoaggiunto}
By formula \eqref{aggiunto} one has that a para-differential operator $\bonyw(a(U;t,x,\x))[\cdot]$
is self-adjoint, w.r.t. the $L^{2}(\TTT;\CCC)$ scalar product, if and only if the symbol $a(U;t,x,\x)$ is real valued
for any $x\in \TTT$, $\x\in \RRR$. 
\end{rmk}

\begin{prop}[{\bf Action of para-differential operator}]\label{azionepara}
One has the following.

$(i)$ Let $m\in \RRR$, $p\in \NNN$. There is $\s>0$ such that for any symbol
 $a\in \widetilde{\Gamma}_{p}^{m}$, the map
 \begin{equation}\label{azione1}
 (U_{1},\ldots,U_{p+1})\to \bonyw(a(U_1,\ldots,U_{p};\cdot))U_{p+1},
 \end{equation}
extends, for any $s\in \RRR$, as a continuous $(p+1)$-linear map
$\big(H^{\s}(\TTT;\CCC^{2})\big)^{p}\times H^{s}(\TTT;\CCC)\to H^{s-m}(\TTT;\CCC)$.
Moreover, there is a constant $C>0$, depending only on $s$ and on \eqref{pomosimbo1} with $\al=\be=0$,
such that
\begin{equation}\label{azione3}
\|\bonyw(a(\mathcal{U};\cdot))U_{p+1}\|_{H^{s-m}}\leq C\prod_{j=1}^{p}\|U_{j}\|_{H^{\s}}\|U_{p+1}\|_{H^{s}},
\end{equation}
where $\mathcal{U}=(U_1,\ldots,U_{p})$. In the case that $p=0$ the r.h.s. of \eqref{azione3} is replaced by $C\|U_{p+1}\|_{H^{s}}$. 
Finally, if for some $(n_0,\ldots,n_{p+1})\in \NNN^{p+2}$,
\begin{equation}\label{azione4}
\Pi_{n_0}\bonyw(a(\Pi_{\vec{n}}\mathcal{U};\cdot))\Pi_{n_{p+1}}U_{p+1}\neq0,
\end{equation}
with $\vec{n}=(n_1,\ldots,n_{p})\in \NNN^{p}$, then there is a choice of signs $\s_{j}\in\{-1,1\}$, $j=0,\ldots,p+1$, such that $\sum_{j=0}^{p+1}\s_{j} n_{j}=0$ and the indices satisfy
\begin{equation}\label{azione5}
n_0\sim n_{p+1}, \quad n_{j}\leq C\delta n_0, \quad n_{j}\leq C\delta n_{p+1}, \quad j=1,\dots,p.
\end{equation}

$(ii)$ Let $r>0$, $m\in \RRR$, $p\in \NNN$, $p\geq1$, $K'\leq K\in \NNN$, $a\in \Gamma^{m}_{K,K',p}[r]$.
There is $\s>0$ such that for any $U\in B_{\s}^{K}(I,r)$, the operator
$\bonyw(a(U;t,\cdot))$ extends, for any $s\in \RRR$, as a bounded linear operator 
 \begin{equation}\label{azione6}
 C_{*}^{K-K'}(I,H^{s}(\TTT;\CCC)) \to  C_{*}^{K-K'}(I,H^{s-m}(\TTT;\CCC)).
 \end{equation}
Moreover, there is a constant $C>0$, depending  only on $s,r$ and \eqref{simbo}
with $0\leq \al\leq 2$, $\be=0$, such that, for any $t\in I$, any $0\leq k\leq K-K'$,
\begin{equation}\label{azione7}
\|\bonyw(\del_{t}^{k}a(U;t,\cdot))\|_{\mathcal{L}(H^{s},H^{s-m})}\leq C\|U\|_{k+K',\s}^{p},
\end{equation}
so that
\begin{equation}\label{azione8}
\|\bonyw(a(U;t,\cdot))V(t)\|_{K-K',s-m}\leq C\|U\|^{p}_{K,\s}\|V\|_{K-K',s}.
\end{equation}

\end{prop}

\begin{proof}
See Proposition $2.2.4$ in \cite{maxdelort}.
\end{proof}

\begin{rmk}\label{inclusionifacili}
We have the following  inclusions.

$\bullet$
Let $a\in \Sigma\Gamma^{m}_{K,K',p}[r,N]$ for $p\geq1$.
By Proposition \ref{azionepara} we have that
the map $(V,U)\to \bonyw(a(V;t,\cdot))U$ defined by \eqref{simbotroncati5}
is in $\Sigma\MM^{m'}_{K,K',p}[r,N]$ for some $m'\geq m$.

$\bullet$
If $a\in \Sigma\Gamma^{m}_{K,K',p}[r,N]$ with $m\leq 0$ and $p\geq1$, 
then the map $(V,U)\to \bonyw(a(V;t,\cdot))U$
is in $\Sigma\RR^{m}_{K,K',p}[r,N]$.

$\bullet$Any smoothing operator $R\in \Sigma\RR^{-\rho}_{K,K',p}[r,N]$
defines an element of $\Sigma\MM^{m}_{K,K',p}[r,N]$ for some $m\geq0$.
\end{rmk}

In the following we shall deal with operators defined on the product space
$H^{s}\times H^{s}$. 

\begin{rmk}\label{equivalenzacutoff}
From Proposition  \eqref{azionepara} we deduce that the Bony-Weyl quantization  of a symbol is unique up to smoothing remainders. More precisely 
consider two admissible cut off functions $\chi^{(1)}_p$ and $\chi^{(2)}_p$ according to Def. \ref{cutoff1} with $\delta_1>0$ and $\delta_2>0$.  Define $\chi_p:=\chi^{(1)}_p-\chi^{(2)}_p$ and for $a$  in $\widetilde{\Gamma}^m_p$ set
\begin{equation}\label{felixotto}
R(\mathcal{U}):=\weyl\Big(\sum_{n\in\N^p}\chi_p(n,\xi)a(\Pi_n\mathcal{U};\cdot)\Big).
\end{equation}
Then, by applying \eqref{azione3} with $s=m$, we get
\begin{equation*}
\norm{\Pi_{n_0}R(\Pi_n\mathcal{U})\Pi_{n_{p+1}}U_{p+1}}{L^2}\leq Cn_1^{\sigma}\cdots n_p^{\sigma}n_{p+1}^{m}\prod_{j=1}^{p+1}\norm{\Pi_{n_j}U_j}{L^2}. 
\end{equation*}
The l.h.s. of the equation above is non zero only if $\delta_1\jap{n_{p+1}}\leq|n|\leq\delta_2\jap{n_{p+1}}$. As a consequence we deduce the equivalence $\max_2(n_1,\ldots, n_{p+1})\sim\max(n_1,\ldots, n_{p+1})$ and hence the operator $R$ belongs to $\widetilde{\mathcal{R}}^{-\rho}_p$.

A similar statement holds for the non homogeneous case. 
\end{rmk}

We have the following definition.
\begin{de}[{\bf Matrices of operators}]\label{matrixmatrxi}
Let $\rho,m\in \RRR$, $\rho\geq0$, $K'\leq K\in \NNN$, $r>0$, $N\in\NNN$, $p\in \NNN$ with $p\geq1$.
We denote by 
$ \Sigma\RR^{-\rho}_{K,K',p}[r,N]\otimes\MM_2(\CCC)$
the space of $2\times 2$ matrices whose entries are smoothing operators in
the class $\Sigma\RR^{-\rho}_{K,K',p}[r,N]$.
 Analogously we denote by  $\Sigma\mathcal{M}_{K,K',p}^{m}[r,N]\otimes\MM_2(\CCC)$
 the space of $2\times 2$ matrices whose entries are maps in
the class $\Sigma\MM^{m}_{K,K',p}[r,N]$.
We also set $\Sigma\mathcal{M}_{K,K',p}[r,N]\otimes\MM_2(\CCC)=\cup_{m\in \RRR}
\Sigma\mathcal{M}_{K,K',p}^{m}[r,N]\otimes\MM_2(\CCC)$.

\end{de}

\begin{de}[{\bf Matrices of symbols}]\label{matrixmatrixsimbo}
Let $m\in\RRR$, $K'\leq K\in \NNN$, $p,N\in \NNN$.
We denote by $\Sigma\Gamma^{m}_{K,K',p}[r,N]\otimes \MM_{2}(\CCC)$ the space $2\times 2$
matrices whose entries are symbols in the class $\Sigma\Gamma^{m}_{K,K',p}[r,N]$.
\end{de}

We have the following result.

\begin{lemma}
Let $\rho,m\in \RRR$, $\rho\geq0$, $m\geq 0$, $K'\leq K\in \NNN$, $r>0$, $N\in\NNN$, $p\in\NNN$, $p\geq1$
and consider $R\in \Sigma\RR^{-\rho}_{K,K',p}[r,N]\otimes\MM_{2}(\CCC)$, $M\in \Sigma\MM^{m}_{K,K',p}[r,N]\otimes\MM_{2}(\CCC)$ 
and $A\in \Sigma\Gamma^{m}_{K,K',p}[r,N]\otimes\MM_{2}(\CCC)$.
There is $\s>0$ such that
\begin{equation}\label{special1}
R : B_{s}^{K}(I, r)\times C^{K-K'}_{*}(I,H^{s}(\TTT;\CCC^{2}))\to
C^{K-K'}_{*}(I,H^{s+\rho}(\TTT;\CCC^{2}));
\end{equation}
\begin{equation}\label{special3}
M : B_{s}^{K}(I,r)\times C^{K-K'}_{*}(I,H^{s}(\TTT;\CCC^{2}))\to
C^{K-K'}_{*}(I,H^{s-m}(\TTT;\CCC^{2})),
\end{equation}
and 
\begin{equation}\label{special2}
\bonyw(A(U;t,\cdot)) : B_{\s}^{K}(I,r)\times C^{K-K'}_{*}(I,H^{s}(\TTT;\CCC^{2}))\to
C^{K-K'}_{*}(I,H^{s-m}(\TTT;\CCC^{2})).
\end{equation}
\end{lemma}
\begin{proof}
The \eqref{special1} follows by Definition \ref{smoothoperator} (see bound \eqref{porto20}).
The \eqref{special3} follows by Definition \ref{smoothoperatormaps} (see bound \eqref{porto20maps}).
The \eqref{special2} follows by Proposition \ref{azionepara}.
\end{proof}

\subsection{Symbolic calculus and Compositions theorems }\label{subsec:COmpo}

We define the following differential operator
\begin{equation}\label{cancello}
\s(D_{x},D_{\x},D_{y},D_{\h})=D_{\x}D_{y}-D_{x}D_{\h},
\end{equation} 
where $D_{x}:=\frac{1}{\ii}\del_{x}$ and $D_{\x},D_{y},D_{\h}$ are similarly defined. 

Let $K'\leq K, \rho,p,q$ be in $\NNN$, $m,m'\in \RRR$, $r>0$ and consider 
$a\in \widetilde{\Gamma}^{m}_{p}$ and $b\in \widetilde{\Gamma}^{m'}_{q}$. 
Set
\begin{equation}\label{multimulti}
\calU:=(\calU',\calU''),\quad \calU':=(U_{1},\ldots, U_{p}), \quad \calU{''}:=(U_{p+1},\ldots, U_{p+1}), \quad U_{j}\in H^{s}(\TTT;\CCC^{2}),\quad j=1,\ldots,p+q.
\end{equation}
We define the asymptotic expansion (up to order $\rho$) of the composition symbol as follows:
\begin{equation}\label{espansione1}
(a\# b)_{\rho}(\calU;x,\x):=\sum_{k=0}^{\rho}\frac{1}{k!}\left(
\frac{\ii}{2}\s(D_{x},D_{\x},D_{y},D_{\h})\right)^{k}\Big[a(\calU';x,\x)b(\calU'';y,\h)\Big]_{|_{\substack{x=y\\\x=\h}}}
\end{equation}
modulo symbols in $\widetilde{\Gamma}_{p+q}^{m+m'-\rho}$.

Consider $a\in \Gamma_{K,K',p}^{m}[r]$ and $b\in \Gamma^{m'}_{K,K',q}[r]$. For $U$ in $B_{\s}^{K}(I,r)$ 
we define, for $\rho< \s-\s_0$,
\begin{equation}\label{espansione2}
(a\# b)_{\rho}(U;t,x,\x):=\sum_{k=0}^{\rho}\frac{1}{k!}
\left(
\frac{\ii}{2}\s(D_{x},D_{\x},D_{y},D_{\h})\right)^{k}
\Big[a(U;t,x,\x)b(U;t,y,\h)\Big]_{|_{\substack{x=y\\\x=\h}}},
\end{equation}
modulo symbols in $\Gamma^{m+m'-\rho}_{K,K',p+q}[r]$.

\begin{rmk}
By Remark \ref{prodottodisimboli} one can note that the symbol
$(a\# b)_{\rho}$ in \eqref{espansione1} belongs to the class $\widetilde{\Gamma}_{p+q}^{m+m'}$
(with the exponent $\mu$ in \eqref{pomosimbo1} large as function of $\rho$).
Similarly by Remark \ref{prodottodisimboli2} one can note that the symbol
$(a\# b)_{\rho}$ in \eqref{espansione2} belongs to the class ${\Gamma}_{p+q}^{m+m'}[r]$
(with $\s_0$ in Def. \ref{nonomorest} large as function of $\rho$).
\end{rmk}

We need a result which  
ensures that $(a\# b)_{\rho}$ is the symbol of the composition
up to smoothing remainders. 

We have the following  proposition.  

\begin{prop}[{\bf Composition of Bony-Weyl operators}]\label{teoremadicomposizione}
Let $K'\leq K, \rho,p,q$ be in $\NNN$, $m,m'\in \RRR$, $r>0$. 

$(i)$
Consider 
$a\in \widetilde{\Gamma}^{m}_{p}$ and $b\in \widetilde{\Gamma}^{m'}_{q}$. 
Then  (recalling the notation in \eqref{multimulti}) one has that
\begin{equation}\label{espansione3}
\bonyw(a(\calU;x,\x))\circ\bonyw(b(\calU'';x,\x))-\bonyw\big(
(a\# b)_{\rho}(\calU;x,\x)
\big)
\end{equation}
belongs to the class of smoothing remainder $\widetilde{\RR}^{-\rho+m+m'}_{p+q}$.

$(ii)$
Consider 
$a\in {\Gamma}^{m}_{K,K',p}[r]$ and $b\in {\Gamma}^{m'}_{K,K',q}[r]$. 
Then one has that
\begin{equation}\label{espansione4}
\bonyw(a(U;t,x,\x))\circ\bonyw(b(U;t,x,\x))-\bonyw\big(
(a\# b)_{\rho}(U;t,x,\x)
\big)
\end{equation}
belongs to the class of non-homogeneous smoothing remainders ${\RR}^{-\rho+m+m'}_{K,K',p+q}[r]$.
If $a$ and $b$ are symbols in the autonomous classes of Definition \ref{autnonomosimbo}
then $(a\# b)_{\rho}(U;t,x,\x)$ belongs to ${\Gamma}^{m+m'}_{K,K',p+q}[r,{\rm aut}]$
and \eqref{espansione4} is an autonomous   smoothing remainder in ${\RR}^{-\rho+m+m'}_{K,K',p+q}[r,{\rm aut}]$.
\end{prop}

\begin{proof}
See the proof of Proposition $2.3.2$ in \cite{maxdelort}.
\end{proof}

Consider now symbols  $a\in \Sigma\Gamma_{K,K',p}^{m}[r,N]$, $b\in \Sigma\Gamma^{m'}_{K,K',q}[r,N]$.
By definition (see Def. \ref{simbototali}) we have
\begin{equation}
\begin{aligned}
&a(U;t;x,\x)=\sum_{k=p}^{N-1}a_{k}(U,\ldots,U;x,\x)+a_{N}(U;t,x,\x), \\ 
&b(U;t;x,\x)=\sum_{k'=p}^{N-1}b_{k'}(U,\ldots,U;x,\x)+b_{N}(U;t,x,\x),\\
& a_{k}\in \widetilde{\Gamma}_{k}^{m}, \quad 
a_{N}\in \Gamma^{m}_{K,K',N}[r],
\quad  b_{k'}\in \widetilde{\Gamma}_{k'}^{m'}, \quad 
b_{N}\in \Gamma^{m'}_{K,K',N}[r].
\end{aligned}
\end{equation}
We set also
\begin{equation}\label{espansione7}
\begin{aligned}
&c_{k''}(\calU;x,\x):=\sum_{k+k'=k''}(a_{k}\# b_{k'})_{\rho}(\calU;x,\x), \quad k''=p+q,\ldots, N-1,\\
&c_{N}(U;t,x,\x):=\sum_{k+k'\geq N}(a_{k}\# b_{k'})_{\rho}(U;t,x,\x),
\end{aligned}
\end{equation}
where the factors $a_{k}$ and $b_{k'}$, for $k,k'\leq N-1$, have to be considered as elements of $\Gamma_{K,0,k}^{m}[r]$ and 
$\Gamma_{K,0,k'}^{m'}[r]$ respectively
according to Remark \ref{nonomoOMOSIMB}. 
We define the composition symbol $(a\# b)_{\rho,N}\in \Sigma\Gamma^{m+m'}_{K,K',p+q}[r,N]$ 
as
\begin{equation}\label{espansione6}
(a\# b)_{\rho,N}(U;t,x,\x):=(a\# b)_{\rho}(U;t,x,\x):=\sum_{k''=p+q}^{N-1}c_{k''}(U,\ldots,U;x,\x)+c_{N}(U;t,x,\x).
\end{equation}

The following proposition collects  the results contained in Section  $2.4$ in \cite{maxdelort}
concerning compositions between Bony-Weyl operators, smoothing remainders and maps.

\begin{prop}[{\bf Compositions}]\label{composizioniTOTALI}
Let $m,m',m''\in \RRR$, $K,K',N,p_1,p_2,p_{3},p_{4},\rho\in \NNN$ with $K'\leq K$, $p_1+p_{2}<N$, $\rho\geq0$ and $r>0$.
Let $a\in \Sigma\Gamma^{m}_{K,K',p_1}[r,N]$, $b\in \Sigma\Gamma^{m'}_{K,K',p_2}[r,N]$, $R\in\Sigma\RR^{-\rho}_{K,K',p_{3}}[r,N]$ and $M\in \Sigma\MM^{m''}_{K,K',p_{4}}[r,N]$.
Then the following holds.

$(i)$ There exists a smoothing operator $R_1$ in the class $\Sigma\RR^{-\rho}_{K,K',p_1+p_2}[r,N]$
such that
\begin{equation}\label{espansione8}
\bonyw(a(U;t,x,\x))\circ \bonyw(b(U;t,x,\x))=\bonyw\big(
(a\# b)_{\rho,N}(U;t,x,\x)
\big)+R_1(U;t)
\end{equation}

$(ii)$ One has that the compositions operators 
$R(U;t)\circ \bonyw(a(U;t,x,\x))$, $ \bonyw(a(U;t,x,\x))\circ R(U;t)$,
are smoothing operators in the class $\Sigma\RR^{-\rho+m}_{K,K',p_1+p_{3}}[r,N]$.

$(iii)$
One has that the compositions operators $R(U;t)\circ M(U;t)$, $ M(U;t)\circ R(U;t),$
are smoothing operators in the class $\Sigma\RR^{-\rho+m''}_{K,K',p_3+p_{4}}[r,N]$.

$(iv)$ Let $R_{2}(U,W;t)[\cdot]$ be a smoothing operator of $\Sigma\RR^{-\rho}_{K,K',p_3}[r,N]$
depending linearly on  $W$, i.e.
\[
R(U,W;t)[\cdot]=\sum_{q=p_3}^{N-1}R_{q}(U,\ldots, U,W)[\cdot]+R_{N}(U,W;t)[\cdot],
\]
where $R_{q}\in \widetilde{R}_{q}^{-\rho}$ and $R_{N}$ satisfies for any $0\leq k\leq K-K'$ (instead of \eqref{porto20})
the following
\begin{equation*}
\begin{aligned}
&\|\del_{t}^{k}R_{N}(U,W;t)V(t,\cdot)\|_{H^{s-2k}} \\
&\quad \leq C\sum_{k'+k''=k}\left(\|U\|^{N-1}_{k'+K',\s}\|W\|_{k'+K',\s}\|V\|_{k'',s}+
\|U\|^{N-1}_{k'+K',\s}\|W\|_{k'+K',s}\|V\|_{k'',\s}\right.\\
&\qquad\qquad\qquad  \left.+
\|U\|^{N-2}_{k'+K',\s}\|U\|_{k'+K',s}\|W\|_{k'+K',\s}\|V\|_{k'',\s}\right).
\end{aligned}
\end{equation*}
Then one has that $R(U,M(U;t)W;t)$ belongs to $\Sigma\RR^{-\rho+m''}_{K,K',p_3+p_{4}}[r,N]$.

$(v)$ Let $c$ be in $\widetilde{\Gamma}_p^{m}$, $p\in \NNN$. 
Then 
\begin{equation}\label{ultimo}
U \rightarrow c(U,\ldots,U,M(U;t)U;t,x,\x)
\end{equation}
is in $\Sigma\Gamma^{m}_{K,K',p+p_4}[r,N]$. If the symbol $c$ is independent of $\x$ (i.e.
$c$ is in $\widetilde{\calF}_p$), so is the symbol in \eqref{ultimo}
(thus it is a  function in $\Sigma\calF_{K,K',p+p_4}[r,N]$). Moreover if $c$ is a symbol in $\Gamma^{m}_{K,K',N}[r]$
then the symbol in \eqref{ultimo} is in $\Gamma^{m}_{K,K',N}[r]$.

All the statements of the proposition have their counterpart for autonomous classes.
\end{prop}

We omit the proof of the proposition above. We refer the reader to 
Propositions $2.4.1$, $2.4.2$, $2.4.3$ in Section $2.4$ of \cite{maxdelort}.

\subsection{Parity and Reversibility properties}\label{subsec:Alge}
In this Section we analyse the parity and the reversibility structure for para-differential and smoothing operators. 

Denote by $S$ the linear involution, i.e. $S^{2}=\uno$,
\begin{equation}\label{involuzione3}
S : \CCC^{2}\to \CCC^{2},\qquad S:=\sm{0}{1}{1}{0}.
\end{equation}
For any $U\in B_{\s}^{K}(I,r)$ we set $U_{S}(t):=(SU)(-t)$.

Note that $U\in L^{2}(\TTT;\CCC^{2})$ belongs to the subspace $ \gotR$ (see \eqref{Hcic}) if and only if
\begin{equation}\label{involuzione4}
(SU)(x)=\ol{U}(x).
\end{equation}
We have the following definitions.
\begin{de}\label{riassunto-mappe}
Let $p,N,K,K'\in \NNN$ with $p\leq N$, $K'\leq K$ and $r>0$, $\rho\geq0$.
 Let $M\in \Sigma\MM_{K,K',p}[r,N]\otimes\MM_{2}(\CCC)$ 
(or $\Sigma\RR^{-\rho}_{K,K',p}[r,N]\otimes\MM_{2}(\CCC)$ respectively) with
 $U$ satisfying \eqref{involuzione4}.
\begin{itemize}
\item {\bf Reality preserving maps.} We say that the map $M(U;t)$
is \emph{reality} preserving if  
\begin{equation}\label{realereale}
\ol{M(U;t)[V]}=SM(U;t)[S\ol{V}].
\end{equation}
\item {\bf Anti-reality  condition.} We say that the map $M(U;t)$
satisfies the \emph{anti-reality} condition if  
\begin{equation}\label{realereale2}
\ol{M(U;t)[V]}=-SM(U;t)[S\ol{V}].
\end{equation}
\item {\bf Reversible maps.} We say that the map $M(U;t)$
is \emph{reversible} w.r.t. the involution \eqref{involuzione3}
if one has
\begin{equation}\label{invo2}
-SM(U;-t)=M(U_{S};t)S.
\end{equation}
\item {\bf Reversibility preserving maps.} We say that the map $M(U;t)$ is \emph{reversibility preserving}
if
\begin{equation}\label{invo3}
SM(U;-t)=M(U_{S};t)S.
\end{equation}
\item {\bf Parity preserving maps.} 
We say that $M(U;t)$ is \emph{parity preserving} if 
\begin{equation}\label{parissimo}
M(U;t)\circ\tau=\tau\circ M(U;t),
\end{equation}
where $\tau$ is the map acting on functions $\tau V(x)=V(-x)$.

\item 
{\bf (R,R,P)-maps/ operators.}
We say that $M(U;t)$ is a (R,R,P)-map (resp. (R,R,P)-operator)
if
it is a reality, reversibility and parity preserving map (resp. operator).
\end{itemize}
\end{de}

\begin{rmk}\label{ProdProd}
Given a smoothing operator $R\in \Sigma\RR^{-\rho}_{K,K',p}[r,N]\otimes\mathcal{M}_2(\CCC)$ (resp. a map in  $\Sigma\MM_{K,K',p}[r,N]\otimes\mathcal{M}_2(\CCC)$)
then the following holds true. If $R(U;t)$ is reality preserving, i.e. satisfies \eqref{realereale}, then it
maps
$\hcic^{s}(\TTT;\CCC^{2})$ into $\hcic^{s+\rho}(\TTT;\CCC^{2})$
(resp. $\hcic^{s}(\TTT;\CCC^{2})$ into $\hcic^{s-m}(\TTT;\CCC^{2})$ for some $m\geq0$).
 Moreover, for any $V\in C^{K-K'}_{*}(I,H^{s}(\TTT;\CCC))$ and any smoothing remainder $Q(U;t)\in \Sigma\RR^{-\rho}_{K,K',p}[r,N]$ (resp. any map in  $\Sigma\MM_{K,K',p}[r,N]$), we set
\begin{equation}\label{barrato}
\ol{Q}(U;t)[V]:=\ol{Q(U;t)[\bar{V}]}.
\end{equation}
One can easily check that a matrix of smoothing operators $R(U;t)\in \Sigma\RR^{-\rho}_{K,K',p}[r,N]\otimes \MM_{2}(\CCC)$ (resp. a matrix of maps in $\Sigma\MM_{K,K',p}[r,N]\otimes\MM_{2}(\CCC)$)
is reality preserving according to Definition \ref{riassunto-mappe}
if and only if can be written as
\begin{equation}\label{vinello}
R(U;t)[\cdot]:=\left(\begin{matrix} R_{1}(U;t)[\cdot] & R_{2}(U;t)[\cdot] \vspace{0.4em} \\
\ol{R_{2}}(U;t)[\cdot] & \ol{R_{1}}(U;t)[\cdot] 
\end{matrix}
\right), 
\end{equation}
for suitable  smoothing operators $R_1(U;t)$ and $R_2(U;t)$ in the class $\Sigma\RR^{-\rho}_{K,K',p}[r,N]$ (resp. $\Sigma\MM_{K,K',p}[r,N]$).
\end{rmk}

We have the following lemmata.

\begin{lemma}\label{pasubio}
Let $p,N,K,K'\in \NNN$ with $p\leq N$, $K'\leq K$ and $r>0$, $\rho\geq0$. Let $M\in \Sigma\MM_{K,K',p}[r,N]$.
If $M$ is decomposed as in \eqref{smoothoperator1maps} 
as a sum
\begin{equation}\label{smoothoperator1maps1000}
M(V;t)U=\sum_{q=p}^{N-1}M_{q}(V,\ldots,V)U+M_{N}(V;t)U,
\end{equation}
in terms of homogeneous operators $M_{q}, q=p,\ldots,N - 1$, 
and if M satisfies the reversibility condition \eqref{invo2}, respectively reversibility preserving \eqref{invo3}, 
we may assume that 
$M_{q}$, $q=p,\dots, N-1$
satisfy the reversibility property
\begin{equation}\label{3.1.20}
M_{q}(SU_1,\ldots, SU_{q})S=-SM_{q}(U_1,\ldots,U_{q}),
\end{equation}
respectively the reversibility preserving  property
\begin{equation}\label{3.1.21}
M_{q}(SU_1,\ldots, SU_{q})S=SM_{q}(U_1,\ldots,U_{q}).
\end{equation}
\end{lemma}

\begin{proof}
See Lemma $3.1.5$ in \cite{maxdelort}.
\end{proof}

We gave the definitions of reality, parity and reversibility preserving and reversible map in the case of 
$M$ belonging to the class 
$\Sigma\MM_{K,K',p}[r,N]$. 
An important subclass of maps we shall use in the following have the form
\[
M(U;t)[\cdot]:=\bonyw(A(U;t,x,\xi))[\cdot], \quad A\in \Sigma\Gamma^{m}_{K,K',p}[r,N]\otimes\MM_2(\CCC),
\]
for some $m\in\RRR$. In this case we give the properties of being reality preserving, parity preserving or reversible, 
directly on the matrix of symbols $A$.

\begin{de}\label{riassunto-simbo}
Let $m\in \RRR$, $p,N,K,K'\in \NNN$ with $p\leq N$, $K'\leq K$ and $r>0$, $\rho\geq0$
and consider a matrix $A(U;t,x,\x)\in \Sigma\Gamma^{m}_{K,K',p}[r,N]\otimes \MM_{2}(\CCC)$
where $U$ satisfies \eqref{involuzione4}.
\begin{itemize}
\item {\bf Reality preserving matrices of symbols.} We say that a matrix $A(U;t,x,\x)$ 
is \emph{reality} preserving if 
\begin{equation}\label{prodottorealereale}
\ol{A(U;t,x,-\x)}=SA(U;t,x,\x)S.
\end{equation}
\item {\bf Anti-reality preserving matrices of symbols.} We say that a matrix $A(U;t,x,\x)$ 
is \emph{anti-reality} preserving if 
\begin{equation}\label{prodottorealereale2}
\ol{A(U;t,x,-\x)}=-SA(U;t,x,\x)S.
\end{equation}
\item {\bf Reversible and reversibility preserving matrices of symbols.}
 We say that $A(U;t,x,\x)$ is reversible if
\begin{equation}\label{revsimbo1}
-SA(U;-t,x,\x)=A(U_{S};t,x,\x)S.
\end{equation}
 We say that $A(U;t,x,\x)$ is reversibility preserving if
\begin{equation}\label{revsimbo2}
SA(U;-t,x,\x)=A(U_{S};t,x,\x)S.
\end{equation}
\item {\bf Parity preserving matrices of symbols.}
 We say that $A(U;t,x,\x)$ is parity preserving if
\begin{equation}\label{parisimbo1}
A(U;t,x,\x)=A(U;t,-x,-\x).
\end{equation}
\item 
{\bf (R,R,P)-matrices.}
We say that $A(U;t,x,\x)$ is a (R,R,P)-matrix
if
it is a reality, reversibility and parity preserving matrix of symbols.
\end{itemize}
\end{de}

\begin{rmk}\label{considerazioni}
Consider $A(U;t,x,\x)\in \Sigma\Gamma_{K,K',p}^{m}[r,N]\otimes \MM_{2}(\CCC)$ and $R(U;t)\in \Sigma\RR^{-\rho}_{K,K',p}[r,N]\otimes\MM_{2}(\CCC)$. 
If $A(U;t,x,\x)$ is reality preserving, i.e. satisfies \eqref{prodottorealereale}, then it has the form
\begin{equation}\label{prodotto}
A(U;t)=A(U;t,x,\x):=\left(\begin{matrix} {a}(U;t,x,\x) & {b}(U;t,x,\x)\vspace{0.4em}\\
{\ol{b(U;t,x,-\x)}} & {\ol{a(U;t,x,-\x)}}
\end{matrix}
\right).
\end{equation}
We note also the following facts:
\begin{itemize}
\item   if  $A(U;t,x,\x)$ satisfy one among the properties \eqref{revsimbo1}, \eqref{revsimbo2}, \eqref{parisimbo1}
and it is invertible, then 
$A(U;t,x,\x)^{-1}$
satisfies the same property;

\item if the matrix $A(U;t,x,\x)$  is reversibility preserving (resp. if $R(U;t)$ is reversibility preserving)
then the matrix $\ii E A(U;t,x,\x)$ (resp. the operator $\ii E R(U;t)$)
is reversible.

\item if the matrix $A(U;t,x,\x)$  is reality preserving (resp. if $R(U;t)$ is reality preserving)
then the matrix of symbols $\ii  A(U;t,x,\x)$ (resp. the operator $\ii  R(U;t)$)
is anti-reality.

\item
if $A(U;t,x,\x)$ is a reality, reversibility and parity preserving (resp. reversible, reality and parity preserving) matrix of symbols,
then the operator $\bonyw(A(U;t,x,\x))[\cdot]$
is a reality, reversibility and parity preserving  (resp. reversible, reality and parity preserving) map.

\item 
the matrix $A(U;t,x,\xi)$ is reversibility preserving if and only if its symbols verify the following
\begin{equation}\label{matrix_to_symbols}
\begin{aligned}
&\ol{b(U;-t,x,\xi)}={b(U_S;t,x,\xi)},\qquad 
\ol{a(U;-t,x,\xi)}={a(U_S;t,x,\xi)};
\end{aligned}
\end{equation}
furthermore note that  in the case that the symbols are autonomous (i.e. when the dependence of time is through the function $U(t,x)$) the conditions above reads
\end{itemize}
\begin{equation}\label{matrix_to_symbols_aut}
\begin{aligned}
&\ol{b(U;x,\xi)}={b(SU;x,\xi)},\qquad \ol{a(U;x,\xi)}={a(SU;x,\xi)};
\end{aligned}
\end{equation}

\end{rmk}

\begin{rmk}\label{formavera}
Recalling \eqref{DEFlambda} and Remark \ref{simbolopot2} we write
\begin{equation}
\Lambda[\cdot]:=\left(
\begin{matrix}
\bonyw(\mathtt{l}(\x))[\cdot] & 0\\ 0 & \bonyw(\mathtt{l}(\x))[\cdot]
\end{matrix}
\right).
\end{equation}
In particular, since the symbol $\mathtt{l}(\x)$ is real and even in $\x$ one has 
that the operator $\Lambda$ is reality, parity and reversibility preserving. 
\end{rmk}

\begin{de}\label{riassunto-simbo-omo}
Fix $\rho>0$, $m\in \RRR$, $p\in \NNN$ and 
let $A_{p}\in \widetilde{\Gamma}^{m}_{p}\otimes\MM_{2}(\CCC)$
and let $U_{j}$, with $j=1,\ldots,p$, be functions satisfying $SU_{j}=\ol{U}_{j}$ (see  \eqref{involuzione3}).
We say that $A_{p}$ is reversibility preserving if
\begin{equation}\label{revomo1000}
A_{p}(SU_1,\ldots, SU_{p};x,\x)S=SA_{p}(U_1,\ldots,U_{p};x,\x).
\end{equation}
We say that $A_{p}$ is reversible  if
\begin{equation}\label{revomo1001}
A_{p}(SU_1,\ldots, SU_{p};x,\x)S=-SA_{p}(U_1,\ldots,U_{p};x,\x).
\end{equation}
\end{de}

We have the following Lemma.
\begin{lemma}\label{lemmaNUOVO}
Let $A\in \Sigma\Gamma_{K,K',p}^{m}[r,N]\otimes\MM_{2}(\CCC)$
and write
\[
A(U;t,x,\x)=\sum_{q=p}^{N-1}A_{q}(U,\ldots,U;x,\x)+A_{N}(U;t,x,\x),
\]
with $A_{q}\in \widetilde{\Gamma}^{m}_{q}$, $q=p,\ldots,N-1$ and 
$A_{N}\in \Sigma\Gamma^{m}_{K,K',N}[r,N]\otimes\MM_{2}(\CCC)$.

\noindent
(i)
If $A_{q}$ satisfies \eqref{revomo1000} (resp. \eqref{revomo1001}), for $q=p,\ldots,N-1$ 
and $A_{N}$ satisfies \eqref{revsimbo2} (resp. \eqref{revsimbo1}), then $A$ satisfies \eqref{revsimbo2}
(resp. \eqref{revsimbo1}).

\noindent
(ii) If $A$ satisfies \eqref{revsimbo2} (resp. \eqref{revsimbo1}) 
then there are symbols $A_{q}^{'}\in \widetilde{\Gamma}^{m}_{q}$ satisfying \eqref{revomo1000} (resp. \eqref{revomo1001}),  and a matrix $A_N'(U;t,x,\xi)$ in $\Gamma^m_{K,K',N}$ satisfying \eqref{revsimbo2} (resp. \eqref{revsimbo1}), such that for any $U$
we have
\[
A(U;t,x,\x)=\sum_{q=p}^{N-1}A'_{q}(U,\ldots,U;x,\x)+A_{N}(U;t,x,\x).
\]
\end{lemma}

\begin{proof}
See Lemma ${3.1.3}$ in \cite{maxdelort}.
\end{proof}

The following Lemma is the counterpart of Lemma $3.1.6$ in \cite{maxdelort}.
\begin{lemma}\label{lemmamaxcompo}
Composition of an operator satisfying the anti-reality property  \eqref{realereale2}
(resp. the reversibility property \eqref{invo2}) 
with one or several operators satisfying the reality property \eqref{realereale} (resp. the reversibility preserving property \eqref{invo3}) 
still satisfies the anti-reality property \eqref{realereale2} (resp. reversibility property \eqref{invo2}). 
Composition of operators which are parity preserving
is as well the parity preserving. 
Composition of operators satisfying the reality property \eqref{realereale} satisfy the reality property \eqref{realereale} as well.
\end{lemma}

\begin{lemma}\label{lemmatempo}
Let $C(U;t,\cdot)\in \Sigma\Gamma^{m}_{K,K',p}[r,N]\otimes\MM_{2}(\CCC)$
for some $m\in \RRR$, $K'\leq K-1$, $0\leq p\leq N$
and assume that $U$ is a solution of an equation
\begin{equation}\label{tempo1}
\del_{t}U=\ii E\widetilde{M}(U;t)U,
\end{equation}
for some $\widetilde{M}\in \Sigma\MM_{K,1,0}[r,N]\otimes\MM_{2}(\CCC)$.
Then the following hold:

\noindent
$(i)$ the symbol
$\del_{t}C(U;t,\cdot)$ belongs to $\Sigma\Gamma^{m}_{K,K'+1,p}[r,N]\otimes \MM_{2}(\CCC)$;

\noindent
$(ii)$ if $\widetilde{M}(U;t)$ is a (R,R,P) map (see Def. \ref{riassunto-mappe})
and $C(U;t,x,\x)$ is a (R,R,P) symbol (see Def. \ref{riassunto-simbo})
 then
 the symbol
$\del_{t}C(U;t,\cdot)$ is reality preserving, parity preserving and reversible, i.e. satisfies respectively
the \eqref{prodottorealereale}, \eqref{parisimbo1} and \eqref{revsimbo1}.
\end{lemma}

\begin{proof}
Item $(i)$ follows by 
 Lemma $2.2.6$ in \cite{maxdelort}.
Let us check  item  $(ii)$.
Assume that $C(U;t,x,\x)$ is a non-homogeneous symbol in $\Gamma^{m}_{K,K',p}[r]$. By differentiating in $t$ the relation
\[
SC(U;-t,x,\x)=C(U_{S};t,x,\x)
\]
one gets that $(\del_{t}C)(U;t,x,\x)$ is reversible. Assume now that $C\in \widetilde{\Gamma}^{m}_{p}$. Since $C(U;x,\x)$ is reversibility preserving then
\begin{equation}\label{revomo2001}
C(U_{S},\ldots, U_{S};t,x,\x)S=SC(U,\ldots,U;-t,x,\x).
\end{equation}
Hence differentiating in $t$ we get 
\begin{equation}\label{revomo100100}
\begin{aligned}
\sum_{j=1}^{p}C(U_{S},\ldots,-\underbrace{(\partial_tU)_{S}}_{j-th}, \ldots, U_{S};t,x,\x)S=-\sum_{j=1}^{p}SC(U,\ldots,\underbrace{\ii E\widetilde{M} (U,t)U}_{j-th},\ldots,U;-t,x,\x).
\end{aligned}
\end{equation}
Using that $\widetilde{M}(U;t)$ is reversibility preserving we have
\[
(\partial_t U)_{S}=S(\ii E \widetilde{M}(U;\cdot)U)(-t)=-\ii ES\widetilde{M}(U;-t)U(-t)=-\ii E\widetilde{M}(U_{S};t)U_{S}(t),
\]
which implies, together with \eqref{revomo100100} the \eqref{revsimbo1} for $(\del_{t}C)(U;t,x,\x)$.
The \eqref{prodottorealereale} and 
\eqref{parisimbo1} follows by using the definitions.
\end{proof}
We prove a lemma which asserts that a (R,R,P) operator which is the sum of a para-differential operator and a smoothing remainder may be rewritten as the sum of a (R,R,P) para-differential operator and a (R,R,P) smoothing remainder.
\begin{lemma}\label{lemmabello} Fix $\rho ,r>0$, $K\geq K'>0$ in $\NNN$, $m',$ $m''$ in $\NNN$. 
Let $A(U;t,x,\xi)=\sum_{j=-m'}^{m''}A_{j}(U;t,x,\xi)$ be a matrix of symbols such that $A_{j}(U;t,x,\xi)$ is in $\Sigma\Gamma^j_{K,K',p}[r,N]\otimes\mathcal{M}_2(\CCC)$ for any $j=-m',\ldots, m''$  and $R(U;t)$ a matrix of operators in 
$\Sigma\mathcal{R}^{-\rho}_{K,K',p}[r,N]\otimes\mathcal{M}_2(\CCC)$. If the sum
$\bonyw(A(U;t,x,\xi))+R(U;t)$ is a (R,R,P) operator, then there exist (R,R,P) matrices 
$\widetilde{A}_{j}(U;t,x,\xi)$  in 
$\Sigma\Gamma^j_{K,K',p}[r,N]\otimes\mathcal{M}_2(\CCC)$ for any $j=-m',\ldots, m''$ and a (R,R,P) smoothing remainder $\widetilde{R}(U;t)$ such that the following facts hold true:
\begin{itemize}
\item[(i)] one has that 
$$
\bonyw(A(U;t,x,\xi))+R(U;t)=\bonyw(\widetilde{A}(U;t,x,\xi))+\widetilde{R}(U;t)
$$
with $\widetilde{A}(U;t,x,\xi)=\sum_{j=-m'}^{m''}\widetilde{A}_{j}(U;t,x,\xi)$;
\item[(ii)]  if one component of a matrix $A_{j}(U;t,x,\x)$ is real valued, then the 
corresponding component in the matrix $\widetilde{A}_j(U;t,x,\xi)$ is real valued;
\item[(iii)] if, for $j\geq 0$, the matrix $A_{j}(U;t,x,\x)$ has the form $B_{j}(U;t,x)(\ii\x)^{j}$,
for some $B_{j}(U;t,x)$ in the class $\Sigma\mathcal{F}_{K,K',p}[r,N]\otimes\mathcal{M}_2(\CCC) $, then 
the corresponding matrix $\widetilde{A}_{j}(U;t,x,\x)$ is equal to $\widetilde{B}_{j}(U;t,x)(\ii\x)^{j}$
where $\widetilde{B}_{j}(U;t,x)$ belongs to $\Sigma\mathcal{F}_{K,K',p}[r,N]\otimes\mathcal{M}_2(\CCC) $.
\end{itemize}
\end{lemma}
\begin{proof}
We show how to construct the reversibility preserving operator associated to the one of the hypothesis, the parity and reality preserving construction is similar.
Since the sum $\bonyw(A(U;t,x,\xi))+R(U;t)$ is a (R,R,P) operator we obtain
\begin{equation}\label{lemmabello1000}
\begin{aligned}
\bonyw(A(U;t,x,\xi))&+R(U;t)=\\
&\frac12\left(\bonyw(A(U;t,x,\xi))+R(U;t)-\bonyw(SA(U_S;-t,x,\xi))S-SR(U_S;t)S\right)
\end{aligned}
\end{equation}
therefore it is sufficient to define $\widetilde{A}(U;t,x,\xi):=\frac12(A(U;t,x,\xi)-SA(U_S;-t,x,\xi))S$ and $\frac12(\widetilde{R}(U;t):=R(U;t)-SR(U_S;t))$. Consequently each term $\widetilde{A}_{j}(U;t,x,\xi)$ may be chosen equal to
$\frac12(A_j(U;t,x,\xi)-SA_j(U_S;-t,x,\xi))$ for $j=-m',\ldots,m''$.
Items $(ii),(iii)$ can be deduced by \eqref{lemmabello1000}.
\end{proof}

\section{Paralinearization of NLS}\label{siparalin}


The main result of this section is the following.

\begin{theo}[{\bf Para-linearization of NLS}]\label{paralineariza}
Consider the system \eqref{6.666} under the Hypothesis \ref{parity}. 
For any $N\in \NNN$, $K\in \NNN$,  $r>0$ and any $\rho>0$ there exists a 
(R,R,P)-matrix 
of symbols (see Def. \ref{riassunto-simbo}) 
$A(U;t,x,\x)$ belonging to $\Sigma\Gamma_{K,0,1}^{2}[r,N,{\rm aut}]\otimes\MM_{2}(\CCC)$
and a (R,R,P)-operator (see Def. \ref{riassunto-mappe}) $R(U)[\cdot]$ belonging to 
$\Sigma\RR^{-\rho}_{K,0,1}[r,N,{\rm aut}]\otimes\MM_{2}(\CCC)$ such that
the system \eqref{6.666} can be written as
\begin{equation}\label{sistemainiziale}
\begin{aligned}
&\del_{t}U=\ii E\Big[
\Lambda U+\bonyw(A(U;t,x,\x))[U]+R(U)[U]
\Big],
\end{aligned}
\end{equation}
where $E$ and $\Lambda$
are defined respectively in \eqref{matrici} and \eqref{DEFlambda}.
Moreover $A(U;t,x,\x)$ has the form
\begin{equation}\label{formadiA}
\begin{aligned}
A(U;t,x,\x)&:=A_2(U;t,x)(\ii\xi)^2+A_1(U;t,x)(\ii\xi)+A_{0}(U;t,x),\\
A_{i}(U;t,x)&:=\left(\begin{matrix} {a}_i(U;t,x) & {b}_i(U;t,x)\vspace{0.4em}\\
{\ol{b_i(U;t,x)}} & {\ol{a_i(U;t,x)}}
\end{matrix}
\right), \\
\end{aligned}
\end{equation}
where $a_{j}(U;t,x),b_{j}(U;t,x)\in \Sigma\calF_{K,0,1}[r,N,{\rm aut}]$ for $j=0,1,2$; $a_{2}(U;t,x)$ is \emph{real} for any $x\in \TTT$.
\end{theo}

\noindent

Theorem \ref{paralineariza} is a consequence of the para-product formula of Bony which we prove below in our multilinear setting.  \\
For fixed $p\in \NNN$, $p\geq2$ for any $u_{1},\ldots,u_{p}$ in $C^{\infty}(\TTT;\C)$, 
define the map $M$  as
\begin{equation}\label{prod}
\begin{aligned}
M \; : \; (u_{1},\ldots,u_{p})\mapsto M(u_{1},\ldots, u_{p})&:= \prod_{i=1}^{p}u_{i}=\sum_{n_{1},\ldots,n_{p}\in \NNN}\prod_{i=1}^{p}
\Pi_{n_{i}}u_{i}.\\
\end{aligned}
\end{equation}

Notice that we can also write
\begin{equation}\label{espando}
M(u_{1},\ldots, u_{p})=\sum_{n_0\in \NNN}\sum_{n_1,\ldots,n_p\in \NNN}\Pi_{n_0}M(\Pi_{n_1}u_{1},\ldots, \Pi_{n_p}u_{p}).
\end{equation}
We remark that the term $\Pi_{n_0}M(\Pi_{n_1}u_{1},\ldots, \Pi_{n_p}u_{p})$ is different from zero
only if there exists a choice of signs $\s_{j}\in\{\pm 1\}$ such that
\begin{equation}\label{condautonomia}
\sum_{j=0}^{p}\s_{j}n_{j}=0,
\end{equation}
since the map $M$ is just a product of functions.

Fix $0<\delta<1$ and consider an admissible cut-off function $\chi_{p-1}: \RRR^{p-1}\times \RRR\to \RRR$ 
(see Def. \ref{cutoff1}). 
We define a new cut-off function $\Theta : \NNN^{p}\to [0,1]$ in the following way:
given any $\vec{n}:=(n_{1},\dots,n_{p})\in\NNN^{p}$ we set
\begin{equation}\label{teta}
\Theta(n_{1},\ldots,n_{p}):=1-\sum_{i=1}^{p}\chi_{p-1}^{(i)}(\vec{n}), \qquad 
\chi_{p-1}^{(i)}(\vec{n}):=\chi_{p-1}(\x',n_i),\quad \x':=(n_1,\ldots,n_{i-1},n_{i+1},\ldots,n_p).
\end{equation}
We use the following notation:  
for any $u_1, \ldots, u_{p}\in C^{\infty}(\TTT;\CCC)$ we shall write
\begin{equation}\label{nota1}
(u_{1},\ldots,\hat{u}_{i},\ldots,u_{p})=(u_{1},\ldots,u_{i-1},u_{i+1},\ldots,u_{p}),
\quad i=1,\ldots,p,
\end{equation}
similarly for any $U_1, \ldots, U_{p}\in C^{\infty}(\TTT;\CCC^{2})$ we shall write
\begin{equation}\label{nota2}
(U_{1},\ldots,\hat{U}_{i},\ldots,U_{p})=(U_{1},\ldots,U_{i-1},U_{i+1},\ldots,U_{p}).
\quad i=1,\ldots,p.
\end{equation}

Using the splitting in \eqref{teta} we  write
\begin{equation}\label{prod1}
\begin{aligned}
&M(u_{1},\ldots,u_p)=\sum_{i=1}^{p}M_{i}(u_1,\ldots,u_p)+M^{\Theta}(u_{1},\ldots,u_p),\\
&M_{i}(u_{1},\ldots,u_{p}):=A^{(i)}(u_{1},\ldots,\hat{u_{i}},\ldots,u_{p})[u_{i}]:=
\sum_{n_{1},\ldots,n_{p}\in\NNN}\chi_{p-1}^{(i)}(\vec{n})\Big(\prod_{\substack{j=1 \\ j\neq i}}^{p}\Pi_{n_{j}}u_j\Big)\Pi_{n_{i}}u_i,\\
&M^{\Theta}(u_{1},\ldots,u_p):=A^{\Theta}(u_1,\ldots,u_{p-1})[u_p]:=\sum_{n_{1},\ldots,n_{p}\in\NNN}\Theta(n_{1},\ldots,n_p)
\prod_{j=1}^{p}\Pi_{n_j}u_{j}.\\
\end{aligned}
\end{equation}

In Lemma \ref{prod2} we prove that the multilinear operator $A^{\Theta}$ in \eqref{prod1}
is a smoothing remainder, in Lemma \ref{paraprod} we show that   $A^{(i)}$ in \eqref{prod1} is a paradifferential operator acting on the function $u_i$ for any $i=1,\ldots,p$.

\begin{lemma}[{\bf Remainders}]\label{prod2} Let $A^{\Theta}$ be the operator defined  in \eqref{prod1}.
There is $Q$ in $\widetilde{\RR}^{-\rho}_{p-1}$, for any $\rho\geq0$, such that
for any $U_i\in C^{\infty}(\TTT;\CCC^{2})$, $U_{p}\in C^{\infty}(\TTT;\CCC)$, for $i=1,\ldots,p-1$, we have
\begin{equation}\label{uguale}
Q(U_1,\ldots,U_{p-1})[U_{p}]\equiv A^{\Theta}(u_1,\ldots,u_{p-1})[U_{p}],
\end{equation}
where $U_i=(u_i,z_{i})^{T}$,  $z_i\in C^{\infty}(\TTT;\CCC)$, for $i=1,\ldots,p-1$. 
\end{lemma}

\begin{proof}
Let $U_i\in C^{\infty}(\TTT;\CCC^2)$ be of the form 
$U_{i}=(u_i,z_{i})^{T}$ for $i=1,\ldots, p-1$, consider also  $U_{p}\in C^{\infty}(\TTT;\CCC)$. In order to obtain \eqref{uguale} it is enough to choose
\begin{equation}\label{uguale2}
Q(U_1,\ldots, U_{p-1})[U_p]=\sum_{n_{1},\ldots,n_{p}\in\NNN}\Theta(n_{1},\ldots,n_p)
\Big[
\left(\begin{matrix}1  \\
 0\end{matrix}\right)\cdot
\left(\begin{matrix}
\prod_{j=1}^{p-1}\Pi_{n_j}u_{j}\\
\prod_{j=1}^{p-1}\Pi_{n_{j}}z_{j}
\end{matrix}
\right)\Big]\Pi_{n_p}U_{p}.
\end{equation}
The ``autonomous'' condition in \eqref{omoresti2} for $Q$
follows from the following fact:   from \eqref{prod1} we have
\[
\Pi_{n_0}Q(\Pi_{n_1}U_{1},\ldots, \Pi_{n_{p-1}}U_{p-1})[\Pi_{n_p}U_p]=
\Pi_{n_0}M^{\Theta}(\Pi_{n_1}u_{1},\ldots, \Pi_{n_{p}}u_{p}),
\] 
which is different from zero
only if \eqref{condautonomia} holds true for a suitable choice of signs $\s_{j}\in \{\pm1\}$.
In order to prove \eqref{omoresti1} we need estimate,
for any $n_0\in \NNN$, 
the term
\begin{equation}\label{prod3}
\left\|\Pi_{n_0}Q(\Pi_{n_1}U_1,\ldots,\Pi_{n_{p-1}}U_{p-1})[\Pi_{n_p}U_p]
\right\|_{L^{2}}=
\left\|\Pi_{n_0}A^{\Theta}(\Pi_{n_1}u_1,\ldots,\Pi_{n_{p-1}}u_{p-1})[\Pi_{n_p}U_p]
\right\|_{L^{2}}
\end{equation}
for $\sum_{j=1}^{p}\s_j n_{j}=0$ for some choice of signs $\s_{j}\in \{\pm1\}$.
We note that, if there exists $i=1,\ldots,p$ such that $\chi_{p-1}^{(i)}\equiv1$, 
i.e. $\sum_{j\neq i}n_{j}\leq (\delta/2) n_i$,
then we have $\Theta(n_1,\ldots,n_{p})\equiv0$.
Hence we have the following inclusion:
\begin{equation}\label{prod4}
\left\{(n_{1},\ldots,n_{p})\in \NNN^{p} \; : \; \Theta(n_{1},\ldots,n_{p})\neq0
\right\}\subseteq \bigcap_{i=1}^{p}\left\{(n_{1},\ldots,n_{p})\in \NNN^{p}\; :\; \frac{\delta}{2} n_{i}< \sum_{j\neq i} n_{j}
\right\}.
\end{equation}
This implies that there exists constants $0<c\leq C$ such that
\begin{equation}\label{prod5}
c\max\{\langle n_1\rangle,\ldots,\langle n_{p}\rangle\}\leq {\max}_2\{\langle n_1\rangle,\ldots,\langle n_{p}\rangle\}
\leq \max\{\langle n_1\rangle,\ldots,\langle n_{p}\rangle\}.
\end{equation}
There exists a constant $K>0$, depending on $p$, such that we can bound \eqref{prod3} by
\begin{equation}\label{prod6}
 K \prod_{j=1}^{p}\|\Pi_{n_j}u_{j}\|_{L^{2}}
 \stackrel{(\ref{prod5})}{\leq} 
 K\frac{\max_2\{\langle n_1\rangle,\ldots,\langle n_{p}\rangle\}^{\mu+\rho}}{\max\{\langle n_1\rangle,\ldots,\langle n_{p}\rangle\}^{\rho}}\prod_{j=1}^{p}\|\Pi_{n_j}u_{j}\|_{L^{2}},
\end{equation}
for any $\mu\geq 0$.
This is the \eqref{omoresti1}.

\end{proof}

\begin{lemma}[{\bf Para-differential operators}]\label{paraprod}
Let $A^{(i)}$ the operators defined in \eqref{prod1}
for $i=1,\ldots,p$. There are
functions $b^{(i)}(U_1,\ldots,\hat{U}_{i},\ldots U_{p};x)$ 
belonging to $\widetilde{\mathcal{F}}_{p-1}$
such that (recalling Definition \ref{quantizationtotale})
\begin{equation}\label{prod7}
A^{(i)}(u_{1},\ldots,\hat{u}_{i},\ldots,u_{p})[u_{i}]={\rm Op}^{\BB}(b^{(i)}(U_{1},\ldots,\hat{U}_i,\ldots,U_{p};x))\left[
\left(\begin{matrix}1 \\
0 
\end{matrix}
\right)\cdot U_{i}
\right],
\end{equation}
where $U_j=(u_j,z_{j})^{T}$,  $z_j\in C^{\infty}(\TTT;\CCC)$ for $j=1,\ldots,p$.
\end{lemma}

\begin{proof}
We introduce the function 
\begin{equation}\label{prod8}
a^{(i)}(u_{1},\ldots,\hat{u}_{i},\ldots,u_{p};x):=\prod_{j\neq i}u_{j} .
\end{equation}
By \eqref{prod1} and Definition \ref{quantizationtotale} we can note that 
\begin{equation}\label{prod777}
A^{(i)}(u_{1},\ldots,\hat{u}_{i},\ldots,u_{p})[u_{i}]={\rm Op}^{\BB}(a^{(i)}(u_{1},\ldots,\hat{u}_i,\ldots,u_{p};x))[u_i].
\end{equation}

\noindent
For $i=1,\ldots,p$ we set 
\begin{equation}\label{defdiB}
b^{(i)}(U_{1},\ldots,\hat{U}_i,\ldots,U_{p};x):=\Big[
\left(\begin{matrix}1  \\
0 \end{matrix}\right)\cdot
\left(\begin{matrix}
\prod_{j\neq i}u_{j}\\
\prod_{j\neq i}z_{j}
\end{matrix}
\right)\Big]
.
\end{equation}
We show that  $b^{(i)}$ 
belongs to 
the class $ \widetilde{\mathcal{F}}_{p-1}$. 
Let us check condition \eqref{pomosimbo1}. 
By symmetry we study the case $i=p$. 
We have that
\begin{equation}\label{prod9}
\begin{aligned}
|\del_{x}^{\al}&b^{(p)}(\Pi_{n_{1}}U_{1},\ldots,\Pi_{n_{p-1}}U_{p-1};x) |
\stackrel{(\ref{defdiB}),(\ref{prod8})}{=}
|\del_{x}^{\al}a^{(p)}(\Pi_{n_{1}}u_{1},\ldots,\Pi_{n_{p-1}}u_{p-1};x) |
\stackrel{(\ref{prod8})}{\leq}\\
&
\leq 
C(\al,\be)\sum_{\substack{s_j\in\NNN, \\s_1+\ldots+s_{p-1}=\al}} 
\prod_{j=1}^{p-1}|n_{j}|^{s_{j}}|\Pi_{n_{j}}u_{j}|\leq C(\al,\be)
\max\{\langle n_{1}\rangle, \ldots, \langle n_{p-1}\rangle\}^{\al }
\prod_{j=1}^{p-1}\|\Pi_{n_{j}}u_j\|_{L^{2}}.
\end{aligned}
\end{equation}
\end{proof}

\begin{proof}[{\bf Proof of Theorem \ref{paralineariza}}] 
Let us consider a single monomial 
 of the non linearity $f$ in \eqref{funz1}, i.e. 
\[
C_{\al,\be}z_{0}^{\al_{0}}\bar{z}_{0}^{\be_{0}}z_{1}^{\al_{1}}\bar{z}_{1}^{\be_{1}}z_{2}^{\al_{2}}
\bar{z}_{2}^{\be_{2}}, \quad C_{\al,\be}\in\RRR,
\]
with 
$(\al,\be):=(\al_0,\be_0,\al_1,\be_1,\al_2,\be_2)\in \NNN^{6}$ and $\sum_{i=0}^{2}\al_i+\be_i=p$ for a fixed $2\leq p\leq \bar{q}$. Recall that the coefficients $C_{\alpha,\beta}$ of the polynomial in \eqref{funz1} are real thanks to item 3 of Hypothesis \ref{parity}. 
We start by proving that 
we can write
\begin{equation}\label{claimo}
\left(\begin{matrix}
C_{\al,\be}u^{\al_0}\bar{u}^{\be_0}u_x^{\al_1}\bar{u}_x^{\be_1}u_{xx}^{\al_2}\bar{u}_{xx}^{\be_2}\vspace{0.4em}\\
C_{\al,\be}\bar{u}^{\al_0}{u}^{\be_0}\bar{u}_x^{\al_1}{u}_x^{\be_1}\bar{u}_{xx}^{\al_2}{u}_{xx}^{\be_2}
\end{matrix}
\right)=
\bony(B^{\al,\beta}(U;x,\x))[U]+Q_1^{\al,\beta}(U)[U],
\end{equation}
where $B^{\al,\beta}(U;x,\x)$ is 
a matrix of symbols and $Q_1^{\al,\beta}(U)$ a matrix of smoothing operator.
 For any $(\al,\be)\in A_{p}$ (see \eqref{nonlinear2}) let $M$ be the multilinear operator defined in \eqref{prod} and write
 \begin{equation}\label{polypoly}
 \begin{aligned}
& C_{\al,\be}u^{\al_0}\bar{u}^{\be_0}u_x^{\al_1}\bar{u}_x^{\be_1}u_{xx}^{\al_2}\bar{u}_{xx}^{\be_2}=\\
 &\quad \quad C_{\al,\be}M(\underbrace{u,\ldots,u}_{\al_0-times},\underbrace{\bar{u},\ldots,\bar{u}}_{\be_0-times},\underbrace{u_x,\ldots,u_x}_{\al_1-times},\underbrace{\bar{u}_x,\ldots,\bar{u}_x}_{\be_1-times},
 \underbrace{u_{xx},\ldots,u_{xx}}_{\al_2-times},\underbrace{\bar{u}_{xx},\ldots,\bar{u}_{xx}}_{\be_2-times}).
\end{aligned} \end{equation}
 Lemmata \ref{prod2}, \ref{paraprod}
guarantee that there are multilinear functions $\tilde{b}^{\al,\beta}_{j}, \tilde{c}^{\al,\beta}_{j}\in \widetilde{\mathcal{F}}_{p-1}$, $j=0,1,2$
and a multilinear remainder $\tilde{Q}^{\al,\beta}\in \widetilde{\mathcal{R}}^{-\rho}_{p-1}$ such that the r.h.s. of
\eqref{polypoly} is equal to
\begin{equation}\label{mozart}
\sum_{j=0}^{2}\bony\big({b}^{\al,\beta}_{j}(U;x)(\ii\x)^{j}\big)u+\sum_{j=0}^{2}\bony\big({c}^{\al,\beta}_{j}(U;x)(\ii\x)^{j}\big)\bar{u}+Q^{\al,\beta}(U)[u],
\end{equation}
where ${b}^{\al,\beta}_{j}(U;x):=\tilde{b}^{\al,\beta}_{j}(U,\ldots,U;x)$, 
${c}^{\al,\beta}_{j}(U;x):=\tilde{c}^{\al,\beta}_{j}(U,\ldots,U;x) $
and $Q^{\al,\beta}(U)[\cdot]:=\tilde{Q}^{\al,\beta}(U,\ldots,U)[\cdot]$.
We set
\begin{equation}\label{apple20}
\begin{aligned}
B^{\al,\beta}(U;x,\x)&:=B^{\al,\beta}_{2}(U;x)(\ii\x)^{2}+B^{\al,\beta}_{1}(U;x)(\ii\x)
+B^{\al,\beta}_0(U;x),\\
B^{\al,\beta}_{j}(U;x)&:=\left(
\begin{matrix}
{b}^{\al,\beta}_{j}(U;x) & {c}^{\al,\beta}_{j}(U;x)\vspace{0.4em}\\
\ol{ {c}^{\al,\beta}_{j}(U;x)} & \ol{ {b}^{\al,\beta}_{j}(U;x)} 
\end{matrix}
\right),\quad j=0,1,2.
\end{aligned}
\end{equation}
The matrix of symbols $B^{\al,\beta}(U;x,\x)$ belongs to $\Sigma\Gamma^{2}_{K,0,1}[r,N,{\rm aut}]\otimes \mathcal{M}_2(\CCC)$,  therefore the \eqref{claimo} follows  
for some 
$Q_1^{\al,\beta}\in \Sigma\RR^{-\rho}_{K,0,1}[r,N,{\rm aut}]\otimes \mathcal{M}_2(\CCC)$ 
for any $N>0$ and $\rho>0$.
Notice that, by construction, 
(see equations \eqref{prod8}, \eqref{defdiB} in Lemma \ref{paraprod})
we have
\begin{equation}\label{claimo801}
b_{2}^{\al,\beta}(U;x)=\al_{2}C_{\al,\beta}u^{\al_0}\bar{u}^{\be_0}u_x^{\al_1}\bar{u}_x^{\be_1}u_{xx}^{\al_2-1}\bar{u}_{xx}^{\be_2}.
\end{equation}
By using equations \eqref{funz1} and \eqref{claimo}, we deduce that
\begin{equation}\label{claimo800}
\left(\begin{matrix}
f(u,u_x,u_{xx})\vspace{0.4em}\\
\ol{f(u,u_x,u_{xx})}
\end{matrix}
\right)=
\bony(B(U;x,\x))[U]+Q_1(U)[U],
\end{equation}
where $B(U;x,\x):=B_{2}(U;x)(\ii\x)^{2}+B_{1}(U;x)(\ii\x)+B_0(U;x)$ with
\[
B_{j}(U;x):=\left(
\begin{matrix}
{b}_{j}(U;x) & {c}_{j}(U;x)\vspace{0.4em}\\
\ol{ {c}_{j}(U;x)} & \ol{ {b}_{j}(U;x)} 
\end{matrix}
\right):=
\sum_{p=2}^{\bar{q}}\sum_{\al,\beta\in A_p}B_j^{\al,\beta}(U;x), \quad j=0,1,2,
\]
and  $Q_1(U)$ is in $\Sigma\RR^{-\rho}_{K,0,1}[r,N,{\rm aut}]\otimes \mathcal{M}_2(\CCC)$.
Notice that
\[
b_2(U;x)=\sum_{p=2}^{\bar{q}}\sum_{\al,\beta\in A_p}b^{\al,\beta}_2(U;x)\stackrel{\eqref{claimo801}}{=}
(\del_{u_{xx}}f)(u,u_x,u_{xx}),
\]
hence $b_{2}(U;x)$ is real thanks to item $2$ of Hypothesis \ref{parity}. 
We now pass to the Weyl quantization. By using the formula \eqref{bambola5} 
one constructs a matrix of symbols $A(U;x,\x)\in \Sigma\Gamma^{2}_{K,0,1}[r,N,{\rm aut}]\otimes\MM_{2}(\CCC)$ such that
\[
\bonyw(A(U;x,\x))[\cdot]=\bony(B(U;x,\x))[\cdot],
\]
up to smoothing remainders in $\Sigma\RR^{-\rho}_{K,0,1}[r,N,{\rm aut}]\otimes\MM_{2}(\CCC)$.
Hence we have
obtained 
\begin{equation}\label{claimo1000}
\left(\begin{matrix}
f(u,u_x,u_{xx})\vspace{0.4em}\\
\ol{f(u,u_x,u_{xx})}
\end{matrix}
\right)=
\bonyw(A(U;x,\x))[U]+R(U)[U],
\end{equation}
for some $R\in \Sigma\RR^{-\rho}_{K,0,1}[r,N,{\rm aut}]\otimes\MM_{2}(\CCC)$.
The matrix $A(U;x,\x)$ has the form \eqref{formadiA},
in particular $a_{2}(U;x)$ is real valued since $a_{2}(U;x)=b_{2}(U;x)$. 
This is a consequence of the fact that the Weyl and the standard quantizations 
coincide at the principal order (see \eqref{bambola5}).

It remains to show the reality, parity and reversibility properties of the matrices 
$A(U;x,\x)$ and $R(U)$.

Since the function $f$ satisfies Hypothesis \ref{parity}, 
we can write the l.h.s. of \eqref{claimo1000} as $M(U)[U]$ for some
(R,R,P)-map  $M\in \Sigma\MM_{K,0,1}[r,N,{\rm aut}]\otimes\MM_{2}(\CCC)$ 
(see Remark \ref{inclusionifacili}).
Therefore by Lemma \ref{lemmabello}
we may assume that both $A(U;x,\x)$ and $R(U)$ are respectively (R,R,P)-matrix of symbols
and (R,R,P)-matrix of operators.
%
Lemma \ref{lemmabello} guarantees also that 
 the new matrix $A(U;x,\x)$ has still the form \eqref{formadiA}.
\end{proof}

\section{Regularization}\label{chicane}
We  proved in Theorem \ref{paralineariza}  that for any $N\in\NNN$ and any $\rho>0$ the equation \eqref{NLS} is equivalent to the system \eqref{sistemainiziale}.

The key result of this section is the following.

\begin{theo}[{\bf Regularization}]\label{regolarizza}

Fix $N>0$, $ \rho\gg N$ and $K\gg\rho$ . There exist $s_0>0$ and $r_0>0$ such that
for any $s\geq s_0$, $0<r \leq r_0$  and 
any $U\in B^K_s(I,r)$  solution  even in $x\in \TTT$ of \eqref{sistemainiziale} the following holds.

There exist two  (R,R,P)-maps 
$\Phi(U)[\cdot],\, \Psi(U)[\cdot] :C^{K-K'}_{*\RRR}(I,\hcic^s(\TTT;\CCC^{2}))
\rightarrow C^{K-K'}_{*\RRR}(I,\hcic^s(\TTT;\CCC^{2})),$
with $K':=2\rho+4$
satisfying the following:
\begin{enumerate}
\item[(i)] there exists a constant $C$ depending on $s$, $r$ and $K$ such that
\begin{equation}\label{stime-descentTOTA}
\begin{aligned}
\norm{\Phi(U)V}{K-K',s},\, \norm{\Psi(U)V}{K-K',s}&\leq \norm{V}{K-K',s}\big(1+C\norm{U}{K,s_0}\big)\\
\end{aligned}
\end{equation} 
for any $V$ in $C^{K-K'}_{*\RRR}(I,\hcic^s)$;
\item[(ii)] $\Phi(U)[\cdot]-\uno$ and  $\Psi(U)[\cdot]-\uno$ belong to the class $\Sigma\mathcal{M}_{K,K',1}[r,N]\otimes\mathcal{M}_2(\CCC)$; moreover $\Psi(U)[\Phi(U)[\cdot]]-\uno$ is a smoothing operator in the class $\Sigma\mathcal{R}^{-\rho}_{K,K',1}[r,N]\otimes\mathcal{M}_2(\CCC)$;
\item[(iii)] the function $V=\Phi(U)U$ solves the system
\begin{equation}\label{problemafinale}
\partial_t V = \ii E\big(\Lambda V+\bonyw(L(U;t,\xi))V+Q_1(U)V+Q_{2}(U)U\big),
\end{equation}
where $\Lambda$ is defined in \eqref{DEFlambda}, 
 the operators $Q_1(U)[\cdot]$ and $Q_2(U)[\cdot]$
are (R,R,P) smoothing operators in the class $\Sigma\mathcal{R}^{-\rho+m}_{K,K',1}[r,N]\otimes\mathcal{M}_2(\CCC)$
for some $m>0$ depending on $N$,  
$L(U;t,\x)$ is a (R,R,P)-matrix 
in 
 $\Sigma\Gamma^{2}_{K,K',1}[r,N]\otimes \MM_{2}(\CCC)$
and
has the form
\begin{equation}\label{constCoeff}
L(U;t,\x):=\left(
\begin{matrix}
\mathtt{m}(U;t,\x) & 0\\ 0 & \ol{\mathtt{m}(U;t,-\x)}
\end{matrix}
\right), \quad \mathtt{m}(U;t,\x)=\mathtt{m}_{2}(U;t)(\ii\x)^{2}+\mathtt{m}_0(U;t,\x),
\end{equation}
where $\mathtt{m}_{2}(U;t)$ is a real symbol in $\Sigma\calF_{K,K',1}[r,N]$,  $\mathtt{m}_0(U;t,\x)$ is in  $\Sigma\Gamma^{0}_{K,K',1}[r,N]$ and  both of them 
are constant in $x\in \TTT$.
\end{enumerate}
\end{theo}
During this sections we shall use the following notation. 
\begin{de}\label{commutatori}
We define the \emph{commutator} between the operators $A$ and $B$ as $[A,B]_-=A\circ B-B\circ A$ and the \emph{anti-commutator} as $[A,B]_{+}=A\circ B+B\circ A$.
\end{de}
\subsection{Diagonalization of the second order operator}\label{diagosecondord}
The goal of this subsection is to transform the matrix of symbols $E(\uno+A_2(U;t,x))(\ii\xi)^2$  (where $A_{2}(U;t,x)$ is defined in \eqref{formadiA}) into a diagonal one up to a smoothing term. 
\begin{prop}\label{diagomax}
Fix $N>0$, $ \rho\gg N$ and $K\gg\rho$, then there exist $s_0>0$, $r_0>0$, such that for any $s\geq s_0$, any $0<r\leq r_0$ and any $U\in B^K_s(I,r)$ solution of \eqref{sistemainiziale} the following holds. There exist two  (R,R,P)-maps 
$\Phi_1(U)[\cdot],\, \Psi_1(U)[\cdot] :C^{K-1}_{*\RRR}(I,\hcic^s)\rightarrow C^{K-1}_{*\RRR}(I,\hcic^s),$
satisfying the following
\begin{enumerate}
\item[(i)] there exists a constant $C$ depending on $s$, $r$ and $K$ such that
\begin{equation}\label{stime-descent-1}
\begin{aligned}
\norm{\Phi_1(U)V}{K-1,s}, \, \norm{\Psi_1(U)V}{K-1,s}&\leq \norm{V}{K-1,s}\big(1+C\norm{U}{K,s_0}\big)\\
\end{aligned}
\end{equation} 
for any $V$ in $C^{K-1}_{*\RRR}(I,\hcic^s)$;
\item[(ii)] $\Phi_1(U)[\cdot]-\uno$ and  $\Psi_1(U)[\cdot]-\uno$ belong to the class $\Sigma\mathcal{M}_{K,1,1}[r,N]\otimes\mathcal{M}_2(\CCC)$; $\Psi_1(U)[\Phi_1(U)[\cdot]]-\uno$ is a smoothing operator in the class $\Sigma\mathcal{R}^{-\rho}_{K,1,1}[r,N]\otimes\mathcal{M}_2(\CCC)$;
\item[(iii)] the function $V_1=\Phi_{1}(U)U$ solves the system
\begin{equation}\label{sistemainiziale-1}
\partial_t V_1 = \ii E\big(\Lambda V_1+\bonyw(A^{(1)}(U;t,x,\xi))V_1+R_1^{(1)}(U)V_1+R^{(1)}_{2}(U)U\big),
\end{equation}
where $\Lambda$ is defined in \eqref{DEFlambda}, $A^{(1)}(U;t,x,\xi)=A_2^{(1)}(U;t,x)(\ii\xi)^2+A_1^{(1)}(U;t,x)(\ii\xi)+A_0^{(1)}(U;t,x,\x)$ is a (R,R,P) matrix in the class $\Sigma\Gamma^2_{K,1,1}[r,N]\otimes\mathcal{M}_2(\CCC)$ with $A^{(1)}_j(U;t,x)$ in $\Sigma\mathcal{F}_{K,1,1}[r,N]\otimes\mathcal{M}_2(\CCC)$ for $j=1,2$, $A_{0}^{(1)}(U;t,x,\x)$
is a matrix of symbols in $\Sigma\Gamma^0_{K,1,1}[r,N]\otimes\mathcal{M}_2(\CCC)$  and 
\begin{equation}\label{diagonale-ordine-2tris}
A^{(1)}_2(U;t,x)=\left(\begin{matrix} a^{(1)}_2(U;t,x) & 0\\
 																	0	&	a^{(1)}_2(U;t,x)
																	\end{matrix}\right),
\end{equation}
with $a^{(1)}_2(U;t,x)$ real valued, the operators $R^{(1)}_1(U)[\cdot]$ and $R^{(1)}_2(U)[\cdot]$
are (R,R,P) smoothing operators in the class $\Sigma\mathcal{R}^{-\rho+2}_{K,1,1}[r,N]\otimes\mathcal{M}_2(\CCC)$.
\end{enumerate}
\end{prop}
\begin{proof}
The matrix $E(\uno+A_2(U;t,x))$ in \eqref{sistemainiziale} and \eqref{formadiA} has  eigenvalues 
\[
\lambda^{\pm}(U;t,x)=\pm\sqrt{(1+a_2(U;t,x))^2-|b_2(U;t,x)|^2},
\]
 which are real and  well defined since $U$ is, by assumption, in $ B^K_s(I,r)$ with $r$ small enough. The matrix of eigenfunctions  is 
\begin{equation*}
M(U;t,x)=\frac12\left(         \begin{matrix} 1+a_2(U;t,x)+\lambda^+(U;t,x) & -b_2(U;t,x)\\
													-\ol{b_2}(U;t,x)   & 1+a_2(U;t,x)+\lambda^+(U;t,x) \end{matrix}\right),
\end{equation*}
it is invertible with inverse
\begin{equation*}
\begin{aligned}
M(U;t,x)^{-1}&=\frac{1}{{\rm{ det }} (M(U;t,x)) }\left(\begin{matrix} 1+a_2(U;t,x)+\lambda^+(U;t,x) & b_2(U;t,x)\\
													\ol{b_2}(U;t,x)   & 1+a_2(U;t,x)+\lambda^+(U;t,x) \end{matrix}\right)\\
                  &=2\left(\begin{matrix} 
                  \frac{1}{\lambda^+(U;t,x)} & \frac{b_2(U;t,x)}{\lambda^+(U;t,x)(\lambda^+(U;t,x)+a_2(U;t,x))}\\
                  \frac{\ol{b_2}(U;t,x)}{\lambda^+(U;t,x)(\lambda^+(U;t,x)+a_2(U;t,x))} & \frac{1}{\lambda^+(U;t,x)}\end{matrix}\right) 
                   \end{aligned}.
\end{equation*}
Therefore one has
\begin{equation}\label{diago}
M(U;t,x)^{-1}E(\uno+A_{2}(U;t,x))M(U;t,x)=\left(\begin{matrix} \lambda^+(U;t,x) & 0\\
																									0 & \lambda^-(U;t,x)\end{matrix}\right)=
										E\left(\begin{matrix} \lambda^+(U;t,x) & 0\\
																									0 & \lambda^+(U;t,x)\end{matrix}\right).
\end{equation}
By using the last item in Remark \ref{considerazioni}, since the matrix $E(\uno+A_2(U;x))$ is reversibility preserving, we have that  the matrix $M(U;t,x)$ (and therefore the matrix $M^{-1}(U;t,x)$ by the first item in Remark \ref{considerazioni}) is reversibility preserving. Arguing in the same way one deduces that both the matrices are (R,R,P). In particular the matrix in \eqref{diago} is (R,R,P).

Note that by Taylor expanding the  function $\sqrt{1+x}$ at $x=0$ one can prove that the matrices $M(U;t,x)-\uno$, $M(U;t,x)^{-1}-\uno$ and $M(U;t,x)^{-1}E(\uno+A_{2}(U;t,x))M(U;t,x)-E$ belong to the space $\Sigma\mathcal{F}_{K,0,1}[r,N]\otimes\mathcal{M}_2(\CCC)$. We set 
\begin{equation*}
\begin{aligned}
&\Phi_1(U;t,x)[\cdot]:=\bonyw\left(M(U;t,x)^{-1}\right)[\cdot],\qquad \Psi_1(U;t,x)[\cdot]:=\bonyw\left(M(U;t,x)\right)[\cdot],
\end{aligned}
\end{equation*}
these are (R,R,P) maps according Definition \ref{riassunto-mappe}, moreover by using Propositions \ref{composizioniTOTALI}, \ref{azionepara} and the discussion above one proves items $(i)$ and $(ii)$ of the statement. The function $V_1:=\Phi_1(U)U$ solves the equation
\begin{equation}\label{melone}
\begin{aligned}
\partial_t V_1=&\bonyw(\partial_t(M(U;t,x)^{-1}))U+\bonyw({M(U;t,x)^{-1})}\partial_t U\\
\stackrel{\eqref{sistemainiziale}}{=}&\bonyw(\partial_t(M(U;t,x)^{-1}))U+\\&+\bonyw(M(U;t,x)^{-1})\ii E\left(\Lambda U+\bonyw(A(U;t,x,\xi)U+R(U)U)\right).
\end{aligned}
\end{equation}
We know by previous discussions  that $U=\Psi_1(U)V_1+\widetilde{R}(U)U$ for a (R,R,P) smoothing operator $\widetilde{R}(U)$
belonging to
$\Sigma\mathcal{R}^{-\rho}_{K,0,1}[r,N]\otimes\mathcal{M}_2(\CCC)$; 
plugging this identity in the equation \eqref{melone} we get 
\begin{equation}\label{inserimento}
\begin{aligned}
\partial_t V_1&= \bonyw(\partial_t(M(U;t,x)^{-1}))\bonyw(M(U;t,x))V_1\\ &+\bonyw(M(U;t,x)^{-1})\ii E\left(\big(\Lambda +\bonyw(A(U;t,x,\xi)\big)\bonyw(M(U;t,x))V_1\right)+\widetilde{\widetilde{R}}(U)U
\end{aligned}
\end{equation}
where $\partial_t(M(U;t,x)^{-1})$ is a reversible, reality and parity preserving matrix of symbols in $\Sigma\Gamma^0_{K,1,1}[r,N]\otimes\mathcal{M}_2(\CCC)$ thanks to Lemma \ref{lemmatempo} and 
\begin{equation*}
\begin{aligned}
\widetilde{\widetilde{R}}(U)[U]=&\Big(\bonyw(\partial_tM(U;t,x)^{-1})+\bonyw(M(U;t,x)^{-1})\circ\bonyw(\ii E(\Lambda+A(U;t,x,\xi)))\Big)\left[\widetilde{R}(U)U\right]\\&+\bonyw(M(U;t,x)^{-1})\ii E R(U)U
\end{aligned}
\end{equation*}
is a reality, parity preserving  and reversible  smoothing operator (according to Def. \ref{riassunto-mappe}) in the class $\Sigma\mathcal{R}^{-\rho+2}_{K,1,1}[r,N]\otimes\mathcal{M}_2(\CCC)$ thanks to Proposition \ref{composizioniTOTALI} and Remark \ref{considerazioni}. Owing  to Proposition \ref{composizioniTOTALI}, the first summand in the r.h.s. of \eqref{inserimento} is equal to 
$$\bonyw\big(\partial_t(M(U;t,x)^{-1})M(U;t,x)\big)V_1 + Q_1(U)V_1,$$
 where $Q_1(U)[\cdot]$ is a reversible, parity and reality preserving  smoothing operator in the class $\Sigma\mathcal{R}^{-\rho}_{K,1,1}[r,N]\otimes\mathcal{M}_2(\CCC)$ and 
 $\partial_t(M(U;t,x)^{-1})M(U;t,x)$  is in $\Sigma\Gamma^0_{K,1,1}[r,N]\otimes\mathcal{M}_2(\CCC)$. Recalling that $A(U;t,x,\xi)$ has the form \eqref{formadiA}, $\Lambda$ has the form \eqref{DEFlambda}
 (see also Remark \ref{formavera}), 
 using Proposition \ref{composizioniTOTALI}  and \eqref{diago} we expand the second summand in the r.h.s. of \eqref{inserimento} as follows
 \begin{equation*}
 \begin{aligned}
\ii E\Lambda V_1+ \ii E& \bonyw\left(\left(\begin{matrix} \lambda^+(U;t,x)-1 & 0\\
				0 & \lambda^+(U;t,x)-1\end{matrix}\right)(\ii\xi)^2\right)V_1 + \ii E \bonyw\big(A_1^{(1)}(U;t,x)(\ii\xi)\big)V_1\\
				&+\ii E\bonyw\big(\widetilde{A}_0^{(1)}(U;t,x,\x)\big)V_1 +Q_2(U)V_1
\end{aligned}
 \end{equation*}
where $Q_2(U)[\cdot]$ is a smoothing operator in the class 
$\Sigma\mathcal{R}^{-\rho}_{K,1,1}[r,N]\otimes\mathcal{M}_2(\CCC)$, $\widetilde{A}_0^{(1)}(U;t,x,\x)$ 
and $A_1^{(1)}(U;t,x)$ are matrices of  symbols respectively 
 in  $\Sigma\Gamma^{0}_{K,1,1}[r,N]\otimes\mathcal{M}_2(\CCC)$ and  in $\Sigma\mathcal{F}_{K,1,1}[r,N]\otimes\mathcal{M}_2(\CCC)$. Moreover, by Lemmata \ref{lemmabello} and \ref{lemmamaxcompo}, each matrix 
 $\widetilde{A}_0^{(1)}(U;t,x,\x)$, $A_1^{(1)}(U;t,x)(\ii\xi)$ and $A_{2}^{(1)}(U;t,x)(\ii\xi)^2$ is (R,R,P) according to Definition \ref{riassunto-simbo} and the operator $Q_2(U)$ is reversible, reality and parity preserving according to Definition \ref{riassunto-mappe}. Therefore the theorem is proved by setting
\begin{equation*}
\begin{aligned}
a^{(1)}_2(U;t,x)&:=\lambda^+(U;t,x)-1,\\
A^{(1)}_0(U;t,x,\x)&:= \widetilde{A}^{(1)}_0(U;t,x,\x)-\ii E (\partial_t(M(U;t,x)^{-1})M(U;t,x)),\\
R_2^{(1)}(U)&:=-\ii E \widetilde{\widetilde{R}}(U)U,\qquad 
R_{1}^{(1)}(U):= -\ii E(Q_1(U)+Q_2(U)).
\end{aligned}
\end{equation*}
\end{proof}

\subsection{Diagonalization of lower order operators}\label{diagosecondord2}
 
 \begin{prop}\label{diago-lower}
There exist $s_0>0$, $r_0>0$, such that for any $s\geq s_0$, any $0<r\leq r_0$ and any $U\in B^K_s(I,r)$ solution of \eqref{sistemainiziale} the following holds. There exist two  
(R,R,P)-maps $\Phi_2(U)[\cdot],\, \Psi_2(U)[\cdot] :C^{K-\rho-2}_{*\RRR}(I,\hcic^s)\rightarrow C^{K-\rho-2}_{*\RRR}(I,\hcic^s),$
satisfying the following
\begin{enumerate}
\item[(i)] there exists a constant $C$ depending on $s$, $r$ and $K$ such that
\begin{equation}\label{stime-descent-2}
\begin{aligned}
\norm{\Phi_2(U)V}{K-\rho-2,s},\, \norm{\Psi_2(U)V}{K-\rho-2,s}&\leq \norm{V}{K-\rho-2,s}\big(1+C\norm{U}{K,s_0}\big)\\
\end{aligned}
\end{equation} 
for any $V$ in $C^{K-\rho-2}_{*\RRR}(I,\hcic^s)$;
\item[(ii)] $\Phi_2(U)[\cdot]-\uno$ and  $\Psi_2(U)[\cdot]-\uno$ belong to the class $\Sigma\mathcal{M}_{K,\rho+2,1}[r,N]\otimes\mathcal{M}_2(\CCC)$; $\Psi_2(U)[\Phi_2(U)[\cdot]]-\uno$ is a smoothing operator in the class $\Sigma\mathcal{R}^{-\rho}_{K,\rho+2,1}[r,N]\otimes\mathcal{M}_2(\CCC)$;
\item[(iii)] the function $V_2=\Phi_{2}(U)V_1$ (where $V_1$ is the solution of \eqref{sistemainiziale-1}) solves the system
\begin{equation}\label{sistemainiziale-2}
\partial_t V_2 = \ii E\big(\Lambda V_2+\bonyw(A^{(2)}(U;t,x,\xi))V_2+R_1^{(2)}(U)V_2+R^{(2)}_{2}(U)U\big),
\end{equation}
where $\Lambda$ is defined in \eqref{DEFlambda}, $A^{(2)}(U;t,x,\xi)=\sum_{j=-(\rho-1)}^2A_j^{(2)}(U;t,x,\xi)$ is a (R,R,P) matrix in the class $\Sigma\Gamma^2_{K,\rho+2,1}[r,N]\otimes\mathcal{M}_2(\CCC)$ with $A^{(2)}_j(U;t,x,\xi)$ diagonal matrices in $\Sigma\Gamma^j_{K,\rho+2,1}[r,N]\otimes\mathcal{M}_2(\CCC)$ for $j=-(\rho-1),\ldots, 2$ and 
\begin{equation}\label{diagonale-ordine-2}
\begin{aligned}
A^{(2)}_2(U;t,x,\xi)&=\left(\begin{matrix} a^{(2)}_2(U;t,x)(\ii\xi)^2 & 0\\
 																	0	&	a^{(2)}_2(U;t,x)(\ii\xi)^2
																	\end{matrix}\right),\quad a^{(2)}_2(U;t,x)\in\Sigma\mathcal{F}_{K,\rho+2,1}[r,N]\\
																	A^{(2)}_1(U;t,x,\xi)&=\left(\begin{matrix} a^{(2)}_1(U;t,x)(\ii\xi) & 0\\
 																	0	&	\ol{a^{(2)}_1(U;t,x)}(\ii\xi)
																	\end{matrix}\right),\quad a^{(2)}_1(U;t,x)\in\Sigma\mathcal{F}_{K,\rho+2,1}[r,N]\\
\end{aligned}
\end{equation}
with $a^{(2)}_2(U;t,x)$ equals to $a^{(1)}_2(U;t,x)$ in \eqref{diagonale-ordine-2tris} real valued, the operators $R^{(2)}_1(U)[\cdot]$ and $R^{(2)}_2(U)[\cdot]$
are (R,R,P) smoothing operators in the class $\Sigma\mathcal{R}^{-\rho+2}_{K,\rho+2,1}[r,N]\otimes\mathcal{M}_2(\CCC)$.
\end{enumerate}
\end{prop}
\begin{proof}
Consider the following matrix
\begin{equation}\label{formadiD}
D_1(U;t,x,\xi):=\left(\begin{matrix}
0 & d_1(U;t,x,\xi)\\
\ol{d_1(U;t,x,-\xi)} & 0
\end{matrix}\right),
\end{equation}
where the symbol $d_1(U;t,x,\xi)$ is in the class $\Sigma\Gamma^{-1}_{K,1,1}[r,N]$. Note that $(\uno+D_1(U;t,x,\xi))(\uno-D_1(U;t,x,\xi))$ is equal to the identity modulo a matrix of symbols in  $\Sigma\Gamma^{-2}_{K,1,1}[r,N]\otimes\mathcal{M}_2(\CCC)$. Define the following matrices of symbols \begin{equation}\label{paramatr1}
\begin{aligned}
G_1(U;t,x,\xi)&:=\uno-\big(\uno+D_1(U;t,x,\xi)\sharp(\uno-D_1(U;t,x,\xi))\big)_{\rho}\in \Sigma\Gamma^{-2}_{K,1,1}[r,N]\otimes\mathcal{M}_2(\CCC)\\
Q_1(U;t,x,\xi)&:= (\uno-D_1)+\big((\uno-D_1)\sharp G_1\big)_{\rho}+\ldots+ \big((\uno-D_1)\sharp \underbrace{G_1\sharp\ldots\sharp G_1}_{\rho-times}\big)_{\rho},
\end{aligned}
\end{equation}
where in the right hand side of  the latter equation we omitted, with abuse of notation, the dependence on $U$, $x$ and $\xi$. Then one has 
\begin{equation}\label{paramatr2}
\begin{aligned}
\big((\uno+D_1)\sharp Q_1\big)_{\rho}&= \big((\uno+D_1)\sharp (\uno-D_1)\big)_{\rho}+\\
&+\big((\uno+D_1)\sharp (\uno-D_1)\sharp G_1\big)_{\rho}+\ldots+\big((\uno+D_1)\sharp (\uno-D_1)\sharp G_1\sharp\ldots\sharp G_1\big)_{\rho}\\
&= (\uno-G_1)+\big((\uno-G_1)\sharp G_1\big)_{\rho}+\ldots +\big((\uno-G_1)\sharp G_1\sharp\ldots\sharp G_1\big)_{\rho}\\
&=\uno -\big(G_1\sharp\ldots\sharp G_1)_{\rho},
\end{aligned}
\end{equation}
moreover the matrix of symbols $\big(G_1\sharp\ldots\sharp G_1)_{\rho}$ is in the class $\Sigma\Gamma^{-2\rho}_{K,1,1}[r,N]\otimes\mathcal{M}_2(\CCC)$. We set
\begin{equation}\label{phipsi2}
\Phi_{2,1}(U)[\cdot]:=\bonyw(\uno+D_1(U;t,x,\xi))[\cdot], \quad \Psi_{2,1}(U)[\cdot]:=\bonyw(Q_1(U;t,x,\xi)).
\end{equation}
The previous discussion proves that the maps $\Phi_{2,1}(U)$ and $\Psi_{2,1}(U)$ satisfy the estimates \eqref{stime-descent-2} with $\rho=0$, moreover thanks to Prop. \ref{composizioniTOTALI} and Remark \ref{inclusionifacili} there exists a smoothing remainder $R(U)$ in the class $\Sigma\mathcal{R}^{-2\rho}_{K,1,1}[r,N]\otimes\mathcal{M}_2(\CCC)$ such that  $(\Psi_{2,1}(U)\circ\Phi_{2,1}(U))V=V+R(U)U$.

 The function $V_{2,1}:=\Phi_{2,1}(U)V_1$ solves the equation
\begin{equation}\label{gatto}
\begin{aligned}
\partial_t V_{2,1}=&\bonyw(\partial_t D_{1}(U;t,x,\xi))V_{2,1}+\Phi_{2,1}(U)\Big[\ii E\bonyw(A^{(1)}(U;t,x,\xi))\Psi_{2,1}(U)V_{2,1}+\\
&+\ii E\Lambda \Psi_{2,1}(U)V_{2,1}+\ii E R^{(1)}_1(U)\Psi_{2,1}(U)V_{2,1}+\ii E R^{(1)}_2(U)U
\Big] +\\
&-\Big\{\bonyw(\partial_t D_{1}(U;t,x,\xi))+\Phi_{2,1}(U)\Big[\ii E\big(\Lambda+\bonyw(A^{(1)}(U;t,x,\xi)\big)+R^{(1)}_1(U))\Big]\Big\}R(U)U.
\end{aligned}
\end{equation}
Owing to Lemma \ref{lemmatempo} the matrix of symbols $\partial_tD_1(U;t,x,\xi)$ is in the class $\Sigma\Gamma^{-1}_{K,2,1}[r,N]\otimes\mathcal{M}_2(\CCC)$. The last summand in the r.h.s. of \eqref{gatto} is a (R,R,P) smoothing remainder in the class $\Sigma\mathcal{R}^{-2\rho+2}_{K,2,1}[r,N]\otimes\mathcal{M}_2(\CCC)$ thanks to Lemmata \ref{lemmatempo}, \ref{lemmamaxcompo}, \ref{lemmabello} and Proposition \ref{composizioniTOTALI}. Recalling that $A^{(1)}(U;t,x,\xi)=A_2^{(1)}(U;t,x)(\ii\xi)^2+A_1^{(1)}(U;t,x)(\ii\xi)+A_0^{(1)}(U;t,x,\x)$, we have, up to a (R,R,P) smoothing operator in the class $\Sigma\mathcal{R}^{-\rho}_{K,2,1}[r,N]\otimes\mathcal{M}_2(\CCC)$, that
\begin{equation*}
\begin{aligned}
\ii\Phi_{2,1}(U)\bonyw&{\Big(E\big(\uno+A_2^{(1)}(U;t,x)\big)(\ii\xi)^2}\Big)\Psi_{2,1}(U)=\ii\bonyw\Big(E\big(\uno+A_2^{(1)}(U;t,x)(\ii\xi)^2\big)\Big)+\\
+&\ii\Big[\bonyw\big(D_1(U;t,x,\xi)\big),\bonyw\big(E\big(\uno+A_2^{(1)}(U;t,x)\big)(\ii\xi)^2\big)\Big]_{-}+\bonyw\big(M_1(U;t,x,\xi)\big),
\end{aligned}
\end{equation*}
where $M_1(U;t,x,\xi)$ is a (R,R,P) matrix of symbols in $\Sigma\Gamma^0_{K,2,1}[r,N]\otimes\mathcal{M}_2(\CCC)$, here the commutator $[\cdot,\cdot]_{-}$ is defined in Definition \ref{commutatori} (actually $M_1(U;t,x,\xi)$ belongs 
to $\Sigma\Gamma^0_{K,1,1}[r,N]\otimes\mathcal{M}_2(\CCC)$, but we preferred  to embed it in the above larger class in order to simplify the notation; we shall do this simplification systematically). The conjugation of $A_1^{(1)}(U;t,x)(\ii\xi)$ is, up to a (R,R,P) smoothing operator in $\Sigma\mathcal{R}^{-\rho}_{K,2,1}[r,N]\otimes\mathcal{M}_2(\CCC)$,
\begin{equation*}
\ii\Phi_{2,1}(U)\bonyw{\Big(EA_1^{(1)}(U;t,x)(\ii\xi)}\Big)\Psi_{2,1}(U)=\ii\bonyw{\Big(EA_1^{(1)}(U;t,x)(\ii\xi)}\Big)+ \bonyw\Big(M_{2}(U;t,x,\xi)\Big)
\end{equation*}
for a (R,R,P) matrix of symbols $M_2(U;t,x,\xi)$ in $\Sigma\Gamma^0_{K,2,1}[r,N]\otimes\mathcal{M}_2(\CCC)$. Therefore the matrix of operators of order one is given by
\begin{equation*}
\begin{aligned}
\ii&\Big[\bonyw\big(D_1(U;t,x,\xi)\big),\bonyw\big(E\big(\uno+A_2^{(1)}(U;t,x)\big)(\ii\xi)^2\big)\Big]_{-}+\ii EA_{1}^{(1)}(U;t,x)(\ii\xi)=\\
&\ii E\left(\begin{matrix} \bonyw\big(a_1^{(1)}(U;t,x)(\ii\xi)\big) & M_+\\
 \ol{M}_{+} &  \bonyw\Big(\ol{a_1^{(1)}(U;t,x)}(\ii\xi)\Big)
\end{matrix}\right)
\end{aligned}
\end{equation*}
where $M_+:= \bonyw\big(b_{1}^{(1)}(U;t,x)(\ii\xi)\big)-\left[\bonyw(d_1(U;t,x,\xi)),\bonyw\big(1+a^{(1)}_2(U;t,x)(\ii\xi)^2\big)\right]_{+}$, thus our aim is to choose  the symbol $d_1(U;t,x,\xi)$ in such a way that $M_+$ at the principal order is $0$. Developing the compositions by means of Proposition \ref{composizioniTOTALI} we obtain that, at the level of principal symbol, we need to solve the equation
\begin{equation*}
2d_1(U;t,x,\xi)(1+a_2^{(1)}(U;t,x))(\ii\xi)^2=b_1^{(1)}(U;t,x)(\ii\xi).
\end{equation*}
We choose the symbol $d_1(U;t,x,\xi)$ as follows
\begin{equation}\label{defd1}
\begin{aligned}
d_1(U;t,x,\xi)& :=\left(\frac{b_1^{(1)}(U;t,x)}{2(1+a_2^{(1)}(U;t,x))}\right)\cdot\gamma(\xi),\quad
\gamma(\xi) :=\left\{\begin{matrix}  \frac{1}{\ii\xi}  &   |\xi|\geq1/2,   \\
 \mbox{odd continuation of class}\,\, C^{\infty} &    |\xi|\in [0,1/2). \\ \end{matrix}\right.
\end{aligned}
\end{equation}
 Note that by Taylor expanding the function $x\mapsto (1+x)^{-1}$ one gets that $d_1(U;t,x,\xi)$ in \eqref{defd1} is a symbol in the class $\Sigma\Gamma^{-1}_{K,2,1}[r,N]$, therefore by symbolic calculus (Prop. \ref{composizioniTOTALI}) one has that $M_+$ is equal to $\bonyw(\widetilde{b}_0(U;t,x,\xi))+\widetilde{R}(U)$ for a symbol $\widetilde{b}_0(U;t,x,\xi)$ in $\Sigma\Gamma^0_{K,2,1}[r,N]$ and a smoothing operator $\widetilde{R}(U)$ in $\Sigma\mathcal{R}^{-\rho}_{K,2,1}[r,N]$.
 
The symbol $d_1(U;t,x,\xi)$ defined in \eqref{defd1} satisfies the equation $\ol{d_1(U;-t,x,\xi)}=d_1({U_S;t,x,\xi})$ since both the symbols $a_2^{(1)}(U;t,x)$ and $b_1^{(1)}(U;t,x)$ fulfil the same condition, therefore by Remark \ref{considerazioni} (see the last item) we deduce that the matrix $D_1(U;t,x,\xi)$ is reversibility preserving. By hypothesis  the symbol $b_1^{(1)}(U;t,x)$ is odd in $x$, $a_2^{(1)}(U;t,x)$ is even  and then, since $\gamma(\xi)$ is odd in $\xi$, we have $d_1(U;t,x,\xi)=d_1(U;t,-x,-\xi)$, which means  that the matrix $D_1(U;t,x,\xi)$ is parity preserving. Furthermore $D_1(U;t,x,\xi)$ is reality preserving by construction, therefore we can deduce that such a matrix is a (R,R,P) matrix of symbols.
Therefore the function $V_{2,1}$ solves the system
\begin{equation*}
\begin{aligned}
\partial_t V_{2,1}&=\ii E\big(\Lambda V_{2,1}+\bonyw(A^{(2,1)}(U;t,x,\xi))V_{2,1}+R_{1}^{(2,1)}(U)V_{2,1}+R^{(2,1)}_2(U)U\big)\\
A^{(2,1)}(U;t,x,\xi)&=\left(\begin{matrix} a_2^{(1)}(U;t,x) & 0\\
																				0 &   a_2^{(1)}(U;t,x)\end{matrix}\right)(\ii\xi)^2+
\left(\begin{matrix} a_1^{(1)}(U;t,x) & 0\\
																				0 &   \ol{a_1^{(1)}(U;t,x)}\end{matrix}\right)(\ii\xi)+\\
																				+&
																			\left(\begin{matrix} a_0^{(2,1)}(U;t,x,\xi) & b_0^{(2,1)}(U;t,x,\xi) \vspace{0.4em}\\
																				\ol{ b_0^{(2,1)}(U;t,x,-\xi) }&   \ol{a_1^{(2,1)}(U;t,x,-\xi)}\end{matrix}\right).
\end{aligned}
\end{equation*}
Suppose now that there exist $j\geq1$ (R,R,P) maps $\Phi_{2,1}(U),\ldots, \Phi_{2,j}(U)$ such that $V_{2,j}:=\Phi_{2,1}(U)\circ\ldots\circ \Phi_{2,j}(U)[V_1]$ solves the problem 
\begin{equation*}
\partial_t V_{2,j}=\ii E\big(\Lambda V_{2,j}+\bonyw(A^{(2,j)}(U;t,x,\xi))V_{2,j}+R_{1}^{(2,j)}(U)V_{2,j}+R^{(2,j)}_2(U)U\big),
\end{equation*}
where $R_{1}^{(2,j)}(U)$ and $R_{2}^{(2,j)}(U)$ are in the class $\Sigma\mathcal{R}^{-\rho}_{K,j+1,1}[r,N]\otimes\mathcal{M}_2(\CCC)$ and 
\begin{equation*}
\begin{aligned}
A^{(2,j)}(U;t,x,\xi)&=\sum_{j'=-j}^2A_{j'}^{(2,j)}(U;t,x,\xi),\\
A_{j'}^{(2,j)}(U;t,x,\xi)&=\left( \begin{matrix} a_{j'}^{(2,j)}(U;t,x,\xi) & 0\\ 
0 & \ol{a_{j'}^{(2,j)}(U;t,x,-\xi)} \end{matrix}\right)\in\Sigma\Gamma^{j'}_{K,j+1,1}[r,N]\otimes\mathcal{M}_{2}(\CCC), \quad  \,\, j'=-j+1,\ldots, 2,\\
a_{2}^{(2,j)}(U;t,x,\xi)&=a^{(1)}_2(U;t,x)(\ii\xi)^2\in\RRR\\
A_{-j}^{(2,j)}(U;t,x,\xi)&= \left( \begin{matrix} a_{-j}^{(2,j)}(U;t,x,\xi) & b_{-j}^{(2,j)}(U;t,x,\xi)
\vspace{0.4em}\\ 
\ol{b_{-j}^{(2,j)}(U;t,x,-\xi)} & \ol{a_{-j}^{(2,j)}(U;t,x,-\xi)} \end{matrix}\right)\in\Sigma\Gamma^{-j}_{K,j+1,1}[r,N]\otimes\mathcal{M}_{2}(\CCC).
\end{aligned}
\end{equation*}
We now explain how to construct a map $\Phi_{j+1}(U)$ which diagonalize the matrix $A_{-j}^{(2,j)}(U;t,x,\xi)$ up to lower order terms. Define
\begin{equation*}
\Phi_{2,j+1}(U):= \uno+\bonyw\big(D_{j+1}(U;t,x,\xi)\big); \quad D_{j+1}(U;t,x,\xi):=\left(\begin{matrix}0 & d_{j+1}(U;t,x,\xi) \vspace{0.4em}\\ \ol{d_{j+1}(U;t,x,-\xi)} & 0  \end{matrix}\right),
\end{equation*}
with $d_{j+1}(U;t,x,\xi)$ a symbol in $\Sigma\Gamma^{-{j-2}}_{K,2+j,1}[r,N]$. An approximate inverse $\Psi_{2,j+1}(U):=\bonyw{(Q_{j+1}(U))}$ can be constructed  exactly as done in \eqref{paramatr1} and \eqref{paramatr2}. Reasoning as done above one can prove that  the function $V_{2,j+1}:=\Phi_{2,j+1}(U)V_{2,j}$ solves the problem
\begin{equation}\label{diago-1-j}
\begin{aligned}
\partial_t V_{2,j+1}&=\ii E\big(\Lambda V_{2,j+1}+\sum_{j'=-j+1}^2\bonyw(A_{j'}^{(2,j)}(U;t,x,\xi))V_{2,j+1}+R^{(2,j+1)}_1(U)V_{2,j+1}+R^{(2,j+1)}_2(U)U\Big)\\
&+\ii\left[\bonyw\big({D_{j+1}(U;t,x,\xi)}\big),E\bonyw\big(\uno+A_2^{(2,j)}(U;t,x,\xi)\big)\right]_{-}+\ii E\bonyw\big(A_{-j}^{(2,j)}(U;t,x,\xi)\big).
\end{aligned}
\end{equation}
Developing the commutator above  one obtains that the sum of the last two terms in \eqref{diago-1-j} is equal to 
\begin{equation*}
\ii E\left(\begin{matrix} \bonyw\big(a_{-j}^{(2,j)}(U;t,x,\xi)\big) & M_{j,+}\vspace{0.4em}\\
 \ol{M}_{j,+} &  \bonyw\big(\ol{a_{-j}^{(2,j)}(U;t,x,-\xi)}\big),
\end{matrix}\right)
\end{equation*}
where $M_{j,+}:= \bonyw\big(b_{-j}^{(2,j)}(U;t,x,\xi)\big)-\left[\bonyw(d_{j+1}(U;t,x,\xi)),\bonyw\big((1+a^{(1)}_2(U;t,x))(\ii\xi)^2\big)\right]_{+}$, therefore, repeating the same argument used  in the case of the symbol of order one, one has to choose 
\begin{equation*}
\begin{aligned}
d_{j+1}(U;t,x,\xi)& :=\left(\frac{b_{-j}^{(2,j)}(U;t,x,\xi)}{2\big(1+a_2^{(1)}(U;t,x)\big)}\right)\cdot\gamma(\xi),\quad
\gamma(\xi)& :=\left\{\begin{matrix}  \frac{1}{(\ii\xi)^2}  &   |\xi|\geq1/2,   \\
 \mbox{odd continuation of class}\,\, C^{\infty} &    |\xi|\in [0,1/2). \\ \end{matrix}\right.
\end{aligned}
\end{equation*}
Therefore we obtain the thesis of the theorem by setting $\Phi_2(U):=\Phi_{2,1}(U)\circ\ldots\circ\Phi_{2,\rho-1}(U)$ and $\Psi_2(U):=\Psi_{2,\rho-1}(U)\circ\ldots\circ\Psi_{2,1}(U)$.
\end{proof} 

\subsection{Reduction to constant coefficients: paracomposition}\label{diagosecondord3}
In this section we shall reduce  the  operator $\bonyw(A^{(2)}_{2}(U;t,x,\xi))$, given in terms of the diagonal matrix 
\eqref{diagonale-ordine-2}, to a constant coefficients one up to smoothing remainders. We shall conjugate the system \eqref{sistemainiziale-2} under the paracomposition operator $\Phi^{\star}_U:=\Omega_{B(U)}\cdot\uno$ defined in Section 2.5 of \cite{maxdelort}, induced by a diffeomorphism of $\TTT^1$, $\Phi_U: x\mapsto x+\beta(U;t,x)$, for a small periodic real valued function $\beta(U;t,x)$ to be chosen. Indeed in Lemma 2.5.2 of \cite{maxdelort} it is shown that if $\beta(U;t,x)$ is a real valued  function in $\Sigma\mathcal{F}_{K,K',1}[r,N]$ with $r$ small enough, then for any $K'\leq K$ the map $\Phi_U$ is a diffeomorphism of the torus into itself, whose inverse may be written as $\Phi_{U}^{-1}:y\mapsto y+\gamma(U;t,y)$ for some small and real valued function $\gamma(U;t,y)$ in $\Sigma\mathcal{F}_{K,K',1}[r,N]$.
We recall below the alternative construction of $\Phi^{\star}_{U}$ given in \cite{maxdelort}.
Set $\beta{(t,x)}:=\beta(U;t,x)$ and we define the following quantities 
\begin{equation}\label{bbb}
\begin{aligned}
B(\tau;t,x,\x)=B(\tau,U;t,x,\x)&:=-\ii b(\tau;t,x)(\ii \x), \quad 
b(\tau;t,x):=\frac{\be(t,x)}{(1+\tau \be_{x}(t,x))}.
\end{aligned}
\end{equation}
Then one defines the paracomposition operator associated to the diffeomorphism $\Phi_U$ as $\Phi^{\star}_U:=\Omega_{B(U)}(1)\cdot\uno$, where $\Omega_{B(U)}(\tau)$ is the flow of the linear para-differential equation
\begin{equation}\label{flow}
\left\{
\begin{aligned}
&\frac{d}{d\tau}\Omega_{B(U)}(\tau)=\ii \bonyw{(B(\tau;t,x,\x))}\Omega_{B(U)}(\tau),\\
&\Omega_{B(U)}(0)=\rm{id}.
\end{aligned}\right.
\end{equation}
The well-posedness issues of the problem \eqref{flow} are studied in Lemma 2.5.3 of \cite{maxdelort}.

The goal is to conjugate the system \eqref{sistemainiziale-2} under the paracomposition operator $\Phi_U^{\star}$. Therefore it is necessary to study the conjugate of the differential operator $\partial_t$. This has been already done in Proposition 2.5.9 in \cite{maxdelort}, we restate below such a proposition being slightly more precise on the thesis. In other words we emphasize an algebraic property that could be deduced from the  proof therein. Such a property was not crucial in their context, however will be very useful for our purposes.
\begin{prop}\label{coniugazione-tempo-1}
Let $\rho>0$, $K\gg \rho$, $r\ll 1$ and a symbol $\beta(U;t,x)$ in $\Sigma\mathcal{F}_{K,\rho+2,1}[r,N]$. If, according to notation \eqref{bbb}, $\Omega_{B(U)}(\tau)$ is the flow of \eqref{flow}, then
\begin{equation}\label{coniugazione-tempo}
\begin{aligned}
\Omega_{B(U)}(\tau)\circ\partial_t\circ\Omega_{B(U)}^{-1}&=\del_{t}+\Omega_{B(U)}(\tau)\circ(\del_{t}\Omega_{B(U)}^{-1}(\tau))
\\
&=\partial_t+\bonyw\big(e(U;t,x,\x)\big)+R(U;t),
\end{aligned}
\end{equation}
where 
\begin{equation}\label{simboloE}
e(U;t,x,\xi)=e_{1}(U;t,x)(\ii\x)+e_{0}(U;t,x,\x),
\end{equation}
with $e_1(U;t,x)\in\Sigma\mathcal{F}_{K,\rho+3,1}[r,N]$, ${e}_0(U;t,x,\xi)$ is in $\Sigma\Gamma^{-1}_{K,\rho+3,1}[r,N]$ and $R(U;t)$ is a smoothing remainder belonging to $\Sigma\mathcal{R}^{-\rho}_{K,\rho+3,1}[r,N]$.
Moreover  ${\rm Re}(e)\in \Sigma\Gamma^{-1}_{K,\rho+3,1}[r,N]$.
\end{prop}
\begin{proof}
The difference between our thesis and the one of Prop. 2.5.9 in \cite{maxdelort} is 
that we have the additional information that, at the highest order, the symbol $e(U;t,x,\x)$ is homogeneous in the variable $\x$.
This properties follows by equations $(2.5.49)$, $(2.5.50)$ in \cite{maxdelort} and the fact that
the symbol $B(U;t,x,\x)$ in \eqref{bbb}
is homogeneous in $\x$.
\end{proof}

\begin{prop}\label{egorov}
In the notation of Prop. \ref{coniugazione-tempo-1} there exists a real valued function 
$\beta(U;t,x) \in\Sigma\mathcal{F}_{K,\rho+2,1}[r,N]$  such that the function $V_{3}:=\Phi^{\star}_U V_2$ (where $V_2$ is a solution of \eqref{sistemainiziale-2}) solves the following problem
\begin{equation}\label{sistemainiziale-3}
\partial_t V_3 = \ii E\big(\Lambda V_3+\bonyw(A^{(3)}(U;t,x,\xi))V_3+R_1^{(3)}(U)V_3+R^{(3)}_{2}(U)U\big),
\end{equation}
where $A^{(3)}(U;t,x,\xi)=\sum_{j=-(\rho-1)}^2A_j^{(3)}(U;t,x,\xi)$ is a (R,R,P) matrix in the class $\Sigma\Gamma^2_{K,\rho+3,1}[r,N]\otimes\mathcal{M}_2(\CCC)$ with $A^{(3)}_j(U;t,x,\xi)$ diagonal matrices in $\Sigma\Gamma^j_{K,\rho+3,1}[r,N]\otimes\mathcal{M}_2(\CCC)$ for $j=-(\rho-1),\ldots, 2$ and 
\begin{equation}\label{diagonale-ordine-2-doppio}
\begin{aligned}
A^{(3)}_2(U;t,\xi)&=\left(\begin{matrix} a^{(3)}_2(U;t)(\ii\xi)^2 & 0\\
 																	0	&	a^{(3)}_2(U;t)(\ii\xi)^2
																	\end{matrix}\right),\quad a^{(3)}_2(U;t)\in\Sigma\mathcal{F}_{K,\rho+3,1}[r,N]\\
																	A^{(3)}_1(U;t,x,\xi)&=\left(\begin{matrix} a^{(3)}_1(U;t,x)(\ii\xi) & 0\\
 																	0	&	\ol{a^{(3)}_1(U;t,x)}(\ii\xi)
																	\end{matrix}\right),\quad a^{(3)}_1(U;t,x)\in\Sigma\mathcal{F}_{K,\rho+3,1}[r,N]\\
																	\end{aligned}
\end{equation}
with $a^{(3)}_2(U;t)$ real valued and independent of $x$,  the operators $R^{(3)}_1(U)[\cdot]$ and $R^{(3)}_2(U)[\cdot]$
are (R,R,P) smoothing operators in the class $\Sigma\mathcal{R}^{-\rho+m}_{K,\rho+3,1}[r,N]\otimes\mathcal{M}_2(\CCC)$
for some $m=m(N)>0$.
\end{prop}
\begin{proof}
The function $V_3:=\Phi^{\star}_UV_2$ solves the following problem
\begin{equation}\label{roberto}
\begin{aligned}
\partial_{t}V_{3}&=(\partial_t\Phi^{\star}_{U})(\Phi^{\star}_U)^{-1}V_3+
\Phi^{\star}_{U}\Big(\ii E\big(\Lambda +\bonyw(A^{(2)}(U;t,x,\xi))\big)
\Big)(\Phi^{\star}_{U})^{-1}V_{3}\\
&+\Phi^{\star}_{U}(\ii E R_{1}^{(2)}(U))(\Phi^{\star}_{U})^{-1}V_{3}+\Phi^{\star}_{U}(\ii E R_{2}^{(2)}(U)U).
\end{aligned}
\end{equation}

Our aim is to choose $\be(U;t,x)$
in such a way that the coefficient in front of $(\ii\x)^{2}$ in the new symbol is constant in $x\in \TTT$.
Recalling that $A^{(2)}(U;t,x,\x)$ has the form \eqref{diagonale-ordine-2}, we have, by Theorem $2.5.8$ in \cite{maxdelort},
that 
the term 
\begin{equation}\label{antonio0}
\Phi^{\star}_{U}\Big(\ii E\big(\Lambda +\bonyw(A^{(2)}(U;t,x,\xi))\big)
\Big)(\Phi^{\star}_{U})^{-1}V_{3}\end{equation} 
in \eqref{roberto} is equal to
\begin{equation}\label{antonio}
\ii E\Big[\Lambda+\bonyw \left(
\begin{matrix}
r(U;t,x)(\ii\x)^{2} & 0 \\
 0  & r(U;t,x)(\ii\x)^{2}
\end{matrix}
\right)V_{3}+ \bonyw(A^{+}_{1}(U;t,x)(\ii\x))V_3+\bonyw(A^{+}_{0}(U;t,x,\x))V_{3}\Big]
\end{equation}
up to smoothing remainders in $\Sigma\RR^{-\rho}_{K,\rho+3,1}[r,N]\otimes \MM_{2}(\CCC)$, where
\[
r(U;t,x)=(1+a_{2}^{(2)}(U;t,y))(1+\partial_y\gamma(U;t,y))^{2}_{|_{y=x+\be(U;t,x)}}-1,
\]
 $A^{+}_{1}(U;t,x)\in \Sigma\calF_{K,\rho+3,1}[r,N]\otimes\MM_{2}(\CCC)$ and $A^{+}_{0}(U;t,x,\x)\in \Sigma\Gamma^{0}_{K,\rho+3,1}[r,N]\otimes\MM_{2}(\CCC)$ are diagonal matrices of symbols.

We define
\begin{equation}\label{mgrande2}
\gamma(U;t,y)=\partial_y^{-1}\left(\sqrt{\frac{1+a_{2}^{(3)}(U;t)}{1+a^{(2)}_2(U;t,y)}}-1\right), 
\end{equation}
where 
\begin{equation}\label{felice6}
a_{2}^{(3)}(U;t):=\left[2\pi \left(\int_{\TTT}\frac{1}{\sqrt{1+a^{(2)}_{2}(U;t,y)}}dy
\right)^{-1}\right]^{2}-1.
\end{equation}
Thanks to 
this choice we have 
\[
r(U;t,x)\equiv a_{2}^{(3)}(U;t),
\]
moreover the paracomposition operator $\Phi^{\star}_U$  is parity and reversibility preserving, satisfies the anti-reality  condition  for the following reasons. The real valued function $\gamma(U;t,x)$ in \eqref{mgrande2} satisfies $\gamma(U;-t,x)=\gamma(U_S;t,x)$ since $a^{(2)}_2(U;t,x)$ satisfies the same equation, moreover $\gamma(U;t,x)$ is an odd function since is defined as a primitive of the even function $a^{(2)}_2(U;t,x)$. It follows that also the function $\beta(U;t,x)$ satisfies the same properties, therefore the matrix of symbols
\begin{equation}
B(\tau,U;t,x,\xi)\cdot \uno= \frac{\beta(U;t,x)}{1+\tau\beta_x(U;t,x)}\xi\cdot\uno
\end{equation}
is a (R,R,P) matrix in $\Sigma\Gamma^1_{K,\rho+3,1}[r,N]\otimes\mathcal{M}_2(\CCC)$. Therefore  $\Omega_{B(U)}(1)\cdot\uno$ generated by $\bonyw(\ii B(U;t,x,\xi)\cdot\uno)$ is parity and reversibility preserving and it satisfies  the anti-reality condition \eqref{realereale2} by Lemma 4.2.2 in \cite{maxdelort}.
Thanks to this the term in \eqref{antonio0} is a parity and reality preserving and reversible vector field, therefore owing to Lemma \ref{lemmabello} each term of the equation \eqref{antonio} (together with the omitted smoothing remainder) is a parity and reality preserving and reversible vector field. 

The term $\Phi^{\star}_{U}(\ii E R_{1}^{(2)}(U))(\Phi^{\star}_{U})^{-1}V_{3}+\Phi^{\star}_{U}(\ii E R_{2}^{(2)}(U)U)$ in \eqref{roberto} is analysed as follows. First of all both the operators $\Phi^{\star}_{U}(\ii E R_{1}^{(2)}(U))(\Phi^{\star}_{U})^{-1}$ and $\Phi^{\star}_{U}(\ii E R_{2}^{(2)}(U)$ are reversible, parity and reality preserving thanks to Lemma \ref{lemmamaxcompo}. Moreover we remark that the paracomposition operator may be written as
\begin{equation}\label{roberto1}
\Phi_{U}^{\star}=\Omega_{B(U)}(1)=U+\sum_{p=2}^{N-1}M_{p}(U,\ldots,U)U+M_{N}(U;t)U,
\end{equation}
where $M_{p}\in \widetilde{\MM}^{m}_{p}$, $M_{N}(U;t)\in \MM_{K,\rho,N}^{m}[r]$ for some 
$m>0$ depending only on $N$.
 This is a consequence of Theorem $2.5.8$ in \cite{maxdelort}. Therefore as a consequence of Prop. \ref{composizioniTOTALI} the operators $\Phi^{\star}_{U}(\ii E R_{1}^{(2)}(U))(\Phi^{\star}_{U})^{-1}$ and $\Phi^{\star}_{U}(\ii E R_{2}^{(2)}(U)$ belong to the class $\Sigma\mathcal{R}^{-\rho+2m}_{K,\rho+3,1}[r,N]\otimes\mathcal{M}_2(\CCC)$.

We are left to study the term $(\partial_t\Phi^{\star}_{U})(\Phi^{\star}_U)^{-1}V_3$ in \eqref{roberto}. 
This is  nothing but $-\Phi^{\star}_{U}(\partial_t\Phi^{\star}_U)^{-1}V_3$, therefore by Prop. \ref{coniugazione-tempo-1} it is equal to $\bonyw\big(e(U;t,x,\x)\cdot\uno\big)+R(U;t)\cdot \uno$ with $e(U;t,x,\xi)$ and $R(U;t)$ given by Prop. \ref{coniugazione-tempo-1}. Since the map $\Phi^{\star}_U$ satisfies the reality condition \eqref{realereale} its derivative $\partial_t\Phi^{\star}_U=\partial_t\Omega_{B(U)}(1)\cdot\uno$ is reality preserving, this follows by taking the derivative with respect to $t$ to both side of the equation \eqref{flow} and by using the fact that 
the fact that $\partial_t B(U;t,x,\tau,\xi)\cdot\uno$ still satisfies the anti-reality condition \eqref{realereale2}.
We deduce, by using Lemma \ref{lemmamaxcompo}, that the maps $-\Phi^{\star}_{U}(\partial_t\Phi^{\star}_U)^{-1}$ is reality preserving. Since the map $\Phi_U^{\star}$ is reversibility preserving and parity preserving, one reasons in the same way as above to
prove that its derivative with respect to $t$ is parity preserving and reversible. Therefore thanks to Lemma \ref{lemmabello} we can also assume that both the matrix of symbols $e(U;t,x,\x)\cdot\uno$ and the operator $R(U;t)$ above are reality and parity preserving and reversible.
\end{proof}

\subsection{Reduction to constant coefficients: elimination of the term of order one}\label{diagosecondord4}
In this section we shall eliminate the matrix of order one by conjugating the system through a multiplication operator.
\begin{prop}
There exist $s_0>0$, $r_0>0$ such that for any $s\geq s_0$, $r\leq r_0$
 and any $U\in B^K_s(I,r)$ solution of \eqref{sistemainiziale} 
 the following holds. There exist two (R,R,P) maps 
 $\Phi_4(U),\,\,\Psi_4(U): C^{K-(\rho+4)}_{*\RRR}(I;\hcic^s(\TTT))\rightarrow C^{K-(\rho+4)}_{*\RRR}(I;\hcic^s(\TTT))$
\begin{itemize}
\item[(i)] there exists a constant $C$ depending on $s$, $r$ and $K$ such that
\begin{equation}\label{stime-descent-4}
\begin{aligned}
\norm{\Phi_4(U)V}{K-(\rho+4),s},\, \norm{\Psi_4(U)V}{K-(\rho+4),s}&\leq \norm{V}{K-(\rho+4),s}\big(1+C\norm{U}{K,s_0}\big)\\
\end{aligned}
\end{equation} 
for any $V$ in $C^{K-(\rho+4)}_{*\RRR}(I,\hcic^s)$;

\item[(ii)] $\Phi_4(U)-\uno$, $\Psi_4(U)-\uno$ belong to $\Sigma\mathcal{M}_{K,\rho+4,1}[r,N]\otimes \mathcal{M}_{2}(\CCC)$, moreover $\Psi_4(U)\circ\Phi_4(U)=\uno+R(U)U$ with $R(U)$ is in $\Sigma\mathcal{R}^{-\rho}_{K,\rho+4,1}[r,N]\otimes\mathcal{M}_2(\CCC)$ and it is reversible, parity and reality preserving.
\item[(iii)] The function  $V_4=\Phi_4(U)V_3$ (where $V_3$ solves \eqref{sistemainiziale-3}) solves the problem
\begin{equation}\label{sistemafinale-4}
\partial_t V_4=\ii E\big(\Lambda V_4+\bonyw(A^{(4)}(U;t,x,\xi))V_4+R_1^{(4)}(U)V_4+R_2^{(4)}(U)U\big),
\end{equation}
where  $A^{(4)}(U;t,x,\xi)$ is a (R,R,P) matrix of symbols in $\Sigma\Gamma^2_{K,\rho+4,1}[r,N]\otimes\mathcal{M}_2(\CCC)$ and it has the form
\begin{equation}\label{formaA4}
A^{(4)}(U;t,x,\xi)=A^{(3)}_2(U;t)(\ii\xi)^2+\sum_{j=-(\rho-1)}^0A_j^{(4)}(U;t,x,\xi)\\
\end{equation}
where the diagonal matrix $A^{(3)}_2(U;t)$ is $x$-independent and it is  defined in \eqref{diagonale-ordine-2-doppio},  $A_j^{(4)}(U;t,x,\xi)$ are diagonal matrices belonging to the class $\Sigma\Gamma^{j}_{K,\rho+4,1}[r,N]\otimes\mathcal{M}_2(\CCC)$ for $j=-(\rho-1),\ldots, 0$. The operators $R^{(4)}_1(U)$ and $R^{(4)}_2(U)$ are (R,R,P) and belong to the class $\Sigma\mathcal{R}^{-\rho+m}_{K,\rho+4,1}[r,N]\otimes\mathcal{M}_2(\CCC)$  for some  $m\in\NNN$ depending on $N$.
\end{itemize}
\end{prop}
\begin{proof}
Let $s(U;t,x)$ be a function in $\Sigma\mathcal{F}_{K,\rho+3,1}[r,N]$ to be chosen later, define the map
\begin{equation}\label{exp}
\Phi_4(U)[\cdot]:=\bonyw\left(\begin{matrix}e^{s(U;t,x)}& 0\\ 0& \ol{e^{s(U;t,x)}}\end{matrix}\right)[\cdot].
\end{equation}
Suppose moreover that $\Phi_4(U)$  in \eqref{exp} is a (R,R,P) map.
Since $s(U;t,x)$ is in $\Sigma\mathcal{F}_{K,\rho+3,1}[r,N]$ then by Taylor expanding the exponential function one gets  that the symbol $e^{s(U;t,x)}-1$ is in $\Sigma\mathcal{F}_{K,\rho+3,1}[r,N]$, in  particular the map $\Phi_4(U)$ satisfies items $(i)$ and $(ii)$ 
in the statement. The matrix  in \eqref{exp} is invertible, therefore the map
\begin{equation}\label{expinv}
\Psi_4(U)[\cdot]:=\bonyw\Big(\begin{matrix}e^{-s(U;t,x)}& 0\\ 0& {e^{-\ol{s(U;t,x)}}}\end{matrix}\Big)[\cdot]
\end{equation}
is an approximate inverse $\Psi_4(U)$ for the map $\Phi_4(U)$ (i.e. satisfying the conditions in the items $(i)$ and $(ii)$ of the statement). 
To prove this last claim one has to argue exactly as done in the proof of Prop. \ref{diagomax}.

The function $V_4:=\Phi_4(U)V_3$ solves the equation
\begin{equation}\label{tir}
\begin{aligned}
\partial_t &V_4= (\partial_t\Phi_4(U))[\Psi_4(U)V_4]+\widetilde{R}(U)U
\\ &+\Phi_4(U)\ii E\big[\Lambda \Psi_4(U)V_4+\bonyw(A^{(3)}(U;t,x,\xi))\Psi_4(U)V_4+R^{(3)}_1(U)\Psi_4(U)V_4+R^{(3)}_2(U)U\big],
\end{aligned}
\end{equation}
where $\widetilde{R}(U)U$ is equal to
\begin{equation*}
\big(\partial_t \Phi_4(U)\big)R(U)U+\Phi_4(U)\ii E\left(\big[\Lambda+\bonyw(A^{(3)}(U;t,x,\xi))\big]R(U)U+R^{(3)}_1(U)R(U)U\right),
\end{equation*}
where $R(U)$ is the (R,R,P) smoothing operator in $\Sigma\mathcal{R}^{-\rho}_{K,\rho+4,1}[r,N]\otimes\mathcal{M}_2(\CCC)$ such that $\Psi_4(U)\circ\Phi_4(U)V_4=V_4+R(U)U$ and the matrix $A^{(3)}(U;t,x,\xi)=\sum_{j=-(\rho-1)}^2A_j^{(3)}(U;t,x,\xi)$ is defined in the statement of Prop. \ref{egorov}.
Therefore  $\widetilde{R}(U)$ is a (R,R,P)  smoothing remainder in the class $\Sigma\mathcal{R}^{-\rho+m}_{K,\rho+4,1}[r,N]\otimes\mathcal{M}_2(\CCC)$ thanks to Lemma \ref{lemmamaxcompo}, Prop. \ref{composizioniTOTALI} and to the fact that ${R}(U)$ is a (R,R,P)  smoothing remainder.

The term $(\partial_t\Phi_4(U))\Psi_4(U)$ is of order $0$ thanks to Lemma \ref{lemmatempo}, and Prop. \ref{composizioniTOTALI}; moreover  it is reversible, parity and reality preserving thanks to  Lemma \ref{lemmamaxcompo} since $\Psi_4(U)$ is (R,R,P) and $(\partial_t\Phi_4(U))$  is reversible, parity and reality preserving.

The term $\Phi_4(U)\ii E R^{(3)}_2(U)U$ is reversible, reality and parity preserving thanks to Lemma \ref{lemmamaxcompo}.

Using symbolic calculus (Prop. \ref{composizioniTOTALI}) one can prove that the term of order one in \eqref{tir} is the following 
\begin{equation}\label{omologica4}
\big(2s_x(U;t,x)(1+a_2^{(3)}(U;t))+a_1^{(3)}(U;t,x)\big)(\ii\xi),
\end{equation}
therefore we have to choose the function $s(U;t,x)$ as
\begin{equation*}
s(U;t,x)=-\partial_x^{-1}\Big(\frac{a_1^{(3)}(U;t,x)}{2(1+a_2^{(3)}(U;t))}\Big).
\end{equation*}
Note that the the function $s(U;t,x)$ is well defined since $a_1^{(3)}(U;t,x)$ is an odd function in $x$ (therefore its mean is zero) and the denominator stays far away from zero since $r_0$ is small enough. With this choice the map $\Phi_4(U)$ defined in \eqref{exp} is (R,R,P) and therefore the ansatz made at the beginning of the proof is correct. Furthermore the term $\Phi_4(U)\ii E\big[\Lambda\Psi_4(U) V_4+\bonyw(A^{(3)}(U;t,x,\xi))\Psi_4(U)V_4]$ in \eqref{tir} is equal to $\ii E[\Lambda V_4+\bonyw(\widetilde{A}(U;t,x,\xi))V_4+Q(U)V_4]$, where $Q(U)$ is a (R,R,P) smoothing remainder in $\Sigma\mathcal{R}^{-\rho}_{K,\rho+4,1}[r,N]\otimes\mathcal{M}_2(\CCC)$ and 
\begin{equation*}
\widetilde{A}(U;t,x,\xi)=A_2^{(3)}(U;t)(\ii\xi)^2+\sum_{j=-(\rho-1)}^0\widetilde{A}_j(U;t,x,\xi),
\end{equation*} 
is a (R,R,P) diagonal matrix of symbols such that $\widetilde{A}_j(U;t,x,\xi)$ is in $\Sigma\Gamma^j_{K,\rho+4,1}[r,N]\otimes\mathcal{M}_2(\CCC)$ thanks to Prop. \ref{composizioniTOTALI} and Lemma \ref{lemmabello}.
\end{proof}
\subsection{Reduction to constant coefficients: lower order terms}\label{diagosecondord5}
Here we reduce to constant coefficients all the symbols from the order $0$ to the order $\rho-1$  of the matrix $A^{(4)}(U;t,x,\xi)$ in \eqref{formaA4}.
 \begin{prop}
There exist $s_0>0$, $r_0>0$, such that for any $s\geq s_0$, any $0<r\leq r_0$ and any $U\in B^K_s(I,r)$ solution of \eqref{sistemainiziale} the following holds. 
There exist two  (R,R,P)-maps 
$\Phi_5(U)[\cdot],\, \Psi_5(U)[\cdot] :C^{K-2\rho-4}_{*\RRR}(I,\hcic^s)\rightarrow C^{K-2\rho-4}_{*\RRR}(I,\hcic^s),$
satisfying the following
\begin{enumerate}
\item[(i)] there exists a constant $C$ depending on $s$, $r$ and $K$ such that
\begin{equation}\label{stime-descent-5}
\begin{aligned}
\norm{\Phi_5(U)V}{K-2\rho-4,s}, \, \norm{\Psi_5(U)V}{K-2\rho-4,s}&\leq \norm{V}{K-2\rho-4,s}\big(1+C\norm{U}{K,s_0}\big)\\
\end{aligned}
\end{equation} 
for any $V$ in $C^{K-2\rho-4}_{*\RRR}(I,\hcic^s)$;
\item[(ii)] $\Phi_5(U)[\cdot]-\uno$ and  $\Psi_5(U)[\cdot]-\uno$ belong to the class $\Sigma\mathcal{M}_{K,2\rho+4,1}[r,N]\otimes\mathcal{M}_2(\CCC)$; $\Psi_5(U)[\Phi_5(U)[\cdot]]-\uno$ is a smoothing operator in the class $\Sigma\mathcal{R}^{-\rho}_{K,2\rho+4,1}[r,N]\otimes\mathcal{M}_2(\CCC)$;
\item[(iii)] the function $V_5=\Phi_{5}(U)V_4$ (where $V_4$ is the solution of \eqref{sistemafinale-4}) solves the system
\begin{equation}\label{sistemainiziale-5}
\partial_t V_5 = \ii E\big(\Lambda V_5+\bonyw(A^{(5)}(U;t,\xi))V_5+R_1^{(5)}(U)V_5+R^{(5)}_{2}(U)U\big),
\end{equation}
where $\Lambda$ is defined in \eqref{DEFlambda}, $A^{(5)}(U;t,\xi)$ is a (R,R,P) diagonal and constant coefficient in $x$ matrix in \linebreak $\Sigma\Gamma_{K,2\rho+4,1}^2[r,N]\otimes\mathcal{M}_2(C)$ of the form
\begin{equation*}
A^{(5)}(U;t,\xi)=A^{(3)}_2(U;t)(\ii\xi)^2+A^{(5)}_0(U;t,\xi),
\end{equation*}
with $A^{(3)}_2(U;t)$ of Prop. \ref{egorov} and $A^{(5)}_0(U;t,\xi)$ in $\Sigma\Gamma_{K,2\rho+4,1}^0[r,N]\otimes\mathcal{M}_2(\CCC)$; the operators $R^{(5)}_1(U)$ and $R^{(5)}_2(U)$ are (R,R,P) smoothing remainders in the class $\Sigma\mathcal{R}^{-\rho+m}_{K,2\rho+4,1}[r,N]\otimes\mathcal{M}_2(\CCC)$ for some $m$ in $\NNN$ depending on $N$.
\end{enumerate}
\end{prop}
\begin{proof}
We first construct a map which conjugates to constant coefficient the term of order $0$ in \eqref{sistemafinale-4} and \eqref{formaA4}.
 Consider a symbol $n_0(U;x,\xi)$ in $\Sigma\Gamma^{-1}_{K,\rho+4,1}[r,N]$ to be determined later and define
\begin{equation*}
\Phi_{5,0}(U):=\uno+\bonyw(N_0(U;t,x,\xi)):=\uno+\bonyw\left(\begin{matrix} n_0(U;t,x,\xi) & 0\\ 
																														0&     \ol{n_0(U;t,x,-\xi)}\end{matrix}\right).
\end{equation*}
Suppose moreover that the map $\Phi_{5,0}(U)$ defined above is (R,R,P). It is possible to construct an approximate inverse of the map above (i.e. satisfying items (i) and (ii) of the statement) of the form 

$$\Psi_{5,0}(U)=\uno-\bonyw(N_0(U;t,x,\xi))+\bonyw\big((N_0(U;t,x,\xi))^2\big)+\bonyw({\widetilde{N}_0(U;t,x,\xi)})$$

 proceeding as done in the proof of Prop. \ref{diago-lower} by choosing a suitable matrix of symbols ${\widetilde{N}_0(U;t,x,\xi)}$ in the class  $\Sigma\Gamma^{-3}_{K,\rho+4,1}[r,N]\otimes\mathcal{M}_2(\CCC)$.
Let $R(U)$ in $\Sigma\mathcal{R}^{-\rho}_{K,\rho+4,1}[r,N]\otimes\mathcal{M}_2(\CCC)$ be the smoothing remainder such that one has $\Psi_{5,0}(U)[\Phi_{5,0}(U)[\cdot]]-\uno=R(U)$. Then the function $V_{5,0}=\Phi_{5,0}(U)V_4$ solves the following problem
\begin{equation}\label{sistemafinale-50}
\begin{aligned}
\partial_t V_{0,5}=&(\partial_t \Phi_{5,0}(U))\Psi_{5,0}V_{5,0}
+\Phi_{5,0}(U)\ii E\big(\Lambda+\bonyw(A^{(4)}(U;t,x,\xi))\big)\Psi_{5,0}(U)V_{5,0}\\
+&\Phi_{5,0}(U)\ii E R^{(4)}_1(U)\Psi_{5,0}(U)V_{5,0}
+\Phi_{5,0}(U)\ii E R^{(4)}_2(U)U +\widetilde{R}(U)U,
\end{aligned}
\end{equation}
where $A^{(4)}(U;t,x,\xi)$, $R^{(4)}_1(U)$ and $R^{(4)}_2(U)$ are the ones of equation \eqref{sistemafinale-4}, while $\widetilde{R}(U)$ is the operator
\begin{equation*}
\Big[(\partial_{t}\Phi_{5,0}(U))+\Phi_{5,0}(U)\ii E \big(\Lambda+\bonyw(A^{(4)}(U;t,x,\xi))+\Phi_{5,0}(U)\ii E R^{(1)}_4(U)\Big]R(U).
\end{equation*}
The operator $\widetilde{R}(U)$ belongs to the class $\Sigma\mathcal{R}^{-\rho+m'}_{K,\rho+5,1}[r,N]\otimes\mathcal{M}_2(\CCC)$ for some $m'\in\NNN$ thanks to Prop. \ref{composizioniTOTALI}, moreover it is reversible, parity and reality preserving by Lemma \ref{lemmamaxcompo}. The first summand in the r.h.s. of \eqref{sistemafinale-50} is equal to $\bonyw{(\partial_t N_0(U;t,x,\xi))}\circ \bonyw(\uno-N_0(U;t,x,\xi)+N_0(U;t,x,\xi)^2+\widetilde{N}_0(U;t,x,\xi))$ therefore by Lemmata \ref{lemmatempo}, \ref{lemmabello} and Prop. \ref{composizioniTOTALI}  can be decomposed as the sum of a para-differential operator of order $-1$ and a smoothing remainder, both of them reversible, parity and reality preserving. The third and the fourth summands in  \eqref{sistemafinale-50} are (R,R,P) remainders in the class $\Sigma\mathcal{R}^{-\rho+m'}_{K,2\rho+4,1}[r,N]\otimes\mathcal{M}_2(\CCC)$ by Lemma \ref{lemmamaxcompo} and Prop. \ref{composizioniTOTALI}.

The remaining term $\Phi_{5,0}(U)\ii E\big(\Lambda+\bonyw(A^{(4)}(U;t,x,\xi))\big)\Psi_{5,0}(U)V_{5,0}$ is equal to 
\begin{equation}\label{drago}
\begin{aligned}
\ii & E\big(\Lambda+\bonyw(A^{(3)}_2(U;t)(\ii\xi)^2)\big)V_{5,0}+
 \left[\bonyw(N(U;t,x,\xi),\ii E\bonyw((\uno+A_2(U;t))(\ii\xi)^2)) \right]_{-}V_{0,5}\\
&+\ii E\bonyw( A^{(4)}_0(U;t,x,\xi))V_{0,5}\\
&- \bonyw(N_0(U;t,x,\xi))\circ\bonyw(\ii E(\uno+A_2(U;t))(\ii\xi)^2)\circ\bonyw(N_0(U;t,x,\xi))V_{0,5}\\
&+\bonyw(\ii E(\uno+A_2(U;t))(\ii\xi)^2)\circ\bonyw(N_0(U;t,x,\xi)^2)V_{0,5}
\end{aligned}
\end{equation}
up to operators of order $-1$ coming from the compositions with $\bonyw(\widetilde{N}(U;t,x,\xi))$. Every operator here can be assumed to be reversible, parity and reality preserving thanks to Lemmata \ref{lemmamaxcompo} and \ref{lemmabello}. The last two summands in \eqref{drago} cancel out up to a (R,R,P) operator of order $-1$ thanks to Prop. \ref{composizioniTOTALI}. In order to reduce to constant coefficient the term of order $0$ in \eqref{drago} we develop the commutator  $[\cdot,\cdot]_{-}$ and we  choose $n_0(U;t,x,\xi)$ in such a way that the following equation is satisfied 
\begin{equation}\label{drago2}
2(n_0(U;t,x,\xi))_x(1+a_2^{(3)}(U;t)(\ii\xi))+a_0^{(4)}(U;t,x,\xi)=\frac{1}{2\pi}\int_{\TTT}a_0^{(4)}(U;t,x,\xi)dx;
\end{equation}
proceeding  as done in the proof of Prop. \ref{diago-lower} we choose the symbol $n_0(U;t,x,\xi)$ as follows:
\begin{equation*}
\begin{aligned}
n_0(U;t,x,\xi)&=\frac{\partial_x^{-1}\Big(\frac{1}{2\pi}\int_{\TTT}a_0^{(4)}(U;t,x,\xi)dx-a_0^{(4)}(U;t,x,\xi)\Big)}{2(1+a_{2}^{(3)}(U;t))}\gamma(\xi);\\
\gamma(\xi):&=\left\{\begin{matrix}  \frac{1}{\ii\xi}  &   |\xi|\geq1/2,   \\
 \mbox{odd continuation of class}\,\, C^{\infty} &    |\xi|\in [0,1/2). \\ \end{matrix}\right.
\end{aligned}
\end{equation*}
The symbol above is well defined since the denominator stays far away from zero since the function $U$ is small, the numerator is  well defined too since it is the periodic primitive of a zero mean function; moreover it is parity preserving and reversibility preserving, therefore the ansatz made at the beginning of the proof is satisfied. Therefore we have reduced the system to the following
\begin{equation*}
\partial_t V_{5,0}=\ii E\left(\Lambda+\bonyw(A^{(5,0)}(U;t,x,\xi))V_{5,0}+R_1^{(5,0)}(U)V_{5,0}+R^{(5,0)}_2(U)U\right),
\end{equation*}
where $R_{1}^{(5,0)}(U)$ and $R_{2}^{(5,0)}(U)$ are (R,R,P) smoothing remainder in the class $\Sigma\mathcal{R}^{-\rho+m'}_{K,\rho+4,1}[r,N]\otimes\mathcal{M}_2(\CCC)$, the matrix of symbols $A^{(5,0)}(U;t,x,\xi)$  is in  $\Sigma\Gamma^2_{K,\rho+5,1}[r,N]\otimes\mathcal{M}_2(\CCC)$ and it has the form
\begin{equation*}
A^{(5,0)}(U;t,x,\xi)=A^{(3)}_2(U;t)(\ii\xi)^2+A_0^{(5,0)}(U;t,\xi)+A_{-1}^{(5,0)}(U;t,x,\xi),
\end{equation*}
with $A^{(3)}_2(U;t)$ given in  Prop. \ref{egorov}, $A_0^{(5,0)}(U;t,\xi)$  equal to $(\frac{1}{2\pi}\int A_0^{(4)}(U;t,x,\xi)dx)\gamma(\xi)$ and $A_{-1}^{(5,0)}(U;t,x,\xi)$ a matrix of symbols in $\Sigma\Gamma^{-1}_{K,\rho+5,1}[r,N]\otimes\mathcal{M}_2(\CCC)$. 

Suppose now that there exist $j+1$, $j\geq 0$, (R,R,P) maps $\Phi_{5,0}(U),\ldots,\Phi_{5,j}(U)$ maps such that the function $V_{5,j}:=\Phi_{5,0}(U)\circ\cdots\circ\Phi_{5,j}(U)V_4$ solves the problem
\begin{equation}\label{induttivo}
\partial_t V_{5,j}=\ii E\big(\Lambda V_{5,j}+\bonyw(A^{(5,j)}(U;t,x,\xi))V_{5,j}+R^{(5,j)}_1(U)V_{5,j}+R^{(5,j)}_2(U)U\big)
\end{equation}
where $R^{(5,j)}_1(U)$, $R^{(5,j)}_2(U)$ are (R,R,P) smoothing remainders in the class $\Sigma\mathcal{R}^{-\rho+m'}_{K,\rho+5+j,1}[r,N]\otimes\mathcal{M}_2(\CCC)$ and  where $A^{(5,j)}(U;t,x,\xi)$ is a (R,R,P) diagonal matrix of symbols in $\Sigma\Gamma^2_{K,\rho+5+j,1}[r,N]\otimes\mathcal{M}_2(\CCC)$ of the form 
\begin{equation*}
A^{(5,j)}(U;t,x,\xi)= A^{(3)}_2(U;t)(\ii\xi)^2+\sum_{\ell=-j}^0 A^{(5,j)}_{\ell}(U;t,\xi) + A^{(5,j)}_{-j-1}(U;t,x,\xi),
\end{equation*}
with $A^{(5,j)}_{\ell}(U;t,\xi)$ in $\Sigma\Gamma^{\ell}_{K,\rho+5+j,1}[r,N]\otimes\mathcal{M}_2(\CCC)$  and  constant in $x$ for $\ell=-j,\ldots, 0$, while $A^{(5,j)}_{-j-1}(U;t,x,\xi)$ is in $\Sigma\Gamma^{-j-1}_{K,\rho+5+j,1}[r,N]$ and may depend on $x$. We explain how to construct a map $\Phi_{5,j+1}(U)$ which put to coefficient in $x$ the term of order $-j-1$. Let $n_{j+1}(U;t,x,\xi)$ be a symbol in $\Sigma\Gamma^{-j-2}_{K,\rho+5+j,1}[r,N]$ and consider the map
\begin{equation*}
\Phi_{5,j+1}(U):=\uno+\bonyw(N_{j+1}(U;t,x,\xi)):=\uno+\bonyw\left(\begin{matrix}n_{j+1} (U;t,x,\xi) & 0 \\
0 & \ol{n_{j+1}(U;t,x,-\xi)}
\end{matrix}\right).
\end{equation*}
Arguing as done in the proof of Prop. \ref{diago-lower} one  obtains the approximate inverse of the map above $\Psi_{5,j}(U)=\uno-\bonyw(N_{j+1}(U;t,x,\xi))$ modulo lower order terms. The same discussion made at the beginning of the proof, 
concerning the conjugation through the map $\Phi_{5,0}(U)$,
 shows that the
function $V_{5,j+1}:=\Phi_{5,j+1}(U)V_{5,j}$ solves the problem
\begin{equation}\label{induttivo1}
\partial_t V_{5,j+1}=\ii E\big(\Lambda V_{5,j+1}+\bonyw(A^{(5,j+1)}(U;t,x,\xi))V_{5,j+1}+R^{(5,j+1)}_1(U)V_{5,j+1}+R^{(5,j+1)}_2(U)U\big),
\end{equation}
where $R^{(5,j+1)}_1(U)$ and $R^{(5,j+1)}_2(U)$ are smoothing remainders in the class $\Sigma\mathcal{R}^{-\rho+m'}_{K,\rho+5+j,2}[r,N]\otimes\mathcal{M}_2(\CCC)$ and where  $A^{(5,j+1)}(U;t,x,\xi)$ has the form
\begin{equation*}
A^{(3)}_2(U;t)(\ii\xi)^2+\sum_{\ell=-j}^0 A^{(5,j)}_{\ell}(U;t,\xi) + A^{(5,j)}_{-j-1}(U;t,x,\xi)+2\big(N_{j+1}(U;t,x,\xi),\big)_x(1+A_2^{(3)}(U;t)(\ii\xi)).
\end{equation*}
The equation we need to solve is
\begin{equation*}
2(n_{j+1}(U;t,x,\xi))_x(1+a_2^{(3)}(U;t)(\ii\xi))+a_{-j-1}^{(5,j)}(U;t,x,\xi)=\frac{1}{2\pi}\int_{\TTT}a_{-j-1}^{(5,j)}(U;t,x,\xi)dx,
\end{equation*}
which has the same structure of \eqref{drago2} and hence one can define the symbol $n_{j+1}(U;t,x,\xi)$ as done above.

To conclude the proof we define the maps $\Phi_{5}(U):=\Phi_{5,0}(U)\circ\cdots\circ\Phi_{5,\rho-1}(U)$ and $\Psi_{5}(U):=\Psi_{5,\rho-1}(U)\circ\cdots\circ\Psi_{5,0}(U)$.
\end{proof} 
At this point we can prove Theorem \ref{regolarizza}.
\begin{proof}[proof of Theorem \ref{regolarizza}]
It is enough to define $\Phi(U):=\Phi_5(U)\circ\Phi_4(U)\circ\Phi^{\star}_U\circ\Phi_2(U)\circ\Phi_1(U)$ and $\Psi(U):=\Psi_1(U)\circ\Psi_2(U)\circ(\Phi^{\star}_U)^{-1}\circ\Psi_4(U)\circ\Psi_5(U)$.
\end{proof}

\section{Proof of the main theorem}\label{sezBNF}

The aim of this section is to prove the following Theorem which, together with Theorem \ref{paralineariza},
implies  Theorem \ref{teototale}.

\begin{theo}\label{longgun}
Fix $N>0$ and assume $M\geq N$ (see \eqref{potenziale1}), $K\in\NNN$,  $\rho\in \NNN$ such that $K\gg\rho\gg N$ 
and consider system \eqref{sistemainiziale}. 
There is a zero measure set $\NN\subseteq \calO$ such that
for any $\vec{m}$ outside the set $\NN$ 
and if $\rho>0$ is large enough there is $s_0>0$ such that for any
$s\geq s_0$ there are $r_0,c,C>0$ 
such that for any $0\leq r\leq r_0$ the following holds. 
For all $U_0\in {\bf H}^{s}_{e}$ with $\|U_0\|_{\hcic^{s}}\leq r$,  there is a unique solution $U(t,x)$ of \eqref{sistemainiziale} 
with
\begin{equation}\label{spaziosol}
U\in \bigcap_{k=0}^{K}C^{k}([-T_{r},T_{r}];\hcic^{s-2k}_{e}(\TTT;\CCC^{2})), 
\end{equation}
with $T_{r}\geq c r^{-N}$.
Moreover one has
\begin{equation}\label{iltempo}
\sup_{t\in[-T_r,T_r]}\|\del_{t}^{k}U(t,\cdot)\|_{\hcic^{s-2k}}\leq Cr, \quad 0\leq k\leq K.
\end{equation}
\end{theo}
In order to prove the Theorem above, the first step is to apply to system 
\eqref{sistemainiziale} the 
Theorem \ref{regolarizza} obtaining the system \eqref{problemafinale}.
At this point it  is not evident that the solution $V$ of \eqref{problemafinale} satisfies 
\eqref{spaziosol} and \eqref{iltempo} because 
$\|\bonyw\big({\rm Im}(\mathtt{m}_0(U;t,\x))\big)\|_{\mathcal{L}({\bf H}^{s}, {\bf H}^{s})}$ and
$\|Q_{i}(U)\|_{\mathcal{L}({\bf H}^{s}, {\bf H}^{s+\rho})}$, $i=1,2$ are of size $O(\|U\|_{{\bf H}^{s}})$.
Therefore we need to reduce the size of the latter two quantities 
by applying a  normal form procedure.  
As already explained in the introduction we shall face a small divisors problem.
We impose some non-resonance conditions on the linear frequencies, by using 
the convolution potential in the equation \eqref{NLS}, in subsection \ref{BNFBNFBNF}.
The Birkhoff procedures 
for the latter two  terms are substantially different.

Concerning the term ${\rm Im}(\mathtt{m}_0(U;t,\x))$,
we shall prove, in Theorem \ref{BNFiniziale}, that we may conjugate system 
\eqref{problemafinale} to \eqref{problemafinale2}
whose symbol  is given in \eqref{constCoeff2}.
The symbol $\pro{\mathtt{m}^{(1)}_0}(U;t,\x)$ in \eqref{constCoeff2}
admits  an expansion in homogeneity whose 
firsts $N-1$ terms do not contribute to the growth of Sobolev norms.
This is a consequence of the reversibility and parity preserving structure of the equation and
we check this fact in Lemma \ref{FEI} (see Definition \ref{kernel}).
This is the content of subsection \ref{BNFnormal1}.

At this point we have obtained system 
\eqref{problemafinale2}. In Lemma \ref{FEI2}
we prove that the solution $W$ of \eqref{problemafinale2}
satisfies the energy estimate \eqref{stimanormasob}, where $L_{p}$
are \emph{multilinear forms} defined in Definition \ref{formaggino}.
As last step we construct modified energies (in the sense of \eqref{botte10})
in order to cancel out the terms $L_p$ in  \eqref{stimanormasob}.
This is the content of subsection \ref{formeenergia}.

We now enter in detail and 
we introduce some further notation.  For any $n\in \NNN$, we define
  \begin{equation}\label{proiettori}
  \Pi^{+}_{n}:=\sm{1}{0}{0}{0}
  \Pi_{n}, \qquad   \Pi^{-}_{n}:=\sm{0}{0}{0}{1}
  \Pi_{n}
  \end{equation}
  the composition of the spectral projector
 $\Pi_{n}$, defined in \eqref{spectralpro} with the projector from $\CCC^{2}$ to $\CCC\times \{0\}$ (resp. $\{0\}\times\CCC$).
  For a function $U$ satisfying \eqref{involuzione4}, i.e. of the form $U=(u,\bar{u})^{T}$, 
  the projectors $\Pi_{n}^{\pm}$ act as follows.
  Let $\f_{n}(x)=1/\sqrt{\pi}\cos(nx)$ be the Hilbert basis of the space of even $L^{2}(\TTT;\CCC)$  functions,
  then, if $\hat{u}(n)=\int_{\TTT}u(x)\f_{n}(x)dx$, one has
  \begin{equation}\label{azioneproiettori}
  \Pi_{n}U=\left(\begin{matrix}\hat{u}(n) \\
  \ol{\hat{u}(n)}
  \end{matrix}
  \right)\f_{n}(x), \quad 
    \Pi_{n}^{+}U=\hat{u}(n)e_{+}\f_n(x), \quad \Pi_{n}^{-}U=\ol{\hat{u}(n)}e_{-}\f_n(x), \quad 
    e_{+}=\vect{1}{0},
    \;\;     e_{-}=\vect{0}{1}.
  \end{equation}
  We now give the following definition.
  
  \begin{de}[{\bf Kernel of the Adjoint Action}]\label{kernel}
Fix $p\in\NNN^{*}$, 
$\rho>0$ and consider 
a symbol $a\in \widetilde{\Gamma}_{p}^{2}$. 

\noindent
$(i)$ 
We denote by $\pro{a}(U;t,x,\x)$ the symbol in  $\widetilde{\Gamma}^{2}_{p}$
 defined,  for $n_1,\ldots,n_{p}\in \NNN$, as
\begin{equation*}
\pro{a}\Big(\Pi_{n_1}^{+}U,\ldots,\Pi_{n_{\ell}}^{+}U,\Pi^{-}_{n_{\ell+1}}U,\ldots,
\Pi_{n_{p}}^{-}U;t,x,\x\Big)=
a\Big(\Pi_{n_1}^{+}U,\ldots,\Pi_{n_{\ell}}^{+}U,\Pi_{n_{\ell+1}}^{-}U,\ldots,
\Pi_{n_{p}}^{-}U;t,x,\x\Big),
\end{equation*}
for $p$ even, $\ell=p/2$ and 
\begin{equation}\label{bnf55}
\{n_1,\ldots,n_{\ell}\}=\{n_{\ell+1},\ldots,n_{p}\},
\end{equation}
while for $p$ odd and $0\leq \ell\leq p$, or $p$ even and $0\leq \ell\leq p$ with $\ell\neq p/2$
we set
\[
\pro{a}\Big(\Pi_{n_1}^{+}U,\ldots,\Pi_{n_{\ell}}^{+}U,\Pi_{n_{\ell+1}}^{-}U,\ldots,
\Pi_{n_{p}}^{-}U;t,x,\x\Big)=0.
\]

\noindent
$(ii)$ Let $a\in \Sigma\Gamma^{0}_{K,K',1}[r,N]$ of the form
\[
a(U;t,x,\x)=\sum_{k=1}^{N-1}a_k(U;t,x,\x)+a_{N}(U;t,x,\x),\quad a_{k}\in \widetilde{\Gamma}^{0}_{k}, \;\; a_{N}\in 
\Sigma\Gamma^{0}_{K,K',N}[r,N],
\]
we define the symbol $\pro{a}(U;t,x,\x)$ as
\[
\pro{a}(U;t,x,\x):=\sum_{k=1}^{N-1}\pro{a_k}(U;t,x,\x)+a_{N}(U;t,x,\x).
\]
\noindent
$(iii)$ For a diagonal matrix of symbols  
$A\in\Sigma{\Gamma}_{K,K',1}^{0}[r,N]\otimes\MM_{2}(\CCC)$ 
of the form
\begin{equation*}
A(U;t,x,\x)=\left(
\begin{matrix}
a(U;t,x,\x) & 0\\
0 & \ol{a}(U;t,x,-\x)
\end{matrix}
\right), 
\end{equation*}
we define
\begin{equation}\label{matsim}
\pro{A}(U;t,x,\x):=\left(
\begin{matrix}
\pro{a}(U;t,x,\x) & 0\\ 0 & \pro{\ol{a}}(U;t,x,-\x)
\end{matrix}
\right).
\end{equation}
\end{de}

In the following lemma we consider the  problem
\begin{equation}\label{conserva}
\left\{\begin{aligned}
&\del_{t}Z=\ii E\Big(\Lambda Z+\bonyw(\mathtt{m}_{2}(U)(\ii\x)^{2})Z+
\bonyw\big(\pro{A}(U;t,\x)\big)[Z]\Big),\\
&Z(0,x)=Z_0\in {\bf H}^{s}_e
\end{aligned}\right.
\end{equation}
with $\mathtt{m}_2(U)$  in \eqref{constCoeff}, $A$ being a diagonal 
(R,R,P) matrix of bounded symbols independent of $x$ in  the class
$\Sigma\Gamma_{K,K',1}^{0}[r,N]\otimes\MM_{2}(\CCC)$ and 
$U\in C^K_{*\R}(I,{\bf{H}}_e^{s}(\TTT;\CCC^{2}))\cap B^{K}_{s_0}(I,r)$. We prove that
the (R,R,P) structure of the matrix $A(U;t,\xi)$ (together with the fact that it is constant in $x$) guarantees a  symmetry  which  produces a key cancellation in the energy estimates for the problem \eqref{conserva}, more precisely we show that the multilinear part of the matrix $\pro{A}(U;t,\xi)$ does not contribute to the energy estimates.

\begin{lemma}\label{FEI}
Let $N\in \NNN$, $r>0$, $K'\leq K\in \N$. Using the notation above consider  $Z=Z(t,x)$  the solution of the problem \eqref{conserva}.
Then one has
\begin{equation}\label{siconserva}
\frac{d}{d t}\|Z(t,\cdot)\|_{{ \hcic}^{s}}^{2}
\leq C\|U(t,\cdot)\|_{K',s}^{N}\|Z(t,\cdot)\|^{2}_{{ \hcic}^{s}}.
\end{equation}

\end{lemma}

\begin{proof}
Consider the Fourier multipliers $\langle D\rangle ^{s}:={\rm Op}(\langle \x\rangle^{s})$.
We have that 
\begin{equation}\label{catena}
\begin{aligned}
\frac{d}{dt}\|Z\|_{{ \hcic}^{s}}^{2}&=\Big(\langle D\rangle ^{s}\Big[
\ii E(\Lambda Z+\bonyw(\mathtt{m}_{2}(U)(\ii\x)^{2})+
\bonyw(\pro{A}(U;t,\x))[Z])
\Big]  ,\langle D\rangle ^{s} Z\Big)_{{\bf H}^{0}}\\
&+\Big(\langle D\rangle ^{s} Z,\langle D\rangle ^{s}\Big[
\ii E(\Lambda Z+\bonyw(\mathtt{m}_{2}(U)(\ii\x)^{2})+\bonyw(\pro{A}(U;t,\x))[Z])
\Big]\Big)_{{\bf H}^{0}}\\
\end{aligned}
\end{equation}
where $(\cdot,\cdot)_{{\bf H}^{0}}$ is defined in \eqref{comsca}.
The contribution given by 
 $\Lambda$ and $\bonyw(\mathtt{m}_0(U)(\ii\x)^{2})[\cdot]$ is zero since 
they are independent of $x$ (therefore they commute with $\langle D^s\rangle$) and their symbols are real valued (hence they are self-adjoint on $\hcic^0$ thanks to Remark \ref{aggiungoaggiunto}).
Let us consider the symbol $\pro{A}(U;t,\x)$.
By definition
we have that
\[
\pro{A}(U;t,\x)=\sum_{p=1}^{N-1}\pro{A_{p}}(U,\ldots,U;t,\x)+A_{N}(U;t,\x)
\]
with $A_{p}\in\widetilde{\Gamma}^{0}_{p}$, $p=1,\ldots,N-1$, and $A_{N}\in \Sigma\Gamma^{0}_{K,K',N}[r,N]$.
The contribution of $\bonyw(A_{N}(U;t,\x))[Z]$ in the r.h.s. of \eqref{catena} 
is bounded by the r.h.s. of \eqref{siconserva}.
We show that $\pro{A_{p}}$ are real valued for $p=1,\ldots,N-1$, this implies that, since they do not depend on $x$,
they do not contribute to the sum in \eqref{catena}.
By hypothesis the matrix of symbols $A(U;t,\x)$ is reversibility preserving, i.e. satisfies \eqref{revsimbo2}, therefore, by Lemma \ref{lemmaNUOVO}, we may assume that $A_{p}$ satisfy condition \eqref{revomo1000} for any $p=1,\ldots,N-1$.
Since $A_{p}$ is reality preserving then we can write (see Remark \ref{considerazioni})
\begin{equation}\label{esplicit}
{A}_p(U;t,\x)=\left(
\begin{matrix}
\tilde{a}_{p}(U;t,\x) & 0\\
0 & \ol{\tilde{a}_{p}(U;t,-\x)}
\end{matrix}
\right),
\end{equation}
for some symbol $\tilde{a}_p\in \widetilde{\Gamma}^{m}_{p}$ independent of $x$.
Recalling Def. \ref{kernel}, since $A_p$ is a symmetric function of its arguments, we have, 
for $\ell,p,n_1,\ldots,n_p$
 satisfying the conditions in \eqref{bnf55}, that
 \begin{equation}\label{lunga}
 \begin{aligned}
\pro{A_p}\Big(\Pi_{n_1}^{+}SU,\ldots,\Pi_{n_{\ell}}^{+}SU,&\Pi^{-}_{n_{1}}SU,\ldots,
\Pi_{n_{\ell}}^{-}SU;t,\x\Big)S\\
&=S
\pro{A_p}\Big(\Pi_{n_1}^{+}U,\ldots,\Pi_{n_{\ell}}^{+}U,\Pi^{-}_{n_{1}}U,\ldots,
\Pi_{n_{\ell}}^{-}U;t,\x\Big).
\end{aligned}
\end{equation}
We recall that $\Pi_{n}^{+}SU={\Pi_{n}^{-}U}$ using \eqref{azioneproiettori}.
On the component $\pro{\tilde{a}_{p}}$ the condition \eqref{lunga} reads
\begin{equation}\label{realere}
\begin{aligned}
\pro{\tilde{a}_p}\Big({\Pi_{n_1}^{-}U},\ldots,& {\Pi_{n_\ell}^{-}U},{\Pi_{n_{1}}^{+}U}, \ldots,
{\Pi_{n_{\ell}}^{+}U};\x\Big)\\
&=\ol{\pro{\tilde{a}_p}\Big(\Pi_{n_1}^{+}U,\ldots, \Pi_{n_\ell}^{+}U,\Pi_{n_{1}}^{-}U, \ldots,
\Pi_{n_{\ell}}^{-}U;-\x\Big)}.
\end{aligned}
\end{equation}
The condition \eqref{parisimbo1} (which holds since $A_{p}$ is parity preserving) implies that 
$\tilde{a}_{p}(U,\ldots, U;t,\x)$ is even in  $\x$ since it does not depend on $x$. 
Therefore by symmetry we deduce that
$\pro{\tilde{a}_p}\Big({\Pi_{n_1}^{+}}U,\ldots, {\Pi_{n_\ell}^{+}U},{\Pi_{n_{1}}^{-}U}, \ldots,
{\Pi_{n_{\ell}}^{-}U};t,\x\Big)$
is real valued.
This concludes the proof.
\end{proof}

The first important result is the following. 
 
 \begin{theo}[{\bf Normal form 1}]\label{BNFiniziale}
 Let $N,\rho,K>0$ as in Theorem \ref{regolarizza}. 
There exists  $\NN\subset \calO$  with zero  measure
such that
for any $\vec{m}\in \calO\setminus \NN$
the following holds. 
There exist $K''>0$ such that $K':=2\rho+4<K''\ll K$, $s_0>0$, $r_0>0$ (possibly different from the ones given by Theorem \ref{regolarizza}) such that, for any $s\geq s_0$, $0<r\leq r_0$ and any $U\in B^K_s(I,r)$ solution even in $x\in \TTT$ 
of \eqref{sistemainiziale} the following holds.
There is an invertible
  (R,R,P)-map
  $\Theta(U)[\cdot] \; :\;  C^{K-K''}_{*\R}(I,{\bf{H}}^{s}(\TTT;\CCC^{2})) \rightarrow C^{K-K''}_{*\R}(I,{\bf{H}}^{s}(\TTT;\CCC^{2})),
$
  satisfying the following:
  \begin{enumerate}
  \item[(i)] there exists a constant $C$ depending n $s,r$ and $K$ such that
\begin{equation}\label{mappeprop}
\begin{aligned}
\|\Theta(U)[V]\|_{K-K'',s}, \|(\Theta(U))^{-1}[V]\|_{K-K'',s}&\leq \|V\|_{K-K'',s}(1+C\|U\|_{K,s_0}) ,\\
\end{aligned}
\end{equation}
for any $ V\in C^{K-K''}_{*\R}(I,{\bf{H}}^{s}(\TTT;\CCC^{2}))$;
\item[(ii)] $\Theta(U)-\uno$ and $(\Theta(U))^{-1}-\uno$ belong to the class $\Sigma\MM_{K,K'',1}[r,N]\otimes\MM_{2}(\CCC)$;

\item[(iii)] the function $W=\Theta(U)[V]$, where $V$ solves \eqref{problemafinale},
 satisfies
\begin{equation}\label{problemafinale2}
\del_{t}W=\ii E\big(\Lambda W+\bonyw(L_1(U;t,\x))[W]+Q_1^{(1)}(U;t)[W]+Q^{(1)}_{2}(U;t)[U]\big)\\
\end{equation}
where 
$Q^{(1)}_1,Q^{(1)}_{2}\in \Sigma\RR^{-\rho+m_1}_{K,K'',1}[r,N]\otimes\MM_{2}({\CCC})$,
for some  $m_1>0$ depending on $N$ (larger than $m$ in Theorem \ref{regolarizza}),
 are (R,R,P)-operators  
and $L_1(U;t,\x)$ is a (R,R,P)-matrix 
in 
 $\Sigma\Gamma^{2}_{K,K'',1}[r,N]\otimes \MM_{2}(\CCC)$
 with constant coefficients in $x\in \TTT$
and
which 
has the form  (recalling Def. \ref{kernel})
\begin{equation}\label{constCoeff2}
L_1(U;t,\x):=\left(
\begin{matrix}
\mathtt{m}^{(1)}(U;t,\x) & 0\\ 0 & \ol{\mathtt{m^{(1)}}(U;-\x)}
\end{matrix}
\right), \quad \mathtt{m}^{(1)}(U;t,\x)=\mathtt{m}_{2}(U)(\ii\x)^{2}+\pro{\mathtt{m}^{(1)}_0}(U;t,\x),
\end{equation}
where $\mathtt{m}_{2}(U)$ is 
given in \eqref{constCoeff} and $\mathtt{m}^{(1)}_0(U;t,\x)\in \Sigma\Gamma^{0}_{K,K'',1}[r,N]$. 
\end{enumerate}
 \end{theo}

\subsection{Non-resonance conditions}\label{BNFBNFBNF}

For $M\in \NNN$ and  $\vec{m}=(m_1,\ldots,m_{M})\in \calO:=[-1/2,1/2]^{M}$ we define,
recalling
\eqref{potenziale1}, \eqref{NLS1000},
for any $N\leq M$ and $0\leq \ell\leq N$  the function  
\begin{equation}\label{smalldiv}
\begin{aligned}
\psi_{N}^{\ell}(\vec{m},\vec{n})&=\lambda_{n_1}+\ldots+\lambda_{n_{\ell}}-\lambda_{n_{\ell+1}}-\ldots-\lambda_{n_{N}}\\
&=\sum_{j=1}^{\ell}(\ii n_{j})^{2}-\sum_{j=\ell+1}^{N}(\ii n_{j})^{2}
+\sum_{k=1}^{M}m_{k}\left(
\sum_{j=1}^{\ell}\frac{1}{\langle n_j\rangle^{2k+1}}-\sum_{j=\ell+1}^{N}\frac{1}{\langle n_j\rangle^{2k+1}}
\right),
\end{aligned}
\end{equation}
where $\vec{n}=(n_1,\ldots,n_N)\in\mathbb{N}^N$ with the convention that $\sum_{j=m}^{m'}a_{j}=0$
when $m>m'$. 
We have the following lemma.

\begin{prop}[{\bf Non resonance condition}]\label{stimemisura}
 There exists $\NN\subset \calO$ with zero Lebesgue measure such that, for any $\vec{m}\in \calO\setminus \NN$,
there exist $\g,N_0>0$ such that the inequality
\begin{equation}\label{stimadalbasso}
 |\psi_{N}^{\ell}(\vec{m},\vec{n})|\geq \g \max(\langle n_1\rangle ,\ldots, \langle n_{N}\rangle)^{-N_0},
 \end{equation}
 holds true for any $\vec{n}=(n_1,\ldots, n_N)\in \NNN^{N}$ if $N$ is odd. In the case that $N$ is even the \eqref{stimadalbasso} holds true if $\ell\neq N/2$ for any $\vec{n}\in \NNN^{N}$; 
 if $\ell=N/2$ the condition \eqref{stimadalbasso} holds true 
  for any $\vec{n}$ in $\NNN^N$ such that
 \begin{equation}\label{noncoppie}
 \{n_1,\ldots,n_{\ell}\}\neq \{n_{\ell+1},\ldots, n_{N}\}.
 \end{equation}
 \end{prop}
\begin{proof}
First of all we show that, if  $N,\ell$, $\vec{n}$  
are as in the statement of the proposition, the function $\psi_{N}^{\ell}(\vec{m},\vec{n})$
is not identically zero as function of $\vec{m}$.
We can write
\begin{equation*}
\psi_{N}^{\ell}(\vec{m},\vec{n})=a_0^{(\ell)}(\vec{n})+\sum_{k=1}^{N}m_{k}a_{k}^{(\ell)}(\vec{n})+\sum_{k=N+1}^{M}m_{k}a_{k}^{(\ell)}(\vec{n}),
\end{equation*}
where
\begin{equation}\label{akl}
a_{0}^{(\ell)}(\vec{n}):=\sum_{j=1}^{\ell}(\ii n_{j})^{2}-\sum_{j=\ell+1}^{N}(\ii n_{j})^{2}; \quad a_{k}^{(\ell)}(\vec{n}):=
\sum_{j=1}^{\ell}\frac{1}{\langle n_j\rangle^{2k+1}}-\sum_{j=\ell+1}^{N}\frac{1}{\langle n_j\rangle^{2k+1}}, \,\, k\geq 1
\end{equation}
for $0\leq \ell\leq N$.
 We show that there exists at least one non zero coefficient $a_{k}^{(\ell)}(\vec{n})$ for $1\leq k\leq N$.
 
Let $q$ in $\NNN^{*}$  such that there are $N_1,\ldots,N_{q}$ in $\NNN^{*}$ satisfying 
$N_1+\ldots+ N_{q}=N$ and 
\begin{equation*}
\set{n_1,\dots,n_N}=\set{n_{1,1},\ldots, n_{1,N_1},\ldots, n_{q,1}\ldots, n_{q,N_{q}}}
\end{equation*}
where
\begin{equation*}
\begin{cases}
&n_{j,i_1}=n_{j,i_2} \,\forall \,\, i_1, \, i_2\in \set{1,\ldots, N_j}, \forall\, \, j\in \set{1,\ldots,q},\\
& n_{j,1}\neq n_{i,1} \, \forall \,\, i\neq j.
\end{cases}
\end{equation*}
Note that, since $\jap{x}=\sqrt{1+x^2}$ for $x\in\RRR$, then $\jap{n_{j,1}}\neq \jap{n_{i,1}}$ for any $i\neq j$.
According to this notation
the element in \eqref{akl} can be rewritten as
\begin{equation}
a_k^{(\ell)}(\vec{n})=\sum_{j=1}^{q}\frac{1}{\jap{n_{j,1}}^{2k+1}}(N_j^+-N_j^{-}),
\end{equation}
where $N_j^+$, resp. $N_j^-$, is the number of times that the term $\jap{n_{j,1}}^{2k+1}$ appears in the sum in equation \eqref{akl} with sign $+$, resp. with sign $-$, and hence $N_j^++N_{j}^{-}=N_j$. 
Note that if $N_j^+-N_j^{-}=0$ for any $j=1,\ldots,q$, then the condition \eqref{noncoppie}   is violated.


Define the $(N\times q)$-matrix
\[
A_{q}(\vec{n}):=\left(
\begin{matrix}
\frac{1}{\langle n_{1,1}\rangle^{3}} & \ldots & \ldots & \frac{1}{\langle n_{q,1}\rangle^{3}}\\
\frac{1}{\langle n_{1,1}\rangle^{5}} &  \ldots & \ldots &  \frac{1}{\langle n_{q,1}\rangle^{5}}\\
\vdots & \ddots & \ddots & \vdots\\
 \frac{1}{\langle n_{1,1}\rangle^{2N+1}} & \ldots & \ldots &  \frac{1}{\langle n_{q,1}\rangle^{2N+1}}
\end{matrix}
\right).
\]
We have
\begin{equation}\label{mondemonde}
\sum_{k=1}^{N}m_{k}a_{k}^{(\ell)}(\vec{n})=\Big(A_{q}(\vec{n}) \vec{\s}^{(\ell)}_{
q}\Big)\cdot \vec{m}_{N}, \quad 
\vec{\s}_{q}^{(\ell)}:=\Big((N_1^+-N_1^-),\ldots,(N_{q}^+-N_{q}^-)\Big)^{T}, 
\quad 
\end{equation}
where $\vec{m}_{N}:=(m_1,\ldots,m_{N})$ and ``$\cdot$'' denotes the standard scalar product on $\RRR^{N}$. By the above reasoning the vector $\vec{\s}_{q}^{(\ell)}$ is different from $\vec{0}$.

We claim that the vector $\vec{v}:=A_{q}(\vec{n})\vec{\s}^{(\ell)}_{q}$
has at least one component different from zero. 
Denote by $A_{q}^{q}(\vec{n})$ the $(q\times q)$-sub-matrix of $A_{q}(\vec{n})$ made of its firsts $q$ rows.
The matrix $A^{q}_{q}(\vec{n})$ is, up to rescaling the $k$-th column by the factor $\langle n_{1,k}\rangle^{3}$,
  a Vandermonde matrix, therefore
\begin{equation}\label{mondomondemon}
\det(A_{q}^{q}(\vec{n}))=\left(\prod_{j=1}^{q}\frac{1}{\langle n_{j,1}\rangle^{3}}\right)\prod_{1\leq i<k\leq q}\left( 
\frac{1}{{\langle n_{i,1}\rangle}^2}-\frac{1}{{\langle n_{k,1}\rangle}^2}
\right),
\end{equation}
which is different from zero
since $n_{i,1}\neq n_{k,1}$ for any $1\leq i<k\leq q$; this implies that ${\rm Rank}(A_{q}(\vec{n}))=q$, hence the claim follows since $\sigma^{(\ell)}_{q}\neq \vec{0}$. 

Fix  $\g>0$, $N\leq M\in \NNN$,  $N_0\in \NNN$ and $0\leq \ell\leq N$;
we introduce the following  ``bad'' set
\begin{equation}\label{BADSets}
\mathcal{B}_{N,N_0}(\vec{n},\g,\ell) :=
\left\{\vec{m}\in \calO \; : \; |\psi^{\ell}_{N}(\vec{m},\vec{n})|<{\g}\max{\big(\langle n_1\rangle,\ldots,\langle n_N\rangle}\big)^{-N_0} \right\}.
\end{equation}
We give an estimate of the sub-levels of the function $\psi^{\ell}_{N}(\vec{m},\vec{n})$. 
By the discussion above there exists $1\leq k'\leq q$ such that
$a_{k'}^{(\ell)}(\vec{n})\neq0$, set $k_{\infty}\in\{1,\cdots, q\}$ 
the index such that $|a_{k_{\infty}}^{(\ell)}(\vec{n})|=|\big(A_{q}^{q}(\vec{n})\vec{\sigma}^{(\ell)}\big)_{k_{\infty}}|=\norm{A_{q}^{q}(\vec{n})\vec{\s}^{\ell}}{\infty}>0$.
We start by proving that 
there exist constants $\mathtt{c}\ll 1$ and $\mathtt{b}\gg1$, both depending only on $q$ (and hence only on $N$),
such that
\begin{equation}\label{stimadalbasso1000}
|\del_{m_{k_{\infty}}}\psi_{N}^{\ell}(\vec{m},\vec{n})|
\geq \frac{\mathtt{c}}{\max{\big(\langle n_1\rangle,\ldots,\langle n_N\rangle}\big)^{\mathtt{b}}}.
\end{equation}
We have
\begin{equation}\label{monde1}
|\del_{m_{k_{\infty}}}\psi_{N}^{\ell}(\vec{m},\vec{n})|=
|a_{k_{\infty}}^{(\ell)}(\vec{n})|\stackrel{(\ref{mondemonde})}{=}|(A_{q}^{q}(\vec{n})\vec{\s}^{(\ell)}_{q})_{k_{\infty}}|
\geq K(\det{A_{q}^{q}(\vec{n})}),
\end{equation}
where $1\gg K=K(N)>0$ depends only on $N$.
The last inequality in \eqref{monde1}
follows by the fact that $A_{q}^{q}(\vec{n})$ is invertible, hence
\[
1\leq |(A^{q}_{q}(\vec{n}))^{-1}A_{q}^{q}(\vec{n})\vec{\s}_{q}^{(\ell)}|\leq (\det{A_{q}^{q}(\vec{n})})^{-1}N^{2}C_N\norm{A_{q}^{q}(\vec{n})\vec{\s}_{q}^{(\ell)}}{\infty},
\]
with $C_{N}>0$ and we have used $q\leq N$.
By formula \eqref{mondomondemon} one can deduce that 
\begin{equation*}
|\det{A_{q}^{q}(\vec{n})}|\geq \frac{\widetilde{K}}{\max{\big(\langle n_1\rangle,\ldots,\langle n_N\rangle}\big)^{\mathtt{b}}}
\end{equation*}
where $\mathtt{b}$ and $\widetilde{K}$ depend only on $N$. 
The latter inequality, together with
\eqref{monde1}, implies
the \eqref{stimadalbasso1000}.
Estimate \eqref{stimadalbasso1000}
implies  that
\[
{\rm meas}\Big(\mathcal{B}_{N,N_0}(\vec{n},\g,\ell)\Big) \leq \frac{\g}{\mathtt{c}\max{\big(\langle n_1\rangle,\ldots,\langle n_N\rangle}\big)^{N_0-\mathtt{b}}}.
\]
Hence, for $N_0\geq \mathtt{b}+2+N$, one obtains
\[
{\rm meas}\Big(\bigcap_{\g>0}\bigcup_{\vec{n}\in \NNN^{N}} 
\mathcal{B}_{N,N_0}(\vec{n},\g,\ell) \Big)\leq 
\lim_{\g\to0}  \frac{\g}{\mathtt{c}}\sum_{\vec{n}\in \NNN^{N}}\frac{1}{\max{\big(\langle n_1\rangle,\ldots,\langle n_N\rangle}\big)^{N_0-\mathtt{b}}}=0.
\]
By setting 
\[
\NN:=\bigcup_{0\leq N\leq M}\bigcap_{\g>0}\bigcup_{\vec{n}\in \NNN^{N}} 
\mathcal{B}_{N,N_0}(\vec{n},\g,\ell), 
\]
one gets the thesis.
\end{proof}

\subsection{Normal forms}\label{BNFnormal1}
In this Section we prove Theorem \ref{BNFiniziale}. The proof will be based on an iterative use of the following lemma.

\begin{lemma}\label{steppone}
Fix $p,K,N\in \NNN$, $r,\rho>0$, $1\leq p\leq N-1$ 
and $K'\leq K$. For $U\in B^{K}_{s_0}(I,r)\cap C^{K}_{*\RRR}(I;{\bf H}^{s}_e)$ 
be a solution of \eqref{sistemainiziale}
consider the system
\begin{equation}\label{problema100}
\del_{t}V=\ii E\big(\Lambda V+\bonyw(\tilde{L}^{(p)}(U;t,\x))[V]+G^{(p)}_1(U;t)[V]+G^{(p)}_2(U;t)[U]\big),
\end{equation}
where $G^{(p)}_{1}(U;t),G^{(p)}_{2}(U;t)\in \Sigma\RR^{-\rho}_{K,K',1}[r,N]\otimes\MM_{2}({\CCC})$ are
 (R,R,P)-operator  and $\tilde{L}^{(p)}(U;t,\x)$ is a  diagonal and constant coefficients in $x$ (R,R,P)-matrix of the form
\begin{equation}\label{constCoeff1000}
\tilde{L}^{(p)}(U;t,\x):=\left(
\begin{matrix}
{\mathtt{m}}^{(p)}(U;t,\x) & 0\\ 0 & \ol{\mathtt{m}^{(p)}(U;t,-\x)}
\end{matrix}
\right), \quad \mathtt{m}^{(p)}(U;t,\x)=\mathtt{m}_{2}(U;t)(\ii\x)^{2}+\mathtt{m}^{(p)}_0(U;t,\x),
\end{equation}
where $\mathtt{m}_{2}(U;t)$ is the real symbol in $\Sigma\calF_{K,K',1}[r,N]$ given in \eqref{constCoeff}, 
while $\mathtt{m}_0^{(p)}(U;t,\x)\in \Sigma\Gamma^{0}_{K,K',1}[r,N]$ is such that
(recalling Def. \ref{kernel})
\begin{equation}\label{kamstep}
\begin{aligned}
\mathtt{m}_0^{(1)}(U;t,\x)&=\sum_{j=1}^{N-1}m_{j}^{(1)}(U,\ldots,U;t,\x)+m^{(1)}_{N}(U;t,\x),\\
\mathtt{m}_0^{(p)}(U;t,\x)&=\sum_{j=1}^{p-1}\pro{m_{j}^{(p)}}(U,\ldots, U;t,\x)
+\sum_{j=p}^{N-1}m_{j}^{(p)}(U,\ldots,U;t,\x)+m^{(p)}_{N}(U;t,\x),\;\;\;\; 
2\leq p\leq N-1,\\
\end{aligned}
\end{equation}
where
\begin{equation}\label{kamstep2}
m^{(p)}_{j}\in \widetilde{\Gamma}^{0}_{j}, \quad j=1,\ldots,N-1, \quad m^{(p)}_{N}\in \Sigma\Gamma^{0}_{K,K',N}[r,N].
\end{equation}
For $r$ small enough and $\vec{m}$ outside the subset $\NN$ given by Proposition \ref{stimemisura} the following holds.
There is $s_0>0$  such that for $s\geq s_0$, 
there is an invertible
  (R,R,P)-map 
 \begin{equation}\label{mappatot200}
\Theta_p(U)[\cdot]
\; :\;  C^{K-K'}_{*\R}(I,{\bf{H}}^{s}(\TTT;\CCC^{2})) \rightarrow C^{K-K'}_{*\R}(I,{\bf{H}}^{s}(\TTT;\CCC^{2})),
\end{equation}
satisfying the following:
\begin{enumerate}
\item[(i)] there exist $C$ depending on $s,r,K$ such that
\begin{equation}\label{mappeprop200}
\begin{aligned}
\|\Theta_p(U)[V]\|_{K-K',s}, \, 
\left\|\big(\Theta_p(U)\big)^{-1}[V]\right\|_{K-K',s}&\leq \|V\|_{K-K',s}(1+C\|U\|_{K,s_0}), \\
\end{aligned}
\end{equation}
for any $V\in C^{K-K'}_{*\R}(I,{\bf{H}}^{s}(\TTT;\CCC^{2}))$;
\item[(ii)]
$\Theta_{p}(U)-\uno$ and $(\Theta_p(U))^{-1}-\uno$ belong to the class $\Sigma\MM_{K,K',1}[r,N]\otimes\MM_{2}(\CCC)$;
\item[(iii)]
the function $W=\Theta_p(U)[V]$ satisfies
\begin{equation}\label{problemafinale201}
\del_{t}W=\ii E\big(\Lambda W+\bonyw(\tilde{L}^{(p+1)}(U;t,\x))[W]+G_1^{(p+1)}(U;t)[W]+G_2^{(p+1)}(U;t)[U]\big),
\end{equation}
where 
$V$ satisfies \eqref{problema100}. 
The operators $G_1^{(p+1)}(U;t),G_{2}^{(p+1)}(U;t)$ are (R,R,P)-operators in  the class \linebreak$\Sigma\RR^{-\rho+\tilde{m}}_{K,K'+1,1}[r,N]\otimes\MM_{2}({\CCC})$ 
for some $\tilde{m}>0$,  $\tilde{L}^{(p+1)}(U;t,\x)$ is a (R,R,P)-matrix 
in 
 $\Sigma\Gamma^{2}_{K,K'+1,1}[r,N]\otimes \MM_{2}(\CCC)$
 with constant coefficients in $x\in \TTT$
and
it
has the form 
\begin{equation}\label{constCoeff2000}
\tilde{L}^{(p+1)}(U;t,\x):=\left(
\begin{matrix}
\mathtt{m}^{(p+1)}(U;t,\x) & 0\vspace{0.4em}\\ 0 & \ol{\mathtt{m^{(p+1)}}(U;t,-\x)}
\end{matrix}
\right), \quad \mathtt{m}^{(p+1)}(U;t,\x)=\mathtt{m}_{2}(U;t)(\ii\x)^{2}+\mathtt{m}^{(p+1)}_0(U;t,\x),
\end{equation}
where $\mathtt{m}_{2}(U;t)$ is 
given in \eqref{constCoeff},  the symbol $\mathtt{m}^{(p+1)}_0(U;t,\x)$ is in $\Sigma\Gamma^{0}_{K,K'+1,1}[r,N]$ and it
has the form
\begin{equation}\label{kamstep10}
\mathtt{m}_0^{(p+1)}(U;t,\x)=\sum_{j=1}^{p}\pro{m^{(p)}_{j}}(U,\ldots,U;t,\x)+\sum_{j=p+1}^{N-1}m^{(p+1)}_{j}(U,\ldots,U;t,\x)+m^{(p+1)}_{N}(U;t,\x),
\end{equation}
where $m_{j}^{(p)}\in \widetilde{\Gamma}^{0}_{j}$, $j=1,\ldots,p$ are given in \eqref{kamstep2} and 
\begin{equation}\label{kamstep11}
m^{(p+1)}_{j}\in \widetilde{\Gamma}^{0}_{j}, \quad j=p+1,\ldots,N-1, \quad m^{(p+1)}_{N}\in \Gamma^{0}_{K,K'+1,N}[r].
\end{equation}
\end{enumerate}
\end{lemma}

\begin{proof}
Let $f(U;t,\x)$ be a symbol in $\widetilde{\Gamma}^{0}_{p}$ which has constant coefficients in $x\in \TTT$.
Consider the system
\begin{equation}\label{generatore}
\del_{\tau}W(\tau)=\bonyw(\widehat{F}^{(p)}(U;t,\x))[W(\tau)], \quad 
\widehat{F}^{(p)}(U;t,\x):=\left(
\begin{matrix}
f(U;t,\x) & 0 \vspace{0.4em}\\
0&\ol{f(U;t,-\x)}
\end{matrix}
\right).
\end{equation}
Suppose moreover that the matrix $\widehat{F}^{(p)}(U;t,\x)$ is a (R,R,P)-matrix of symbols.
By standard theory of ODEs on Banach spaces the flow $\Theta_{p}^{\tau}(U)[\cdot]$ of \eqref{generatore} 
is well defined for $\tau\in[0,1]$. We set $\Theta_{p}(U)[\cdot]:=\Theta^{\tau}_p(U)[\cdot]_{|_{\tau=1}}$.
Estimates \eqref{mappeprop200}
hold by direct computation.
 Item $(ii)$
 follows by Taylor expanding $\Theta_{p}^{\tau}(U)[\cdot]$ in $\tau=0$ and by using 
 Remark \ref{inclusionifacili} and item $(i)$ of Proposition \ref{composizioniTOTALI}.
The same argument implies the following further properties  of the map $\Theta_{p}(U)[\cdot]$:
 \begin{equation}\label{tempononrimane}
\Theta_p(U)[\cdot]=\uno+\bonyw(\widehat{F}^{(p)}(U;t,\x))+\bonyw(C^{+}(U;t,\x))[\cdot]+{R}^{+}(U;t)[\cdot],
 \end{equation}
  \begin{equation}\label{tempononrimane10}
(\Theta_p(U))^{-1}[\cdot]=\uno-\bonyw(\widehat{F}^{(p)}(U;t,\x))+\bonyw({C}^{-}(U;t,\x))[\cdot]+{R}^{-}(U;t)[\cdot],
 \end{equation}
 for some (R,R,P)-matrices of symbols  $C^{+}(U;t,\x)$, ${C}^{-}(U;t,\x)$ independent of $x$ in $\Sigma\Gamma^{0}_{K,K',p+1}[r,N]\otimes\MM_{2}(\CCC)$ and some (R,R,P)-operators ${R}^{+}(U;t)[\cdot]$,  ${R}^{-}(U;t)[\cdot]$  belonging to 
 $\Sigma\RR^{-\rho}_{K,K',1}[r,N]\otimes\MM_{2}(\CCC)$ (actually the homogeneity of these remainders is bigger than $p$ but,
 at this level, we do not emphasize this property and we embed them in the remainders of homogeneity 1).

 Finally, since $\widehat{F}^{(p)}(U,\ldots,U;t,\x)$ is a (R,R,P)-matrix of symbols
 then the flow $\Theta_{p}^{\tau}(U)[\cdot]$ is reversibility preserving. 
 Indeed, by setting $G^{\tau}=S\Theta^{\tau}_{p}(U;-t)-\Theta^{\tau}_p(U_S;t)S$, one can note that
 \[
 \del_{\tau}G^{\tau}=\bonyw(\widehat{F}^{(p)}(SU;t))G^{\tau},
 \]
 with $G^{0}=0$, where we used that $S\widehat{F}^{(p)}(U;-t)=\widehat{F}^{(p)}(U_S;t)$ (which is \eqref{invo3}). 
 This implies that $G^{\tau}\equiv0$ for $\tau=[0,1]$, which means that $\Theta_{p}(U;t)$ is reversibility preserving.
 
 Since $U$ solves \eqref{sistemainiziale}, there is a (R,R,P)-map $M\in \Sigma\MM^{\tilde{m}}_{K,0,1}[r,N]\otimes\MM_{2}(\CCC)$,
 for some $\tilde{m}>0$,
 such that
 \begin{equation}\label{sitosito}
 \del_{t}U=\ii E\Lambda U+\ii EM(U;t)[U].
 \end{equation}
 Hence, by taking the derivative w.r.t. the variable $t$ in \eqref{tempononrimane}, we have
\begin{align}
 \del_{t}(\Theta_p(U))[\cdot]&=\sum_{j=1}^{p}\bonyw(\widehat{F}^{(p)}(U,\ldots,\underbrace{\del_t U}_{j-th},U,\ldots,U;t,\x))+
 \bonyw(\del_{t}{C}^{+}(U;t,\x))[\cdot]
+( \del_{t}{R}^{+}(U;t))[\cdot]\nonumber\\
&=\sum_{j=1}^{p}\bonyw(\widehat{F}^{(p)}(U,\ldots,\underbrace{\ii E\Lambda U}_{j-th},U,\ldots,U;t,\x))+\bonyw(B(U;t,\x))[\cdot]+\widetilde{R}^{+}(U;t)[\cdot],\label{tempononrimane2}
\end{align}
 for some $B(U;t,\xi)\in \Sigma\Gamma^{0}_{K,K'+1,p+1}[r,N]\otimes \MM_{2}(\CCC)$, and $\widetilde{R}^{+}(U;t)\in \Sigma\RR^{-\rho+\tilde{m}}_{K,K'+1,1}[r,N]\otimes\MM_{2}(\CCC)$,
 where we used Proposition \ref{composizioniTOTALI} 
 (in particular items $(iv),(v)$)
 and $\tilde{m}$ is the loss given by $M$ in \eqref{sitosito}.
 We fix $0\leq \rho'= \rho-\tilde{m}$ which is possible since $\rho\gg1$.
 Now
 if $W=\Theta_{p}(U)[V]$ one has that
 \begin{equation}\label{coniugo}
 \begin{aligned}
 \del_{t}W&=\Theta_p(U)\Big[\ii E\big(\Lambda +\bonyw(\tilde{L}^{(p)}(U;t,\x))+G^{(p)}_1(U;t)\big) \Big](\Theta_{p}(U))^{-1}[W]\\
 &+\ii E \Theta_p(U)G^{(p)}_2(U;t)U+(\del_{t}\Theta_p(U))(\Theta_p(U))^{-1}[W]=\\
 &=\ii E\big(\Lambda W+\bonyw(\tilde{L}^{(p)}(U;t,\x))[W]\big)+\sum_{j=1}^{p}\bonyw(\widehat{F}^{(p)}(U,\ldots,\underbrace{\ii E\Lambda U}_{j-th},U,\ldots,U;t,\x))[W]+\\
 &+\ii E\bonyw(C_1(U;t,\x))[W]+\ii EG_{3}(U;t)[W]+\ii EG_{4}(U;t)[U],
 \end{aligned}
 \end{equation}
 for some (R,R,P)-matrix of symbols ${C}_1(U;t,\x)$ independent of $x$ belonging to 
 $\Sigma\Gamma^{0}_{K,K'+1,p+1}[r,N]\otimes\MM_{2}(\CCC)$ and some (R,R,P)-operators ${G}_3(U;t)[\cdot],G_{4}(U;t)[\cdot]$ 
 belonging to $\Sigma\RR^{-\rho'}_{K,K',1}[r,N]\otimes\mathcal{M}_2(\CCC)$. In the previous computation  we used Proposition \ref{composizioniTOTALI},
 the \eqref{tempononrimane}, \eqref{tempononrimane10} and \eqref{tempononrimane2}.
 In particular we used also the fact that the matrix $\tilde{L}^{(p)}(U;t,\xi)$, together with the matrices of symbols appearing in \eqref{tempononrimane} and \eqref{tempononrimane10} do not depend on $x$.
 In the notation of item $(iii)$ of the statement 
 we look for
 $\widehat{F}^{(p)}(U,\ldots,U;t,\x)$ such that
\begin{equation}\label{omologica}
 \sum_{j=1}^{p}f(U,\ldots,\underbrace{\ii E\Lambda U}_{j-th},U,\ldots,U;t,\x)+\ii m^{(p)}_{p}(U,\ldots,U;t,\x)=
\ii \pro{m^{(p)}_{p}}(U;\ldots,U;t,\x).
\end{equation}
Recalling the definition of the operator $\Lambda$ in \eqref{DEFlambda} (see also \eqref{NLS1000}, \eqref{potenziale1}, 
and \eqref{smalldiv}), we have
that, passing to Fourier series,  the equation \eqref{omologica} is equivalent to
\begin{equation*}\label{omologica2}
\begin{aligned}
\!\!\psi^{\ell}_{p}(\vec{m},\vec{n})f\big(
\Pi^{+}_{n_1}U, \ldots,\Pi^{+}_{n_\ell}U, \Pi^{-}_{n_{\ell+1}}U, \ldots,\Pi^{-}_{n_p}U;t,\x
\big)=
-m_{p}^{(p)}\big(
\Pi^{+}_{n_1}U, \ldots,\Pi^{+}_{n_\ell}U, \Pi^{-}_{n_{\ell+1}}U, \ldots,\Pi^{-}_{n_p}U;t,\x
\big)
\end{aligned}
\end{equation*}
in the following cases:
\begin{itemize}
\item $p$ is odd, $0\leq \ell\leq p$ and for any $ \vec{n}=(n_1,\ldots, n_p)\in \NNN^{p}$;
\item $p$ is even , $0\leq \ell\leq p$ with $\ell\neq p/2$ and for any $ \vec{n}=(n_1,\ldots, n_p)\in \NNN^{p}$;
\item $p$ is even,  $\ell= p/2$ and for any $ \vec{n}=(n_1,\ldots, n_p)\in \NNN^{p}$ such that
\[
\{n_1,\ldots, n_\ell\}\neq\{n_{\ell+1},\ldots,n_p\}.
\]
\end{itemize}

By estimate \eqref{stimadalbasso} on $\psi^{\ell}_{p}(\vec{m},\vec{n})$, we get that
$f(U,\ldots,U;t,\x)$ is a symbol in $\widetilde{\Gamma}_{p}^{0}$
and does not depend on $x$ since so does  $m_{p}^{(p)}$.
Furthermore, since $\widetilde{L}^{(p)}$ in \eqref{constCoeff1000} is a (R,R,P)-matrix of symbols, 
one has that  $\widehat{F}^{(p)}(U,\ldots,U;t,\x)$  in \eqref{generatore} is even in $\x$, and reality preserving (i.e. satisfies
resp.  
\eqref{parisimbo1} and \eqref{prodotto}). Finally, since $m^{(p)}_{p}(U,\ldots,U;t,\x)$ satisfies \eqref{revomo1000}
and the function $\psi_{p}^{\ell}(\vec{m},\vec{n})$  in \eqref{smalldiv}
is real and even in each component of $\vec{n}$, one has that the symbol $\widehat{F}^{(p)}(U,\ldots,U;t,\x)$
satisfies \eqref{revomo1000}.
Thanks to the choice of $f$ above the equation \eqref{coniugo} has the form \eqref{problemafinale201}
for a suitable (R,R,P)-matrix of symbols $\tilde{L}^{(p+1)}$ of the form \eqref{constCoeff2000}.
\end{proof}

\begin{proof}[\bf Proof of Theorem \ref{BNFiniziale}]
Let $\NN$ be the set of parameters $\vec{m}\in [-1/2,1/2]^{M}$ given in Proposition \ref{stimemisura}.
We apply Lemma \ref{steppone} to the system \eqref{problemafinale} 
since it has the form \eqref{problema100}
with $p=1$,  $\tilde{L}^{(1)}\rightsquigarrow L$ in \eqref{constCoeff}, $G^{(1)}_1,G^{(1)}_2\rightsquigarrow Q_1,Q_2$, 
and $\rho\rightsquigarrow \rho-m$ (with $m$ given by Theorem \ref{regolarizza}).
The lemma guarantees  the existence of a map $\Theta_1(U)[\cdot]$ (see \eqref{mappatot200})
such that 
the function $W_{1}=\Theta_{1}(U)[V]$ satisfies a system of the form \eqref{problemafinale201} with 
$\tilde{L}^{(2)}$ given in \eqref{constCoeff2000}
 with $p=1$ and where $G_{1}^{(2)},G_{2}^{(2)}$
 are some operators in 
 $\Sigma\RR^{-\rho^{(1)}}_{K,K'+1,1}[r,N]\otimes\MM_{2}(\CCC)$.
 Here $\rho^{(1)}=\rho-m-\tilde{m}$ where $\tilde{m}$ is the loss of derivatives 
 produced  by the map $M(U;t)$ given in \eqref{sitosito}. 
This new system still satisfies the hypotheses of Lemma \ref{steppone}, hence
 we may apply it iteratively.  We obtain
 a sequence of maps $\Theta_{j}(U)[\cdot]$
 for  $j=1,\ldots,N-1$ such that
 $W_{j}:=\Theta_{j}(U)[W_{j-1}]$ satisfies a system of the form \eqref{problemafinale201}
 for suitable matrices of symbols
 $ \tilde{L}^{(j+1)}$ given in \eqref{constCoeff2000} with $p=j$
and where the remainders  
$G_1^{(j+1)},G_{2}^{(j+1)}$ belong to $\Sigma\RR^{-\rho^{(j)}}_{K,K'+j,1}[r,N]\otimes\MM_{2}(\CCC)$
where $\rho^{(j)}\sim \rho-m-j\tilde{m}$, which is positive since  $\rho\gg N$ in Theorem \ref{regolarizza}. 
  We set $\Theta(U)[\cdot]:=\Theta_{N-1}(U)\circ\cdots \circ \Theta_{1}(U)[\cdot]$
which satisfies items $(i)$, $(ii)$ because each  map $\Theta_{j}$, $j=1,\ldots, N-1$ has similar properties by Lemma \ref{steppone}.
With this choice, the constant coefficients in $x$ matrix of symbols $L_1(U;t,\x)$ in \eqref{problemafinale2} 
is  equal to
$\tilde{L}^{(N-1)}(U;t,\x)$ (given in \eqref{constCoeff2000}, \eqref{constCoeff} and \eqref{kamstep10} with $p=N-1$), which 
satisfies \eqref{constCoeff2}.
The smoothing remainders $Q_1^{(1)}(U;t),Q_{2}^{(2)}(U;t)$ belong to
the class $\Sigma\RR^{\rho-m_1}_{K,K'',1}[r,N]\otimes\MM_{2}(\CCC)$
where $K''\sim K'+N$ and $m_1=m+(N-1)\tilde{m}$. 
\end{proof}

\subsection{Modified energies}\label{formeenergia}

In this subsection we  give the proof of Theorem \ref{longgun}.
We introduce the classes of multilinear forms which will be used to construct modified energies  for 
a system of the form \eqref{problemafinale2}. The following definition is Definition $4.4.1$ in \cite{maxdelort}.

\begin{de}\label{formaggino}
Let $\rho,s\in \RRR$ with $\rho,s\geq0$ and $p\in \NNN$. One denotes by $\widetilde{\calL}_{p,\pm}^{s,-\rho}$
the space of symmetric $(p+2)$-linear forms
$(U_0,\ldots,U_{p+1}) \rightarrow L(U_0,\ldots,U_{p+1})$
defined on $C^{\infty}(\TTT;\CCC^{2})$ and satisfying for some $\mu\in \RRR_{+}$
and any $(n_0,\ldots, n_{p+1})\in \NNN^{p+2}$
 and any $(U_0,\ldots,U_{p+1})\in (C^{\infty}(\TTT;\CCC^{2}))^{p+2}$,
\begin{equation}\label{stimeene}
\begin{aligned}
|L(\Pi_{n_0}U_0,\ldots, \Pi_{n_{p+1}}U_{p+1})|&\leq C\max(\langle n_0\rangle, \ldots, \langle n_{p+1}\rangle)^{2s-\rho}\\
&\times \max_{3}(\langle n_0\rangle, \ldots, \langle n_{p+1}\rangle)^{\mu+\rho}\prod_{j=0}^{p+1}\|\Pi_{n_{j}}U_{j}\|_{L^{2}}
\end{aligned}
\end{equation}
where $\max_{3}(\langle n_0\rangle, \ldots, \langle n_{p+1}\rangle)$ is the third largest value among
$\langle n_0\rangle, \ldots, \langle n_{p+1}\rangle$, and such that
\begin{equation}\label{stimeene1}
\begin{aligned}
L(\Pi_{n_0}U_0,\ldots, \Pi_{n_{p+1}}U_{p+1})\neq 0 \Rightarrow \sum_{j=0}^{p+1}\s_{j}n_0=0,
\end{aligned}
\end{equation}
for some choice of the signs $\s_{j}\in \{+1,-1\}$ for $j=0,\ldots, p+1$, and for any 
$U_0,\ldots,U_{p+1}$ satisfying  \eqref{involuzione4},
\begin{equation}
L(SU_0,\ldots,SU_{p+1})=\pm L(U_0,\ldots,U_{p+1}).
 \end{equation}

\end{de}

%
%
%

The following lemma collects some properties of the class $\widetilde{\calL}^{s,-\rho}_{p,\pm}$
which are proved in Section $4.4$ of \cite{maxdelort}.

\begin{lemma}\label{fattiSulleForme}
The following facts hold true.

\noindent
$(i)\!$
Fix $\rho\geq0$, $p\in \NNN^{*}$ and consider $R\in \widetilde{R}^{-\rho}_{p}$ satisfying \eqref{3.1.21} (resp. \eqref{3.1.20}). One has that the
$L(U_0,\ldots,U_{p+1})$ defined as the symmetrization of 
\begin{equation}\label{fattiSulleForme2}
(U_0,\ldots,U_{p+1}) \rightarrow \int_{\TTT}(\langle D\rangle^{s}SU_0)(
\langle D\rangle^{s}R(U_1,\ldots,U_{p})U_{p+1}
)dx
\end{equation}
belongs to $\widetilde{\calL}^{s,-\rho}_{p,+}$ (resp. $\widetilde{\calL}^{s,-\rho}_{p,-}$).

\noindent
$(ii)$ Let $L\in \widetilde{\calL}^{s,-\rho}_{p,\pm}$. Then for any $m\geq0$ such that $\rho>m+1/2$
and any $s>\rho+\mu+m+1/2$, $L$ extends as a continuous $(p+2)$-linear form 
on $H^{s}(\TTT;\CCC^2)\times\cdots \times H^{s}(\TTT;\CCC^2)\times H^{s-m}(\TTT;\CCC^2)\times H^{s}(\TTT;\CCC^2)\times \cdots \times H^{s}(\TTT;\CCC^2)$.

\noindent
$(iii)$ Let $p=2\ell$ with $\ell\in \NNN^{*}$ and $L\in \widetilde{\mathcal{L}}^{s,-\rho}_{p,-}$.
For $U$ even in $x$ satisfying \eqref{involuzione4} one has, for $n_0,\ldots,n_{\ell}\in \NNN^{*}$,
\begin{equation}\label{kerneladj1}
L(\Pi^{+}_{n_0}U,\ldots, \Pi_{n_{\ell}}^{+}U,\Pi_{n_{0}}^{-}U,\ldots,\Pi_{n_{\ell}}^{-}U)=0.
\end{equation}

\noindent
(iv) Let $\vec{m}$ be outside the subset $\mathcal{N}$ and $N_0$  given by Proposition \ref{stimemisura}.
Then for any $L\in \widetilde{\calL}^{s,-\rho}_{p,-}$ there is $\tilde{L}\in \widetilde{\calL}^{s,-\rho+N_0}_{p,+}$
such that
\begin{equation}\label{omoequa}
\sum_{j=0}^{p+1}\widetilde{L}(U,\ldots,\underbrace{E\Lambda U}_{j-th}, \ldots,U )=\ii L(U,\ldots, U),
\end{equation}
where $E$ and $\Lambda$ are defined  in \eqref{matrici} and \eqref{DEFlambda} respectively.

\noindent
$(v)$ Let $L\in \widetilde{\mathcal{L}}^{s,-\rho}_{p,\pm}$ and $M\in \Sigma\mathcal{M}^{m}_{K,K',q}[r,N]\otimes\MM_{2}(\CCC)$
(see Def. \ref{smoothoperatormaps})
which is reality preserving and reversible (resp. reversibility preserving) according to Def. \ref{riassunto-mappe}.
Then $U\rightarrow L(U,\ldots,U,M(U;t)U,U,\ldots, U)$
can be written as the sum $\sum_{q'=0}^{N-p-q-1}L_{q'}$, for some 
$L_{q'}\in \widetilde{\calL}^{s,-\rho+m}_{p+q+q',\mp}$ (resp. $\widetilde{\calL}^{s,-\rho+m}_{p+q+q',\pm}$),
plus a term that, at any time $t$, is
\begin{equation}\label{compoforme2}
O(\|U(t,\cdot)\|^{p+2}_{\hcic^{s}}\|U(t,\cdot)\|^{N-p}_{K',\s}+\|U(t,\cdot)\|^{p+1}_{\hcic^{s}}\|U(t,\cdot)\|^{N-p}_{K',\s}\|U(t,\cdot)\|_{K',s})
\end{equation}
if $s>\s\gg \rho$ and if $\|U(t,\cdot)\|_{K',\s}$ is bounded.
\end{lemma}

\begin{lemma}[{\bf First energy inequality}]\label{FEI2}
Let
$U(t,x)\in B_{s}^{K}(I,r)$ be the solution of \eqref{sistemainiziale} with $r$ small enough.
If $\rho>0$ is large enough 
there are constants $s\geq s_0\gg K\geq \rho\gg\rho''\gg N$ 
and multilinear forms $L_{p}\in \widetilde{\calL}^{s,-\rho''}_{p,-}$, $p=1,\ldots,N-1$, 
such that for $s\geq s_0$ the following holds.

Consider the functions $V=\Phi(U)[U]$, given by Theorem \ref{regolarizza}, 
and $W=\Theta(U)[V]$ given by Theorem \ref{BNFiniziale}.
Then, for any $s\geq s_0$,
one has
\begin{equation}\label{stimanormasob}
\frac{d}{dt}\int_{\TTT}|\langle D\rangle^{s}W(t,x)|^{2}dx=\sum_{p=1}^{N-1}L_{p}(U,\ldots,U)+O(\|U(t,\cdot)\|_{\hcic^{s}}^{N+2})
\end{equation}
for $t\in I$. Moreover
\begin{equation}
C_{s}^{-1}\|W\|_{\hcic^{s}}\leq \|U\|_{\hcic^s}\leq C_{s}\|W\|_{\hcic^{s}},
\end{equation}
for some constant $C_{s}>0$.
\end{lemma}

Before giving the proof of Lemma \ref{FEI2} we need to prove the following.

\begin{lemma}\label{aslongas}
Let $U(t,\cdot)$ be the solution of \eqref{sistemainiziale}
defined on some interval $I\subset \RRR$ and belonging to 
$C^{0}(I;\hcic^{s}_{e}(\TTT;\CCC^{2}))$. For any $0\leq k \leq K$
there is a constant $C_{k}$ such that, as long as $\|U(t,\cdot)\|_{\hcic^{s}}\leq 1$ with $s\gg K$,
one has
\begin{equation}\label{aslong1}
\|\del_{t}^{k}U(t,\cdot)\|_{\hcic^{s-2k}}\leq C_{k}\|U(t,\cdot)\|_{\hcic^{s}}.
\end{equation}
\end{lemma}

\begin{proof}
We argue by induction. Clearly \eqref{aslong1} holds for $k=0$.
Assume \eqref{aslong1} holds for $k=0,\ldots, k'\leq K-1$. Since 
by assumption $\|U(t,\cdot)\|_{\hcic^{s}}\leq 1$, then 
\[
\sum_{k=0}^{k'}\|\del_{t}^{k}U(t,\cdot)\|_{\hcic^{s-2k}}\leq \tilde{C}_{k'},
\]
for some $\tilde{C}_{k'}$ uniformly for $t\in I$. In order to get \eqref{aslong1}
 it is enough to show $\|\del_{t}^{k'+1}U(t,\cdot)\|_{\hcic^{s-2(k'+1)}}\leq C\|U(t,\cdot)\|_{\hcic^{s}}$.
Using \eqref{sistemainiziale} we have that
\begin{equation}
\begin{aligned}
\del_{t}^{k'+1}U&=\ii E(\Lambda\del_{t}^{k'}U+\del_{t}^{k'}\big(\bonyw(A(U;t,x,\x))[U]\big)+
\del_{t}^{k'}\big(R(U;t)[U]\big))\\
&=\ii E\Lambda\del_{t}^{k'}U+\ii E\sum_{j_1+j_2=k'}C_{j_1,j_2}
\bonyw(\del_{t}^{j_1}A(U;t,x,\x))[\del_{t}^{j_2}U]\\
&+\ii E\sum_{j_1+j_2=k'}C_{j_1,j_2}
(\del_{t}^{j_1}R(U;t)))[\del_{t}^{j_2}U],
\end{aligned}
\end{equation}
where $C_{j_1,j_2}$ are some binomial coefficients.
By \eqref{azione7} in Proposition \ref{azionepara},  \eqref{porto20} with $K'=0$ (recalling Remark \ref{starship}),
the inductive hypothesis and using that $\|U(t,\cdot)\|_{\hcic^{s}}\leq 1$,
we get
\begin{equation}
\|\del_{t}^{k'+1}U(t,\cdot)\|_{\hcic^{s-2(k'+1)}}\leq C\|U(t,\cdot)\|_{\hcic^{s}}.
\end{equation}
This concludes the proof.

\end{proof}

\begin{proof}[{\bf Proof of Lemma \ref{FEI2}}]
Since the maps $\Phi,\Theta$ are (R,R,P)-maps, then the function $W=\Theta(U)[\Phi(U)[U]]$
is even in $x$ and 
satisfies \eqref{involuzione4}. 
In particular, by  items $(ii)$ of Theorems \ref{regolarizza} and  \ref{BNFiniziale}, we have that
\begin{equation}\label{malditesta}
W=U+\sum_{p=1}^{N-1}M_{p}(U,\dots,U)[U]+M_{N}(U;t)[U],
\end{equation}
for some (R,R,P) maps $M_{p}\in \widetilde{\MM}_{p}\otimes\MM_{2}(\CCC)$, $p=1,\ldots,N-1$ and 
$M_{N}\in \Sigma\MM_{K,K'',N}[r,N]\otimes \MM_{2}(\CCC)$.

We remark also that, by Lemma \ref{aslongas} and \eqref{spazionorm}, 
we have $\|U(t,\cdot)\|_{K,s}\leq C_{s,K}\|U(t,\cdot)\|_{\hcic^{s}}$ for some $C_{s,K}>0$.
According to system \eqref{problemafinale2}, recalling \eqref{constCoeff2}, Remark \ref{formavera}
and Definition \ref{kernel}, 
we get
\begin{equation}\label{catena10}
\begin{aligned}
\frac{d}{dt}\int_{\TTT}|\langle D\rangle^{s}W(t,x)|^{2}dx&=2{\rm Re}\;\ii\int_{\TTT}\ol{(\langle D\rangle^s W)}\big[
\bonyw(\mathtt{l}(\x)+\mathtt{m}_2(U;t)(\ii\x)^{2})E
\langle D\rangle^{s}W
\big]dx\\
&+2{\rm Re}\;\ii\int_{\TTT}\ol{(\langle D\rangle^s W)}\big[
\bonyw(\pro{\mathtt{M}_0^{(1)}}(U;t,\x))E
\langle D\rangle^{s}W
\big]dx\\
&+2{\rm Re}\;\ii\int_{\TTT}\ol{(\langle D\rangle^s W)}
(\langle D\rangle^{s}EQ_{1}^{(1)}(U;t)[W])
dx\\
&+2{\rm Re}\;\ii\int_{\TTT}\ol{(\langle D\rangle^s W)}
(\langle D\rangle^{s}EQ_{2}^{(1)}(U;t)[U])
dx,
\end{aligned}
\end{equation}
where  
\begin{equation*}
\pro{\mathtt{M}_0^{(1)}}(U;t,\x):=\left(\begin{matrix}\pro{\mathtt{m}_0^{(1)}}(U;t,\x) & 0\\ 0
& \pro{\ol{\mathtt{m}_0^{(1)}}}(U;t,-\x)\end{matrix}\right).
\end{equation*}

The contribution of the first integral is zero since the symbol $\mathtt{l}(\x)+\mathtt{m}_2(U;t)(\ii\x)^{2}$ is real.
By using Lemmata \ref{FEI} and \ref{aslongas} 
we have that the contribution of the second integral 
is bounded by 
\[
O(\|U(t,\cdot)\|^{N}_{\hcic^{s}}\|W\|_{\hcic^{s}}^{2}).
\]
Let us consider the fourth integral term in \eqref{catena10}. By definition we have that
\[
Q_{2}^{(1)}(U;t)[U]=\sum_{p=1}^{N-1}Q_{2,p}^{(1)}(U,\ldots,U)[U]+Q_{2,N}^{(1)}(U;t)[U]
\]
where $Q_{2,p}^{(1)}\in \widetilde{\RR}^{-\rho'}_{p}\otimes\MM_{2}(\CCC)$, 
$p=1,\ldots,N-1$ and $Q_{2,N}^{(1)}\in \RR^{-\rho'}_{K,K'',N}[r]\otimes\MM_{2}(\CCC)$ are (R,R,P)-operators
and with $\rho\gg\rho'\gg N$, $\rho':=\rho-m_1$ given in Theorem \ref{BNFiniziale}.
The contribution given by the term $Q_{2,N}^{(1)}(U;t)$ is bounded by 
\[
O(\|U(t,\cdot)\|^{N+2}_{\hcic^{s}}).
\]
Furthermore the operators $Q_{2,p}^{(1)}(U,\ldots,U)$ satisfy  \eqref{3.1.21}  by Lemma \ref{pasubio}
($\ii E Q_{2,p}^{(1)}$ satisfies \eqref{3.1.20} by Remark \ref{considerazioni}).
Hence the contribution to the fourth integral in \eqref{catena10} coming from the terms 
$Q_{2,p}^{(1)}(U,\ldots,U)$
can be written as in \eqref{fattiSulleForme2}. By item $(i)$ of Lemma \ref{fattiSulleForme}
such contributions can be written as
$\tilde{L}_{p}(U,\ldots, U)$ for some multilinear form $\tilde{L}_{p}(U_0,\ldots,U_{p+1})$ belonging to $\widetilde{\calL}^{s,-\rho'}_{p,-}$.
Consider now the operator $Q_{1}^{(1)}(U;t)$ in the third integral in \eqref{catena10}.
If $\rho'\gg \rho''$ is large enough, then, by \eqref{malditesta} and item $(iii)$ of Proposition \ref{composizioniTOTALI}, we get
\[
Q_{1}^{(1)}(U;t)[W]=\sum_{p=1}^{N-1}\tilde{Q}_{p}(U,\ldots,U)[U]+\tilde{Q}_{N}(U;t)[U],
\]
for some $\tilde{Q}_{p}\in \widetilde{\RR}^{-\rho''}_{p}\otimes\MM_{2}(\CCC)$, 
$p=1,\ldots,N-1$ and $\tilde{Q}_{N}\in \RR^{-\rho''}_{K,K'',N}[r]\otimes\MM_{2}(\CCC)$ which are (R,R,P)-operators.
Hence the contribution of the third integral can be studied as done for the term coming from $Q_{2}^{(1)}(U;t)$.
This concludes the proof.
\end{proof}

%

\begin{proof}[{\bf Proof of Theorem \ref{longgun}}]

%

Let $s>0$ large and $r>0$ small enough.
By Theorem $1.2$ in \cite{FIloc}
for any even function $u_0\in H^{s}(\TTT;\CCC)$ with $\|u_0\|_{H^s}\leq r$,
there is a unique solution $u(t,x)$ of \eqref{NLS} with initial condition $u(0,x)=u_0(x)$
belonging to $C^{1}(I;H^{s-2}(\TTT;\CCC))\cap C^{0}(I;H^{s}(\TTT;\CCC))$
with $I=(-T_r,T_r)$, $T_r>0$.

By Theorem \ref{paralineariza} the function $U=(u,\bar{u})$ 
solves the problem \eqref{sistemainiziale} with initial condition $U_0=(u_0,\bar{u}_0)$, furthermore
by Lemma \ref{aslongas}
such a solution belongs to the ball $B_{s}^{K}(I,r)$.

We now prove  that  $T_{r}\geq c r^{-N}$ for some $c>0$ depending on $s$.
By applying to the system \eqref{sistemainiziale} Theorems \ref{regolarizza} and \ref{BNFiniziale} we have that
$U(t,x)$ solves \eqref{sistemainiziale} if and only if the function $W(t,x)$ given in Theorem \ref{BNFiniziale} solves
\eqref{problemafinale2}. By Lemma \ref{FEI2}
we have that 
\begin{equation}\label{equiequi}
\|U\|_{\hcic^{s}}\sim \|W\|_{\hcic^{s}},
\end{equation}
and that \eqref{stimanormasob} holds. 

We claim that there are multilinear forms $F_p\in \widetilde{\calL}^{s,-\rho''+\tilde{\rho}}_{p,-}$
for $p=1,\ldots, N-1$, for some $\tilde{\rho}<\rho''$ (the constant $\rho''$ is given in Lemma \ref{FEI2}), such that, by setting
\begin{equation}\label{modiene}
\calG(U,W):=\int_{\TTT}|\langle D\rangle^{s}W(t,x)|^{2}dx+\sum_{p=1}^{N-1}F_{p}(U,\ldots,U),
\end{equation}
the following conditions hold:
\begin{equation}\label{equivalenza}
\|U\|_{\hcic^{s}}^2\sim \|W\|^{2}_{\hcic^{s}}\sim \calG(U,W),
\end{equation}
\begin{equation}\label{equivalenza2}
\frac{d}{dt}\calG(U,W)\leq K_1\|U(t,\cdot)\|^{N+2}_{\hcic^{s}}, \quad t\in [-T_r,T_{r}],
\end{equation}
for some $K_1>0$ depending on $s,N$.
To prove this fact we reason as follows.
Note that system \eqref{sistemainiziale} can be written, by Remark \ref{inclusionifacili}, as 
\begin{equation}\label{sistemainizialeMAPPA}
\del_{t}U=\ii E\Lambda  U+M(U;t)[U],
\end{equation}
for some $M\in \Sigma\MM^{m}_{K,0,1}[r,N]\otimes\MM_{2}(\CCC)$, $m\geq0$.
 We show that it is possible to find recursively
multilinear 
forms
\begin{equation}\label{induindu}
\begin{aligned}
&\tilde{L}_{p}\in \widetilde{\calL}^{s,-\rho''+(N_0+m)(p-1)+N_0}_{p,+}, \quad 1\leq p\leq N-1,\\
&{L}^{(q)}_{p}\in \widetilde{\calL}^{s,-\rho''+(N_0+m)q}_{p,-}, \quad q+1\leq p\leq N-1,
\end{aligned}
\end{equation}
such that, for $q=1,\ldots,N-1$,
\begin{equation}\label{induindu2}
\begin{aligned}
\frac{d}{dt}&\left[\int_{\TTT}|\langle D\rangle^{s}W(t,x)|^{2}dx+\sum_{p=1}^{q}\tilde{L}_{p}(U(t,\cdot),\ldots,U(t,\cdot))
\right]\\
&\qquad \qquad=\sum_{p=q+1}^{N-1}L^{(q)}_{p}(U(t,\cdot),\ldots,U(t,\cdot))+O(\|U(t,\cdot)\|^{N+2}_{\hcic^{s}}).
\end{aligned}
\end{equation}
Here $m>2$ is the loss coming form $M(U;t)$ in \eqref{sistemainizialeMAPPA},
the constant
 $N_0$ is in \eqref{stimadalbasso} of Proposition \ref{stimemisura}.
This loss is compensated by the fact that $\rho>0$ in Theorem $\ref{regolarizza}$ is arbitrary large, and hence also $\rho''$ can be taken large enough.
We argue by induction on $q$. For $q=0$ the \eqref{induindu2} follows by \eqref{stimanormasob}.
Assume inductively that \eqref{induindu2} holds for $q-1$.
Let us define $\tilde{L}_{q}\in \widetilde{\calL}^{s,-\rho''+(N_0+m)(q-1)+N_0}$ as the multilinear form
given by  item $(iv)$ of Lemma \ref{fattiSulleForme} applied to $L=L_{q}^{(q-1)}$.
We get
\begin{equation}
\begin{aligned}
\frac{d}{dt}\tilde{L}_{q}&(U(t,\cdot),\ldots,U(t,\cdot))=\ii \sum_{j=0}^{q+1}\tilde{L}_{q}(\underbrace{U,\ldots,U}_{j-times},E\Lambda U,U,\ldots, U)\\
&+\ii \sum_{j=0}^{q+1}\tilde{L}_{j}(\underbrace{U,\ldots,U}_{j-times},M(U;t)[U],U,\ldots, U).
\end{aligned}
\end{equation}
Using items $(iv)$ and $(v)$ of Lemma \ref{fattiSulleForme} we have that
\begin{equation}
\begin{aligned}
\frac{d}{dt}\tilde{L}_{q}(U(t,\cdot),\ldots,&U(t,\cdot))=- {L}_{q}^{(q-1)}({U,\ldots,U})+ \sum_{j=0}^{N-q-3}{L}'_{j}(U,\ldots,U) +O(\|U\|^{N+2}_{\hcic^{s}}),
\end{aligned}
\end{equation}
for some $L'_{j}\in \widetilde{\calL}^{s,-\rho''+(N_0+m)(q-1)+m+N_0}_{p+2+j,-}$.
Thus we get \eqref{induindu2}
at rank $q$.
We conclude by setting $F_{p}$ in \eqref{modiene} equals to $\tilde{L}_{p}$.
Since $r$ is small enough, then, thanks to 
 item $(ii)$ of Lemma \ref{fattiSulleForme}, equation \eqref{equiequi} and Lemma \ref{aslongas},  we get
\[
\calG(U,W)\leq C_s( \|U(t,\cdot)\|_{\hcic^{s}}^{2}+\|U\|^{3}_{\hcic^{s}}),
\]
as long as $\|U(t,\cdot)\|_{\hcic^{s}}\leq Cr$, therefore \eqref{equivalenza} holds. 
The \eqref{equivalenza2}
follows by \eqref{induindu2} for $q=N-1$. 


The thesis follows by using the following bootstrap argument.
The integral form of \eqref{equivalenza2} is 
\begin{equation}\label{profin3}
\calG(U(t,\cdot),W(t,\cdot))\leq \calG(U(0,\cdot),W(0,\cdot))+K_{1}\int_{0}^{t}\|U(\tau,\cdot)\|^{N+2}_{\hcic^{s}}d\tau,
\end{equation}
and by \eqref{equivalenza} we have that
$\calG(U(0,\cdot),W(0,\cdot))\leq c_{0}r^{2}$,
for some $c_0$ depending on $s$.
Fix $K_{2}=K_{2}(s,N)>1$ and let $\bar{T}$ the supremum of those $T$
such that 
\begin{equation}\label{maiale5}
\sup_{t\in [-T,T]}\calG(U(t,\cdot),W(t,\cdot))\leq K_{2}r^{2}.
\end{equation}
Assume, by contradiction, that $\bar{T}<\tilde{c} r^{-N}$. Then,
if $K_1$ is the constant appearing in the r.h.s. of \eqref{profin3}, 
 we have
\begin{equation}\label{profin4} 
\begin{aligned}
\calG(U(t,\cdot),W(t,\cdot))&\leq c_0 r^{2}+K_{1}\int_{0}^{t}K_{2}^{N+2}r^{N+2}d\tau\leq c_0r^{2}+K_1K_{2}^{N+2}r^{N+2}\bar{T}\\
&\leq c_0r^{2}+K_1 K_{2}^{N+2}r^{N}\tilde{c} r^{-N}r^{2}\leq r^{2}\big(c_0+K_{1}K_{2}^{N+2}\tilde{c}\big)\leq K_{2}r^{2}\frac{3}{4},
\end{aligned}
\end{equation}
for $\tilde{c}>0$ small enough and $K_{2}\gg c_0$ large enough hence the contradiction. 
By \eqref{equivalenza} the reasoning above implies also that 
\[
\sup_{t\in[-T,T]}\|U(t,\cdot)\|_{\hcic^{s}}\leq Cr, \quad T\geq \tilde{c}r^{-N},
\]
for some fixed $C>0$ depending on $s,N$. This is \eqref{iltempo} for $k=0$.
Moreover by Lemma  \ref{aslongas} we also obtain
that $\del_{t}^kU(t,\cdot)$ satisfies 
$
\sup_{t\in[-T,T]}\|\del_{t}^kU(t,\cdot)\|_{\hcic^{s-2k}}\leq Cr, \quad T\geq \tilde{c}r^{-N},
$
if $r$ is small, $s\gg K $  and where $C$ is a large enough constant depending on $K$.
\end{proof}

\def\cprime{$'$}


\begin{thebibliography}{10}

\bibitem{alaz-baldi-periodic}
T.~Alazard and P.~Baldi.
\newblock Gravity capillary standing water waves.
\newblock {\em Archive for Rational Mechanics and Analysis}, 217(3):741--830,
  2015.
\newblock DOI: 10.1007/s00205-015-0842-5.


\bibitem{BBM}
P.~Baldi, M.~Berti, and R.~Montalto.
\newblock {KAM} for quasi-linear and fully nonlinear forced perturbations of
  {A}iry equation.
\newblock {\em Math. Ann.}, 359, 2014.
\newblock 10.1007/s00208-013-1001-7.

\bibitem{BBM1}
P.~Baldi, M.~Berti, and R.~Montalto.
\newblock {KAM} for autonomous quasilinear perturbations of {K}d{V}.
\newblock {\em Ann. I. H. Poincar\'e (C) Anal. Non Lin\'eaire}, 33, 2016.
\newblock 10.1016/j.anihpc.2015.07.003.

\bibitem{Bambu}
D.~Bambusi.
\newblock Birkhoff normal form for some nonlinear {P}{D}{E}s.
\newblock {\em Comm. Math. Phys,}, 234, 2003.
\newblock 10.1007/s00220-002-0774-4.

\bibitem{BDGS}
D.~Bambusi, J.~M. Delort, B.~Gr\'ebert, and J.~Szeftel.
\newblock {A}lmost global existence for {H}amiltonian semi-linear
  {K}lein-{G}ordon equations with small {C}auchy data on {Z}oll manifolds.
\newblock {\em Comm. Pure Appl. Math.}, 60:1665--1690, 2007.
\newblock 10.1002/cpa.20181.

\bibitem{BG}
D.~Bambusi and B.~Gr\'ebert.
\newblock {B}irkhoff normal form for partial differential equations with tame
  modulus.
\newblock {\em Duke Math. J.}, 135 n. 3:507--567, 2006.
\newblock 10.1215/S0012-7094-06-13534-2.

\bibitem{BFG}
J. Bernier, E. Faou and B. Gr\'ebert.
\newblock {R}ational Normal Forms and stability of small solutions to Nonlinear Schr\"odinger Equations
\newblock {\em preprint} Arxiv: 1812.11414.


\bibitem{maxdelort}
M.~Berti and J.M. Delort.
\newblock Almost global  solutions for capillarity-gravity water
  waves equations on the circle.
\newblock UMI Lecture Notes 2018, ISBN 978-3-319-99486-4.

\bibitem{BM1}
M.~Berti and R.~Montalto.
\newblock Quasi-periodic standing wave solutions of gravity-capillary water
  waves.
\newblock {\em Memoirs of the American Math. Society, MEMO 891}, 2016.
\newblock to appear - arXiv:1602.02411.

\bibitem{CraigSulem}
W.~Craig and C.~Sulem.
\newblock {N}ormal form transformations for capillary-gravity water waves.
\newblock {\em Field Inst. Commun.}, 75:73--110, 2015.

\bibitem{Craigworfo}
W.~Craig and P.A. Worfolk.
\newblock An integrable normal form for water waves in infinite depth.
\newblock {\em Phys. D}, 84:513--531, 1995.
\newblock 10.1016/0167-2789(95)00067-E.

\bibitem{Delort-2009}
J.~M. Delort.
\newblock {A} quasi-linear {B}irkhoff normal forms method. {A}pplication to the
  quasi-linear {K}lein-{G}ordon equation on $\mathds{S}^1$.
\newblock {\em Ast\'erisque}, 341, 2012.

\bibitem{DelortSzeft1}
J.~M. Delort and J.~Szeftel.
\newblock {L}ong-time existence for small data nonlinear {K}lein--{G}ordon
  equations on tori and spheres.
\newblock {\em Internat. Math. Res. Notices}, 37, 2004.
\newblock 10.1155/S1073792804133321.

\bibitem{DelortSzeft2}
J.~M. Delort and J.~Szeftel.
\newblock {L}ong-time existence for semi-linear {K}lein--{G}ordon equations
  with small cauchy data on {Z}oll manifolds.
\newblock {\em Amer. J. Math.}, 128, 2006.
\newblock 10.1353/ajm.2006.0038.

\bibitem{Delort-sphere}
J.M. Delort.
\newblock {\em {Q}uasi-{L}inear {P}erturbations of {H}amiltonian
  {K}lein-{G}ordon {E}quations on {S}pheres}.
\newblock American Mathematical Society, 2015.
\newblock 10.1090/memo/1103.

\bibitem{grefaou}
E.~Faou and B.~Gr\'ebert.
\newblock {Q}uasi invariant modified {S}obolev norms for semi linear reversible
  {P}{D}{E}s.
\newblock {\em Nonlinearity}, 23:429--443, 2010.
\newblock 10.1088/0951-7715/23/2/011.

\bibitem{FIloc}
R.~Feola and F.~Iandoli.
\newblock {L}ocal well-posedness for quasi-linear {N}{L}{S} with large {C}auchy
  data on the circle.
\newblock {\em Annales de l'Institut Henri Poincare (C) Non Linear Analysis},
36(1): 119-164, 
  2019.
\newblock  10.1016/j.anihpc.2018.04.003.

\bibitem{FP}
R.~Feola and M.~Procesi.
\newblock Quasi-periodic solutions for fully nonlinear forced reversible
  {S}chr{\"o}dinger equations.
\newblock {\em Journal of Differential Equations}, 2014.
\newblock 10.1016/j.jde.2015.04.025.

\bibitem{FP2}
R.~Feola and M.~Procesi.
\newblock {K}{A}{M} for quasi-linear autonomous {N}{L}{S}.
\newblock preprint- arXiv:1705.07287, 2017.

\bibitem{LLNR}
J. Laurie, V.S. L'Vov, S. V. Nazarenko, O. Rudenko. 
\newblock {I}nteraction of Kelvin waves and nonlocality of energy transfer in superfluids.
Phys. Rev. B  81,  2010.

\bibitem{YuZhan}
X.~Yuan and J.~Zhang.
\newblock {L}ong {T}ime stability of {H}amiltonian {P}artial {D}ifferential
  {E}quations.
\newblock {\em SIAM J. of Math. Anal.}, 46(5):3176--3222, 2014.
\newblock 10.1137/120900976.

\bibitem{zach}
V. E. Zakharov, editor. 
\newblock {What is integrability?} Springer, 1991.

\end{thebibliography}
\end{document}